 \documentclass[draft]{article}

\usepackage{amsmath,amsfonts,amsthm,amssymb,amscd,cancel,color}
\usepackage{enumitem}
\usepackage{verbatim}
\usepackage[dvipdfmx]{graphicx}
\usepackage{ulem,color}

\renewcommand{\Re}{\mathop{\rm Re}\nolimits}

\def\S{\mathhexbox278}

\theoremstyle{plain}
\newtheorem{theorem}{Theorem}[section]
\newtheorem{lemma}[theorem]{Lemma}
\newtheorem{proposition}[theorem]{Proposition}
\newtheorem{corollary}[theorem]{Corollary}
\theoremstyle{definition}
\newtheorem{definition}[theorem]{Definition}
\theoremstyle{remark}
\newtheorem{remark}[theorem]{Remark}
\newtheorem{example}[theorem]{Example}

\newtheorem{assumption}[theorem]{Assumption}

\newcommand{\R}{{\mathbb R}}

\newcommand{\Z}{{\mathbb Z}}

\newcommand{\N}{{\mathbb N}}

\def\im{{\rm i}}

\newcommand{\C}{\mathbb{C}}

\newcommand{\sech}{{\mathrm{sech}}}
\def\({\left(}
\def\){\right)}
\def\<{\left\langle}
\def\>{\right\rangle}

\numberwithin{equation}{section}

\begin{document}

\title{  On selection  of standing wave  at small energy  in the 1D Cubic Schr\"odinger Equation with a trapping potential }

\author{Scipio Cuccagna, Masaya Maeda}
\maketitle

\begin{abstract} Combining virial inequalities by Kowalczyk, Martel and Munoz  and Kowalczyk, Martel, Munoz  and Van Den Bosch with our theory on how to derive nonlinear induced dissipation on discrete modes, and in particular the notion of Refined Profile, we show how to extend the theory by  Kowalczyk, Martel, Munoz  and Van Den Bosch to the case when there is a large number of discrete modes in    the cubic NLS with a trapping potential which is associated to  a repulsive potential by a series of  Darboux  transformations. Even though, by its
non  translation invariance,   our model avoids some of the difficulties related to the
effect that translation has on virial inequalities  of the  kink stability problem for wave equations, it  still is a classical model and it retains some of the main difficulties.
\end{abstract}

%\tableofcontents
%\newpage

\section{Introduction}
In this paper, we consider the cubic nonlinear Schr\"odinger equation (NLS) with potential,
\begin{equation}\label{nls}
\im \partial_t u = -\partial_x^2u +V u +  |u|^2 u,\quad (t,x)\in \R^{1+1},
\end{equation}
where the potential $V$ satisfies
\begin{align}\label{ass:V}
&V\in \mathcal S(\R, \R)\ \text{(Schwartz functions)},\  |V(x)|+|V'(x)|\leq C e^{-a_0|x|}\  \text{for some }C, a_0>0,\end{align}
and we assume that for $\sigma_{\mathrm{d}}(H)$, which is the set of discrete spectrum of $H:=-\partial_x^2+V$, we have
\begin{align}
\label{ass:discsp}
\sigma_{\mathrm{d}}(H)=\{\omega_j\ |\ j=1,\cdots,N\},\ \omega_1<\cdots<\omega_N<0,\ N\geq 2.
\end{align}

 \begin{remark}
  It is well known that $\sigma_{\mathrm{d}}(H)$ is finite. Our assumption is that $H$ has more than two eigenvalues. The case  $N=0 $
  has been treated by Naumkin  \cite{naumkin2016}, by    Germain \textit{{et al.}} \cite{GPR18}, by Delort \cite{Delort} and by Chen--Pusateri \cite{ChenPus}. See also Masaki  \textit{{et al.}} \cite{MMS2} along with \cite{CM19SIMA} for the case of a repulsive $\delta$ potential.  The case  $N=1 $   in the case of an attractive $\delta$ potential with rather general nonlinear terms, which include as a special case also the cubic  nonlinearity, is  treated in \cite{CM19SIMA}. The  case of a generic potential with  $N=1 $,  hence  a less stringent hypothesis than Assumption \ref{ass:VN} below,    is discussed in Chen \cite{Chen}.  In this paper we focus only on the case $N\ge 2$, which is more delicate than the cases $N=0,1$.

  \end{remark}

\begin{remark}
In the sequel we will often use the notation $\dot f=\partial _t f$ and $  f'=\partial _x f$.
We also use the notation $a \lesssim b$, which means that there exists $C>0$ s.t.\ $a\leq Cb$ with $C$ not depending on important quantities.
We also write $a\sim b$ if $a\lesssim b$ and $b\lesssim a$, and $a\lesssim_\alpha b$ if $a\leq C_\alpha b$ with the implicit constant $C_\alpha$ depending on $\alpha$.
Finally, $a\sim_\alpha b$ will mean $a\lesssim_\alpha b$ and $b\lesssim_\alpha a$.
\end{remark}

The aim of this paper is to study the long time behavior of small solutions of \eqref{nls}.
Here, we recall that from the energy
\begin{align}\label{def:energy}
E(u):=\frac{1}{2}\int_{\R}\(|\partial_xu|^2+V|u|^2+\frac{1}{2}|u|^4\)\,dx
\end{align}
and mass
\begin{align}\label{def:mass}
Q(u):=\frac{1}{2}\int_{\R}|u|^2\,dx
\end{align}
conservation, if we have $u_0\in H^1$, then
\begin{align}\label{eq:energybound}
\|u(t)\|_{H^1}\lesssim \|u_0\|_{H^1}.
\end{align}
Thus, global well-posedness of small solutions is trivial.
In this paper, we seek a  more precise  understanding of the asymptotic behavior of $u(t)$.
By elementary bifurcation argument, under appropriate hypotheses   there exist $N$ families of standing wave solutions (i.e.\ solutions with the form $u(t,x)=e^{\im \omega t}\phi(x)$) bifurcating from the eigenvalues of $H$.
That is, there exist $\phi_j[z](x)=z\phi_j(x)+O(|z|^3)$ and $\omega_j(|z|^2)=\omega_j+O(|z|^2)$ ($j=1,\cdots,N$) for small $z\in \C$ s.t.\  $e^{-\im \omega_j(|z|^2)t}\phi_j[z](x)$  are standing wave solutions of \eqref{nls}.
In the linear case ($\im \partial_t u = Hu$), by the superposition principle, there also exist  quasi-periodic solutions which are given by the sum of standing waves.
However, it was shown that in the 3D case with smooth nonlinearity (corresponding to the 1D case with the nonlinearity $|u|^2u$ replaced by $g(|u|^2)u$ with $g\in C^\infty$ and $g(0)=g'(0)=0$, e.g.\ $|u|^4u$), the solutions locally converge  to the orbit of  one single  standing wave (selection of standing waves) and therefore there exist  no quasi-periodic solution, see \cite{CM15APDE} and the references therein.
Thus, even though the short time dynamics of small solutions of linear Schr\"odinger equation and NLS are similar, in 3D the long time dynamics are completely different.
The aim of this paper is to prove a similar  result also for 1D cubic NLS \eqref{nls}, under  an additional hypothesis on the potential, see \S \ref{sec:rephyp}. This is   more difficult than the 3D  case   because of  the fact that  $|u|^2 u$   is a long range nonlinearity  in 1D  and by the weakness of linear dispersion in 1D.  The main idea in this paper consists in a combined use of the dispersion theory of  Kowalczyk, Martel and Munoz  \cite{KMM3}, from which we draw extensively,  with the notion of Refined Profiles we introduced in \cite{CM20} and which we discuss now, before stating the main result.

\subsection{Set up}
For $a\in \R$  and $\<x\> := \sqrt{1+x^2}$, we set
\begin{align}&\label{eq:sigmas--}
\|u\|_{X_a}:=\|e^{a\<x\>}u\|_X,\ \text{for } \ X=H^s,\ L^p \text{  and  }a :=2^{-1}\sqrt{|\omega_N|},
\\&\label{eq:sigmas}
\Sigma^s:=
H^s_{a } ,
\end{align}
where we make the convention that,  when we write $L^p$,  $H^s$ or other analogous spaces like $L^{2,s}$ below, they are  $L^p(\R )$,  $H^s(\R )$, etc.

\noindent For any $s\in \R$, we will use also other weighted spaces, defined by the norm
\begin{align}\label{eq:withL2}
   \| u \| _{L^{2,s} }:= \|  \<x\> ^s u \| _{L^{2 }(\R )}.
\end{align}

We will consider repeatedly  several Bochner spaces of the form $L^p(I, X)$, with $p\in [1,\infty ]$, $I \subseteq \R $ an interval and $X$ a Banach space, see \cite[Chapter 1]{ch}, with norms
\begin{align}&\label{eq:spacetime}
\|u\|_{L^p(I, X)}:=\|   \|   u\|_X  \| _{L^p(I)} .
\end{align}
In particular,  we will use spaces like $X= \widetilde{\Sigma}$, introduced in   \eqref{eq:norm_rho},  and  $X= X_a$  like in \eqref{eq:sigmas--}.

We recall that all eigenvalues of $H$ are simple.
We pick  $\phi_j$ such that $\| \phi _j \| _{L^2} =1$, $\R$-valued eigenfunctions of $H$ associated with $\omega_j$.
Since $\phi_j^{(n)}\sim_n e^{-\sqrt{|\omega_j|}|x|}$, we have $\phi_j \in \cap_{s>0}\Sigma^s$ for all $j=1,\cdots,N$.

\subsubsection{Refined profile and Fermi Golden Rule assumption}
One of the keys   is the notion of refined profile introduced in \cite{CM20}.
We set
\begin{align}\label{def:bomega}
\boldsymbol{\omega}=(\omega_1,\cdots,\omega_N).
\end{align}
For the discrete spectrum $\sigma_{\mathrm{d}}(-\partial_x^2+V)$, we assume the following.
\begin{assumption}\label{ass:disc_ratio_indep}
For $\mathbf{m}:=(m_1,\cdots,m_N)\in \Z^N\setminus\{0\}$, $\mathbf{m}\cdot\boldsymbol{\omega} =\sum_{j=1}^{N}m_j\omega_j\neq 0$.
\end{assumption}

\begin{remark} In reality we need Assumption \ref{ass:disc_ratio_indep} for a restricted and finite set of indexes.
\end{remark}

In the following, for $\mathbf{x}=(x_1,\cdots,x_N)\in X^N$, for $X=\Z,\R,\C$ we write
 $\|\mathbf{x}\|:=\|\mathbf{x}\|_1=\sum_{j=1}^N|x_j|$.
We consider the sets of multi--indexes
\begin{align}\label{eq:RandNR}
\mathbf{NR}:=\{\mathbf{m}\in \Z^N\ |\ \sum_{j=1}^N m_j=1,\ \mathbf{m}\cdot \boldsymbol{\omega}<0\},\
\mathbf{R}:=\{\mathbf{m}\in \Z^N\ |\ \sum_{j=1}^N m_j=1,\ \mathbf{m}\cdot \boldsymbol{\omega} >0\} .
\end{align}
By Assumption \ref{ass:disc_ratio_indep}, we have $\{\mathbf{m}\in \Z^N\ |\ \sum_{j=1}^Nm_j=1\}=\mathbf{NR}\cup \mathbf{R}$.
Furthermore, we set
\begin{align}\label{def:Rmin}
\mathbf{R}_{\mathrm{min}}:=\{\mathbf{m}\in \mathbf{R}\ |\ \not\exists \mathbf{n}\in \mathbf{R}\ \mathrm{s.t.}\ \mathbf{n}\prec\mathbf{m}\},
\end{align}
where
\begin{align*}
\mathbf{n}\prec\mathbf{m}\ &\Leftrightarrow\ \mathbf{n}\preceq \mathbf{m}\ \text{and}\ \|\mathbf{n}\|<\|\mathbf{m}\|,\\
\mathbf{n}\preceq\mathbf{m}\ &\Leftrightarrow\ \forall j=1,\cdots,N,\ |n_j|\leq |m_j|.
\end{align*}
Related to $\mathbf{R}_{\mathrm{min}}$ are the sets  $\mathbf{I}$ and  $\mathbf{NR}_1$, defined  by
\begin{align}\label{def:NR1I}
\mathbf{I}:=\{\mathbf{m}\in \mathbf{R}\cup \mathbf{NR}\ |\ \exists\mathbf{n}\in \mathbf{R}_{\mathrm{min}}\ \text{s.t. }\mathbf{n}\prec\mathbf{m}\},\ \mathbf{NR}_1:=\mathbf{NR}\setminus\mathbf{I}.
\end{align}
\begin{remark}
The set $\mathbf{I}$ will be the collection of    negligible multi--indexes, in the sense that for $\mathbf{n}\in \mathbf{I}$,
\begin{align*}
|\mathbf{z}^{\mathbf{n}}|\leq \|\mathbf{z}\| \sum_{\mathbf{m}\in \mathbf{R}_{\mathrm{min}}}|\mathbf{z}^{\mathbf{m}}|\ \mathrm{for}\ \|\mathbf{z}\|\leq 1,
\end{align*}
where $\mathbf{z}^{\mathbf{n}}$ is defined right  below. Here $\mathbf{R}$ stands for resonant, while $\mathbf{NR}$ stands for non resonant. The most significant elements of  $\mathbf{NR}$  are those of $\mathbf{NR}_1$. The corresponding monomials $\mathbf{z} ^{\mathbf{m}}$ for $\mathbf{m}\in \mathbf{NR}_1$ are eliminated and do not appear in the key equation  \eqref{eq:rp} by a, rather elementary, normal forms argument, while the  $\mathbf{z} ^{\mathbf{m}}$ for $\mathbf{m}\in \mathbf{NR} \cap \mathbf{I}$
are small remainder terms, absorbed in $\mathcal{R}_{\mathrm{rp}}[\mathbf{z}]$ and  easy to bound in the course of the proof.
\end{remark}
For $\mathbf{z}\in \C^N$, we set
\begin{align}\label{def:zn}
\mathbf{z}^{\mathbf{m}}:=\prod_{j=1}^{N}z_j^{(m_j)},\ z_j^{(m_j)}:=\begin{cases}
z_j^{m_j} & \text{if}\ m_j\geq 0,\\ \overline{z_j}^{-m_j} & \text{if}\ m_j<0.
\end{cases}
\end{align}

We inductively define $G_{\mathbf{m}}$ for $\mathbf{m}\in \mathbf{R}_{\mathrm{min}} \cup \mathbf{NR}_1$ by $G_{\mathbf{m}}=0$ if $\|\mathbf{m}\|\leq 1$ and
\begin{align}\label{def:Gm0}
G_{\mathbf{m}}=\sum_{\substack{\mathbf{m}^1,\mathbf{m}^2,\mathbf{m}^3\in \mathbf{NR}_1,\\ \mathbf{m}^1-\mathbf{m}^2+\mathbf{m}^3=\mathbf{m},\  \|\mathbf{m}^1\|+\|\mathbf{m}^2\|+\|\mathbf{m}^3\|=\|\mathbf{m}\|}}\widetilde{\phi_{\mathbf{m}^1}}\widetilde{\phi_{\mathbf{m}^2}}\widetilde{\phi_{\mathbf{m}^3}},
\end{align}
where
\begin{align}\label{def:tilde}
\widetilde{\phi_{\mathbf{m}}}=\begin{cases}
0, & \mathbf{m}=0\\
\phi_j, & \mathbf{m}=\mathbf{e}^j:=(\delta_{j1},\cdots,\delta_{jN}),\\
-(H-\mathbf{m}\cdot \boldsymbol{\omega})^{-1}G_{\mathbf{m}}, & \|\mathbf{m}\|\geq 2
\end{cases}
\end{align}

\begin{example}
We give the first few $G_{\mathbf{m}}$ for the case $N=2$.
When $\mathbf{m} \in \mathbf{R}_{\mathrm{min}} \cup \mathbf{NR}_1$ and $\|\mathbf{m}\|=3$, we have $\mathbf{m}=(2,-1)$ or $(-1,2)$ and
\begin{align*}
G_{(2,-1)}=\phi_1^2 \phi_2,\quad G_{(-1,2)}=\phi_1\phi_2^2.
\end{align*}
When $\mathbf{m} \in \mathbf{R}_{\mathrm{min}} \cup \mathbf{NR}_1$ and $\|\mathbf{m}\|=5$, we have $\mathbf{m}=(3,-2)$ or $(-2,3)$ and
\begin{align*}
G_{(3,-2)}&=-2\phi_1\phi_2(H-(2\omega_1-\omega_2))^{-1}\(\phi_1^2\phi_2\)-\phi_1^2(H-(-\omega_1+2\omega_2))^{-1}(\phi_1\phi_2^2),\\
G_{(-2,3)}&=-2\phi_1\phi_2(H-(-\omega_1+2\omega_2))^{-1}\(\phi_1 \phi_2^2\)-\phi_2^2(H-(2\omega_1-\omega_2))^{-1}(\phi_1^2\phi_2),
\end{align*}
\end{example}

By the inductive definition, we have the following.
\begin{lemma}\label{lem:GmRval}
	 For $\mathbf{m}\in \mathbf{R}_{\mathrm{min}} \cup \mathbf{NR}_1$  the $G_{\mathbf{m}}$ are $\R$-valued.
\end{lemma}

\begin{proof}
	If $\|\mathbf{m}\|\leq 1$, then the statement is obvious because we have chose $\phi_j$' to be $\R$-valued.
	Next, for  $\mathbf{m}\in \mathbf{R}_{\mathrm{min}} \cup \mathbf{NR}_1$, we assume that for all $\mathbf{n}\in \mathbf{NR}_1$ with $\|\mathbf{n}\|<\|\mathbf{m}\|$, $G_{\mathbf{n}}$ is $\R$-valued.
	Then, from \eqref{def:tilde}, $\widetilde{\phi_{\mathbf{n}}}$ is $\R$-valued and by \eqref{def:Gm0}, $G_{\mathbf{m}}$ is also $\R$-valued.
\end{proof}

An important assumption, related to the  Fermi Golden Rule (FGR), is the following.
\begin{assumption}\label{ass:FGR}
We assume that for all $\mathbf{m}\in \mathbf{R}_{\mathrm{min}}$,
\begin{align*}
\sum_{\pm}|\widehat{G_{\mathbf{m}}}\(\pm \sqrt{\boldsymbol{\omega}\cdot \mathbf{m}}  \)  |>0,
\end{align*}
where $\widehat{G_{\mathbf{m}}}$  is the distorted Fourier transform of $G_{\mathbf{m}}$ associated to the operator $H$, see \cite{weder}.
\end{assumption}

For a Banach space $X$ and $x\in X$, $r>0$, we set
\begin{align}\label{def:ball}
B_X(x,r):=\{y\in X\ |\ \|y-x\|_X<r\}.
\end{align}
For $F\in C^1(B_{\C^N}(0,\delta),X)$, $\mathbf{z}\in B_{\C^N}(0,\delta)$ and $\mathbf{w}\in \C^N$, we set $D_{\mathbf{z}}F(\mathbf{z})\mathbf{w}:=\left.\frac{d}{ds}\right|_{s=0}F(\mathbf{z}+s\mathbf{w})$.

The following is proved in \cite{CM20}.
\begin{proposition}[Refined profile]\label{prop:rp}
For any $s>0$, there exists $\delta_s>0$ s.t.\ there exists $\phi[\cdot]\in C^\infty(B_{\C^N}(0,\delta_s),\Sigma^s)$   of the form
\begin{align}\label{def:refpexp}
\phi[\mathbf{z}]= \sum _{\mathbf{m}\in \mathbf{NR}_1} \mathbf{z} ^{\mathbf{m}}\widetilde{\phi_{\mathbf{m}}} \text{  where }\widetilde{\phi_{\mathbf{m}}}\in \Sigma ^{s} \text{  for all $s$,}
\end{align}
 $\boldsymbol{\varpi}\in C^\infty(B_{\C^N}(0,\delta_s),\R^N)$ and $\mathcal{R}_{\mathrm{rp}}[\cdot]\in C^\infty(B_{\C^N}(0,\delta_s),\Sigma^s)$ s.t.
\begin{align}\label{eq:rp}
-\im D_{\mathbf{z}}\phi[\mathbf{z}]\(\im \boldsymbol{\varpi}(\mathbf{z})\mathbf{z}\)=H\phi[\mathbf{z}]+|\phi[\mathbf{z}]|^2\phi[\mathbf{z}]-\sum_{\mathbf{m}\in \mathbf{R}_{\mathrm{min}}}\mathbf{z}^{\mathbf{m}} G_{\mathbf{m}}-\mathcal{R}_{\mathrm{rp}}[\mathbf{z}],
\end{align}
and $\phi[0]=0$, $D_{\mathbf{z}}\phi[0]\mathbf{e}^j=\phi_j$, $D_{\mathbf{z}}^2\phi[0]=0$, $\boldsymbol{\varpi}(0)=\boldsymbol{\omega}$, $\phi[e^{\im \theta}\mathbf{z}]=e^{\im \theta}\phi[\mathbf{z}]$, $\boldsymbol{\varpi}(\mathbf{z})=\boldsymbol{\varpi}(|z_1|^2,\cdots,|z_N|^2)$,
\begin{align}\label{est:Rrp}
\|\mathcal{R}_{\mathrm{rp}}[\mathbf{z}]\|_{\Sigma^s}\lesssim_s \|\mathbf{z}\|^2\sum_{\mathbf{m}\in \mathbf{R}_{\mathrm{min}}}|\mathbf{z}^{\mathbf{m}}|,
\end{align}
and
$
\boldsymbol{\varpi}(\mathbf{z})\mathbf{z}:=(\varpi_1(\mathbf{z})z_1,\cdots,\varpi_N(\mathbf{z})z_N).
$
\end{proposition}
\qed

\begin{remark}\label{rem:refprofiles} The  Refined Profile generalizes the notion of   standing waves, which are    generated from the Refined Profile setting
\begin{align*}
\phi_j(z_j):=\phi(z _j\mathbf{e}_j)  \text{ for $z_j\in \mathcal{B}_{\C}  (0, \delta _s)$    and $\mathbf{e}_j= ( \delta _{1 j} ,..., \delta _{N j})$, with $\delta _{jk}$ the Kronecker delta.}
\end{align*}
It represents an effective way to  represent the discrete component because it
  provides a modulation  $u=\phi[\mathbf{z}]+\eta$ of the solution $u$,   where in the equation of  the continuous mode $\eta$, see \eqref{eq:modnls} below,
there are no monomials  $\mathbf{z} ^{\mathbf{m}}$  with $\mathbf{m}\in \mathbf{NR}_1$.  It plays an analogous role to that of  Fraiman's like ansatz   in  papers like Merle and Raphael \cite{MR4} where it allows to   bypass normal forms arguments in the course of the analysis of  the Fermi Golden Rules.

\end{remark}

\subsubsection{The Repulsivity Hypothesis}\label{sec:rephyp}

In order to use the dispersion theory of Kowalczyk  \textit{{et al.}}  \cite{KMM2}--\cite{KMM4} we need the following,    inspired by a more general notion in \cite{KMM4}.
\begin{definition}\label{def:repulPot} Let $\mathcal{V}$ be a potential like in \eqref{ass:V}. We say that  $\mathcal{V}$ is repulsive if
 $\mathcal{V}$ is not identically zero and   $x\mathcal{V}'(x)\leq 0$ for all $x\in \R$.
\end{definition}
Obviously the above notion,   framed in terms of the origin, could be reframed with respect to any other $x_0\in \R$, but we can always normalize by translation so that $x_0=0$.

Crucial in the theory in  Kowalczyk  \textit{{et al.}} \cite{KMM3,KMM4} is a mechanism of addition or subtraction of eigenvalues which can be traced to Darboux. Kowalczyk \textit{{et al.}}  \cite{KMM3,KMM4} treat some very special situations and  refer  to the theory in Sect.\ 3.2--3.3 in
Chang-Gustafson-Nakanishi-Tsai  \cite{Chang}.  In reality, a systematic  and quite general treatment  of this topic is in Sect.\ 3 Deift-Truwobitz \cite{DT}, to which we refer for the following, where we impose much stricter hypotheses than in  \cite{DT}.

\begin{proposition}\label{prop:Darboux}
Let $W\in \mathcal{S}(\R,\R)$ s.t $\sigma_{\mathrm{d}}(-\partial_x^2+W)\neq \emptyset$.
Let $\psi$ be the ground state (positive eigenfunction) of $-\partial_x^2+W$ and set $A_W=\frac{1}{\psi}\partial_x\(\psi \cdot\)$.
Then, there exists $W_1\in \mathcal{S}(\R,\R)$ s.t.
\begin{align*}
A_WA_W^*=-\partial_x^2+W-\omega,\ A_W^*A_W=-\partial_x^2 +W_1 -\omega,
\end{align*}
where $\omega=\min \sigma_{\mathrm{d}}(-\partial_x^2+W)$.
Further, we have $\sigma_{\mathrm{d}}(-\partial_x^2+W_1)=\sigma_{\mathrm{d}}(-\partial_x^2+W)\setminus\{\omega\}$.
\end{proposition}

Using Proposition \ref{prop:Darboux}, we inductively define $V_j \in \mathcal{S}(\R,\R)$ ($j=1,\cdots,N+1$) by
\begin{enumerate}
\item $V_1:=V$, $H_1:=-\partial_x^2+V_1$, $\psi_1=\phi_1$ and $A_1=A_{V_1}$.
\item Given $V_k$, we define
\begin{align}\label{def:Ak}
A_k:=A_{V_k}\text{ and  }H_{k+1}:=-\partial_x^2+V_{k+1}:=A_k^*A_k+\omega_k.
\end{align}
\end{enumerate}
From Proposition \ref{prop:Darboux}, we have
\begin{align*}
\sigma_{\mathrm{d}}(H_k)=\{\omega_j\ |\ j=k,\cdots,N\},\ k=1,\cdots,N,\ \text{and}\ \sigma_{\mathrm{d}}(H_{N+1})=\emptyset.
\end{align*}
If $\psi_k$ is the ground state of $H_k$  and  $A_k=\frac{1}{\psi_k}\partial_x\(\psi_k\cdot\)$ then,
from
\begin{align}\label{eq:DarConj1}
A_j^* H_j=A_j^*(A_jA_j^*+\omega_j)=(A_j^*A_j+\omega_j)A_j^*=H_{j+1}A_j^* ,
\end{align}
we have the conjugation relation
\begin{align}\label{eq:DarConj2}
\mathcal{A}^* H_1=H_{N+1}\mathcal{A}^*,
\end{align}
where
\begin{align}\label{def:calA}
\mathcal{A}=A_1\cdots A_N \text{  and }\mathcal{A}^*=A_N^*\cdots A_1^*.
\end{align}

The crucial assumption to implement the theory in  Kowalczyk  \textit{{et al.}}  \cite{KMM3,KMM4} and to overcome the   strength of the cubic nonlinearity in 1D is the following.
\begin{assumption}\label{ass:VN}
$V_{N+1}$ is  a repulsive potential in the sense of Definition  \ref{def:repulPot}.
\end{assumption}

\begin{remark}
By reverting the transformation given in Proposition \ref{prop:Darboux}, starting from any repulsive potential, for any $N$, one can construct a potential $V$ satisfying  Assumptions \ref{ass:disc_ratio_indep} and \ref{ass:VN}.
\end{remark}

\subsection{Main theorem}\label{sec:mainth}
We are now in the position to state the main theorem of this paper.
\begin{theorem}\label{thm:main} Assume \eqref{ass:V}--\eqref{ass:discsp} and Assumptions \ref{ass:disc_ratio_indep}, \ref{ass:FGR} and \ref{ass:VN}. Then for any $\epsilon >0$ and $a>0$ there exists $\delta _0>0$ s.t.\ for all $u_0\in H^1$ with $\|u_0\|_{H^1}:=\delta<\delta _0$ there are   $\mathbf{z}\in C^1(\R,\C^N)$, and  $\eta\in C(\R,H^1)$  s.t.
\begin{align}\label{eq:main1}
\|\mathbf{z}\|_{W^{1,\infty}\(\R \) }+\|\eta\|_{L^\infty\( \R ,  H^1 \)}\lesssim \delta,
\end{align}
with,  for all $t\geq 0$,
\begin{align}\label{eq:main11}
u(t)=\phi[\mathbf{z}(t)]+\eta(t).
\end{align}
Moreover, we have  for $I=[0,\infty)$
\begin{align}\label{eq:main2}
\|\dot {\mathbf{z}}+\im \boldsymbol{\varpi}(\mathbf{z})\mathbf{z}\|_{L^2(I)}+\sum_{\mathbf{m}\in \mathbf{R}_{\mathrm{min}}}\|\mathbf{z}^\mathbf{m}\|_{L^2(I)}+\|\eta\|_{L^2(I, H^1_{-a})}\le  \epsilon  .
\end{align}
Finally, there  exists $j(u_0)\in \{1,\cdots,N\}$ and $\rho_+(u_0)\geq 0$ such that
\begin{align}\label{eq:main3}
\lim_{t\to \infty}|z_k(t)|=\begin{cases}
0 & k\neq j(u_0),\\
\rho_+(u_0) & k=j(u_0).
\end{cases}
\end{align}
\end{theorem}

\begin{remark} The fact that in the paper the cubic term is defocusing plays no role in our proof. Theorem \ref{sec:mainth}  holds also with a focusing cubic term.
Obviously, inverting time we conclude that \eqref{eq:main2} holds for $I=\R$.
Notice also that  it is not possible to prove decay rates because all the estimates need necessarily to be invariant by translation in time. Hence, all the literature which proves a rate of decay in time needs to take initial data in spaces smaller than $H^1(\R )$.
\end{remark}

The novel difficulties in Theorem \ref{thm:main}  come  from  the cubic nonlinearity, which   in dimension 1 is \textit{long range}.  For quintic or higher power, but always smooth,   nonlinearities, which are short range,  then the theory in \cite{CM21,CM20,CM15APDE,mizu08} can be applied.

  Assumption \ref{ass:VN} is  very important  in the theory developed by Kowalczyk  \textit{{et al.}} in \cite{KMM3,KMM4}.
Here we are able to  generalize their ideas thanks to the theory by  Deift  and Trubowitz \cite[Sect. 3]{DT}, which treats in   great   generality the transformations considered in
Chang \textit{{et al.}} \cite{Chang}.        It is then possible to develop the two virial inequalities of Kowalczyk  \textit{{et al.}} in \cite{KMM3} in the presence of    general discrete spectra and to combine the theory of dispersion in \cite{KMM3} with the theory of nonlinear dissipation of the discrete modes in \cite{CM21,CM20}. The latter  was initiated by Buslaev and Perelman \cite{BP2}, was then considered by  Soffer and Weinstein \cite{SW3}  and generalized in papers like \cite{CM15APDE}. The approach in \cite{CM21,CM20}, to which we refer for a more detailed discussion on this point,  is far simpler than the previous literature, thanks to the notion of Refined Profile.

 Recently, there has been a considerable interest on the asymptotic stability of patterns for dispersive equations with inhomogeneities and with long range nonlinearities in 1D, especially in view of the analysis of kinks of appropriate wave equations.

 Kowalczyk  \textit{{et al.}}  showed that a framework based on virial inequalities is very suitable and provides a very penetrating grasp of these problems, starting with  their   partial proof in   \cite{KMM2} of the asymptotic stability of the kinks of the $\phi ^4$ model, with other insightful contributions in papers like
   \cite{KMM3,KMM4}. See  also the work of  Alammari and    Snelson
\cite{AS1,AS2} and of Martinez  for  long range Schr\"odinger and Hartree equations \cite{Ma}  and for Zakharov systems   \cite{Ma1} in  one dimension.

Quite different set ups from that of  Kowalczyk  \textit{{et al.}}
are in \cite{ChenPus,GP20,GPR18,MMS2,naumkin2016},   in the absence of discrete modes.  See also  the series \cite{LS1}--\cite{LS4} and \cite{Sterbenz}. Recently Chen \cite{Chen} considered the case of our NLS \eqref{nls} with one eigenvalue mode, that is $N=1$, without the repulsivity Assumption \ref{ass:VN}
while the book by Delort and Masmoudi   \cite{DM20} looked at the  $\phi ^4$ model for long times but not asymptotically.
The methods  employed in all these works  need yet to be tested in the case where there are more than one discrete modes, which potentially  will  slow   the decay, see Gang  Zhou and Sigal  \cite{ZS}, creating additional difficulties.  For the      literature which uses     dispersive estimates   in the context of short range nonlinearities,  the case $N=2$    is significantly more complicated than the case $N=1$, see  Soffer and Weinstein \cite{SW04RMP},  the series by Tsai and Yau \cite{TY1}--\cite{TY4} and,  for special situations with $N>2$,  Nakanishi et  al. \cite{NPT}.  This seems to be related to the need of using different weighted norms as the solution evolves through different stages.  More general spectra than the special ones in these references are likely to complicate this kind of analysis. Obviously
long range nonlinearities will add further complications.

Our main contribution in this paper involves the use of the notion of Refined Profile, which is significant only in the case of two or more eigenvalues, and so is not relevant to the problem considered in Chen \cite{Chen} (where however, if the potential is repulsive after a Darboux transformation, the virial inequality argument provides an alternative proof of dispersion). As we show, the Refined Profile  notion allows to avoid normal forms arguments and the search of canonical coordinates. As we have shown also elsewhere in \cite{CM20} and expecially in \cite{CM21}, we provide a very simple alternative to  \cite{TY1}--\cite{TY4}, \cite{SW04RMP} and to  the more general   \cite{CM15APDE}.  The additional complication here, compared to \cite{CM20,CM21} is the   long range nonlinearity, which does not allow to treat dispersion with a simple perturbation argument.

  To prove dispersion we follow the framework of the virial inequalities of    Kowalczyk  \textit{{et al.}} \cite{KMM3} which,   while  subtle,  is  simple and  robust and, as a consequence, is shown here to apply easily in contexts with complicated spectra. Unfortunately, an important limitation  is the repulsivity Assumption \ref{ass:VN}. Kowalczyk  \textit{{et al.}} \cite{KMM4} discuss how to avoid it in some cases, but we do not consider here the analogous situations. Our NLS problem is in some respect simpler than kink problems because  virial inequalities like in \cite{KMM4}, which involve three distinct functionals,  are more complicated than the single one that suffices here.

For alternative proofs of dispersion, we notice  that a rather simple  framework is due to Ze Li \cite{LiZe}, but the argument does not apply to the cubic nonlinearity and it too, needs to be tested in the presence of eigenvalues. Similar limitations are  true for   \cite{CVV}.   We do not discuss here  the nonlinear Steepest Descent method of Deift and Zhou initiated in \cite{DZ93}, which has been used extensively for integrable systems  but  for non integrable systems,  to our knowledge, only in   \cite{DZ02}.
There is a large literature
on PDE methods in the context of integrable systems. Here we mention only the recent paper on kinks of sine-Gordon by L\"uhrmann and Schlag \cite{LuSchlag} and the paper on the
  black  soliton  for defocusing cubic NLS  by Gravejat and  Smets \cite{GS}.

In the context of stability problems  of ground states of  the NLS, virial inequalities were introduced by
Merle and Raphael \cite{MR1}--\cite{MR4}. The papers by  Kowalczyk  \textit{{et al.}} \cite{KMM2,KMM3,KMM4} developed further applications of virial inequalities in stability problems. In this paper we show that the theory can be applied in a relatively elementary fashion also in the presence of any number $N\ge 2$ of eigenvalues. The case $N=1$ with Assumption  \ref{ass:VN} should be analogous to \cite{KMM3}. An analogue of Soffer and Weinstein \cite{SW3}, or of the more general \cite{BC11AJM},   with Assumption  \ref{ass:VN},  should hold in 1D for real valued solutions of the  Nonlinear Klein Gordon with a quadratic nonlinearity, using arguments similar to the ones of this paper, which should simplify greatly \cite{SW3,BC11AJM} with the use of Refined Profiles. The same should hold for an analogue of \cite {CMP2016} for the unitary invariant NLKG in 1D where there are small complex valued standing waves. %It is reasonable to expect that readapting  Kowalczyk  \textit{{et al.}} \cite{KMM4} with ideas analogous to the ones on Refined Profiles and Fermi Golden Rule in the present paper, should allow a generalization of the results tokinks with  internal eigenvalues.

\section{Modulation coordinate and transformed problem}\label{sec:dar}

In this section we write the equation in modulation coordinates and consider the  transformed problem induced by the conjugation relation \eqref{eq:DarConj2}.
We start from the modulation coordinate.
First of all, let
\begin{align}
\(u,v\):=\int_{\R}u\overline{v}\,dx,\ \<u,v\>:=\Re \(u,v\),
\end{align}
and set
\begin{align}\label{def:PdPc}
P_{\mathrm{d}}:=\sum_{j=1}^N\(\cdot,\phi_j\)\phi_j,\ P_{\mathrm{c}}:=1-P_{\mathrm{d}}.
\end{align}
Then, the space $P_c L^2$ is the continuous space w.r.t.\ $H$. Recalling the refined profile $\phi[\mathbf{z}]$ from Proposition \ref{prop:rp},
we introduce the following analogue of  nonlinear continuous space of Gustafson, Nakanishi and Tsai \cite{GNT},
\begin{align}
\mathcal{H}_c[\mathbf{z}]:=\{v\in L^2\ |\ \forall \widetilde{\mathbf{z}}\in \C^N,\ \<\im v, D_{\mathbf{z}}\phi[\mathbf{z}]\widetilde{\mathbf{z}}\>=0\}. \label{genker}
\end{align}

The modulation coordinates are given by decomposing $u$ as follows.
\begin{lemma}\label{lem:lincor}
There exist  $\delta>0$ and $\mathbf z\in C^\infty (   \mathcal{B}_{  \Sigma^{-1}}  (0, \delta)  , \C^N)$ s.t.\ $\eta(u):=u-\phi[\mathbf{z}(u)]\in \mathcal{H}_c[\mathbf{z}(u)]$.
Further, we have
\begin{align}\label{eq:lincor1}
\|\mathbf{z}(u)\|+\|\eta(u)\|_{H^1}\sim \|u\|_{H^1}.
\end{align}
\end{lemma}

\begin{proof}
The proof is standard, so we omit it.
\end{proof}
In the following we write $\mathbf{z}=\mathbf{z}(u)$ and $\eta = \eta(u)$.
Notice that the bound \eqref{eq:main1} is a straightforward consequence of Mass and Energy conservation, which imply $\|u\|_{H^1} \lesssim \delta$, and of \eqref{eq:lincor1}.

Substituting $u=\phi[\mathbf{z}]+\eta$ in \eqref{nls} and using \eqref{eq:rp}, we obtain
\begin{align}\label{eq:modnls}
\im \partial_t \eta + \im D_{\mathbf{z}}\phi[\mathbf{z}]\( \dot {\mathbf{z}} + \im \boldsymbol{\varpi}(\mathbf{z})\mathbf{z}\)=H[\mathbf{z}]\eta + \sum_{\mathbf{R}_{\mathrm{min}}}\mathbf{z}^{\mathbf{m}}G_{\mathbf{m}} +\mathcal{R}_{\mathrm{rp}}[\mathbf{z}]+F[\mathbf{z},\eta]  + |\eta|^2\eta   ,
\end{align}
where $H[\mathbf{z}]:=H+ L[\mathbf{z}]$,
\begin{align} \label{linerror}&
L[\mathbf{z}]:= 2|\phi[\mathbf{z}]|^2+ \phi[\mathbf{z}]^2\mathrm{C},\text{  with the complex conjugation operator } \mathrm{C}u=\overline{u}\ \text{and}\
\\& F[\mathbf{z},\eta]:=2\phi[\mathbf{z}]|\eta|^2 + \overline{\phi[\mathbf{z}]}\eta^2 . \label{quadrat}
\end{align}
Following Gustafson, Nakanishi and Tsai \cite{GNT}, we can construct an inverse of $P_c$ on $\mathcal{H}_c[\mathbf{z}]$.

\begin{lemma}\label{lem:GNTR}
There exists $\delta>0$ s.t.\ there exists $\{a_{jA}\}_{j=1,\cdots,N,A=\mathrm{R},\mathrm{I}}\in C^\infty(B_{\C^N(0,\delta)},\Sigma^1)$ s.t.
\begin{align}\label{eq:estRa}
\|a_{jA}(\mathbf{z})\|_{\Sigma^1}\lesssim \|\mathbf{z}\|^2, \ j=1,\cdots,N,\ A=\mathrm{R},\mathrm{I}
\end{align}
and
\begin{align}\label{eq:GNTR1}
R[\mathbf{z}]:=\mathrm{Id}-\sum_{j=1}^N\(\<\cdot,a_{j\mathrm{R}}(\mathbf{z})\>\phi_j+\<\cdot,a_{j\mathrm{I}}(\mathbf{z})\>\im \phi_j\),
\end{align}
satisfies $\left.R[\mathbf{z}]P_c\right|_{\mathcal{H}_c[\mathbf{z}]}=\left.\mathrm{Id}\right|_{\mathcal{H}_c[\mathbf{z}]}$, $\left.P_cR[\mathbf{z}] \right|_{P_c L^2}=\left.\mathrm{Id} \right|_{P_c L^2}$.
\end{lemma}

\begin{proof} The proof, which we skip,  is an analogue of that in   \cite{CM15APDE}, which in turn generalizes the one in \cite{GNT}.
\end{proof}

We set $\widetilde{\eta}:=P_c \eta$.
Then, by Lemma \ref{lem:GNTR} we have $\eta=R[\mathbf{z}]\widetilde{\eta}$ and furthermore, from the estimate \eqref{eq:estRa}, we have
\begin{align}\label{eq:etaequiv}
\|\eta\|_{H^s_a}\sim \|\widetilde{\eta}\|_{H^s_a},
\end{align}
for $s=0,1$ and $|a|\leq a_1$.
Applying $P_c$ to \eqref{eq:modnls}, we obtain
\begin{align}\label{eq:tildeeta}
\im \partial_t \widetilde{\eta}=&H\widetilde{\eta}+\sum_{\mathbf{m}\in \mathbf{R}_{\mathrm{min}}}\mathbf{z}^{\mathbf{m}}P_c G_{\mathbf{m}}+\mathcal{R}_{\widetilde{\eta}},
\end{align}
where
\begin{align}\label{eq:Reta}
\mathcal{R}_{\widetilde{\eta}}=P_c\(-\im D_{\mathbf{z}}\phi[\mathbf{z}]\(\dot {\mathbf{z}}+\im \boldsymbol{\varpi}(\mathbf{z})\mathbf{z}\) +\mathcal{R}_{\mathrm{rp}}[\mathbf{z}]+F[\mathbf{z},\eta]+L[\mathbf{z}] \eta  +|\eta | ^2 \eta \).
\end{align}
The rest of this section is framed like in \cite{KMM3}.

\noindent We will consider constants  $A, B,\varepsilon >0$ satisfying
 \begin{align}\label{eq:relABg}
\log(\delta ^{-1})\gg\log(\epsilon ^{-1})\gg  A\gg    B^2\gg B \gg  \exp \( \varepsilon ^{-1} \) \gg 1.
 \end{align}
We will denote by    $o_{\varepsilon}(1)$  constants depending on $\varepsilon$ such that
 \begin{align}\label{eq:smallo}
 \text{ $o_{\varepsilon}(1) \xrightarrow {\varepsilon  \to 0^+   }0.$}
 \end{align}
 For the two virial inequalities,  we will use  the approximations    in    \cite{KMM3} of   $2^{-1}\< \im \( 1/2+x\partial _x\) u(t), u(t)\>$, which is the quantized analogue of the form $2 ^{-1}x\cdot \xi$ for a finite dimensional  hamiltonian system $\dot x= \nabla _{\xi} E$ and $\dot \xi= -\nabla _{x} E$ with energy $2^{-1}|\xi |^2+V(x)$ (recall that $\frac{d}{dt}(2 ^{-1}x\cdot \xi) =|\xi |^2-x\cdot \nabla V(x)  $ where, if $V$ is repulsive as in Definition  \ref{def:repulPot},
  then $x\cdot \xi$ is strictly increasing for all $t$: this simple classical argument explains the heuristics around the notion of Virial Inequalities).

\noindent The first virial inequality, Sect. \ref{sec:1virial},   involves  a truncation of   the function $x$  outside an interval of size $\sim A ^{-1}$  around 0. The fact that the initial  potential $V=V_1$  is obviously not repulsive, makes necessary another virial inequality.  This involves applying  the operator $\mathcal{A}^*$ to \eqref{eq:tildeeta}, see \eqref{def:calA},  in order to exploit the conjugation \eqref{eq:DarConj2}, which transforms $H_1$ into the repulsive operator $H_{N+1}$. However, to offset the loss of regularity due to $\mathcal{A}^*$, which is a differential operator of order $N$, it is necessary to use a regularization and consider
\begin{align}\label{def:Tg}
\mathcal{T} :=\< \im \varepsilon\partial_x\>^{-N} \mathcal{A}^* ,
\end{align}
  for $\varepsilon>0$.  However,  this does not work either, because symmetries  of the nonlinear  term $P_c(|\eta | ^2 \eta )$, used to get estimates like \eqref{eq:1stvcub} below, do not hold any more.  This is why the argument considers $\chi\in C_c^\infty(\R,\R)$ such that
\begin{align}\label{def:chi}
x \chi'(x)\leq 0\ \text{ and }1_{|x|\leq 1}\leq \chi \leq 1_{|x|\leq 2},
\end{align}
and  $\chi_C:=\chi(\cdot/C)$ for $C>0$. Then   multiplying  equation \eqref{eq:tildeeta}  by $\chi _{B^2}$, obtaining
\begin{align}\label{eq:tildeetacutoff}
\im \partial_t     (\chi _{B^2}\widetilde{\eta})=&H(\chi _{B^2}\widetilde{\eta}) + \(2\chi_{B^2}'\partial_x + \chi_{B^2}''\)\tilde{\eta} +\sum_{\mathbf{m}\in \mathbf{R}_{\mathrm{min}}}\mathbf{z}^{\mathbf{m}}\chi _{B^2}P_c G_{\mathbf{m}}+\chi _{B^2}\mathcal{R}_{\widetilde{\eta}},
\end{align}
 setting
\begin{align}\label{def:vBg}
v  :=\mathcal{T}  \chi_{B^2}\widetilde{\eta},
\end{align}
 and applying $\mathcal{T}  $   to  \eqref{eq:tildeetacutoff},  we obtain
\begin{align}\label{eq:vBg}
\im \partial_t v  =H_{N+1}v   + \sum_{\mathbf{m}\in \mathbf{R}_{\mathrm{min}}}\mathbf{z}^{\mathbf{m}}\widetilde{G} _{\mathbf{m}}+\mathcal{R}_{v  },
\end{align}
where
\begin{align}\label{def:Gbgm}
 {\widetilde{G}}_{ \mathbf{m}}&:=\mathcal{T} \chi_{B^2} P_c G_{\mathbf{m}},\\
\mathcal{R}_{v   }&:=\mathcal{T}  \chi_{B^2}\mathcal{R}_{\widetilde{\eta}} +\< \im \varepsilon\partial_x\>^{-N}[V_{N+1},\< \im \varepsilon \partial_x\>^N]v  +\mathcal{T}  \(2\chi_{B^2}'\partial_x + \chi_{B^2}''\)\tilde{\eta}.\label{def:RvBg}
\end{align}
Here it is possible to apply the second virial inequality, which involves a truncation of $x$ in an interval centered in 0 of size $\sim B$. The technical fact that $A\gg B^2$ is required to work out the argument, see also  \cite{KMM3}.

 \section{The continuation argument}

The proof of \eqref{eq:main2} in Theorem \ref{thm:main} is by means of a continuation argument. In particular, we will show the following.
\begin{proposition}\label{prop:continuation}
There exists  a    $\delta _0= \delta _0(\epsilon )   $ s.t.\  if    \eqref{eq:main2}
holds  for $I=[0,T]$ for some $T>0$ and for $\delta  \in (0, \delta _0)$
then in fact for $I=[0,T]$    inequality   \eqref{eq:main2} holds   for   $\epsilon$ replaced by $\epsilon/2$.
\end{proposition}
By completely routine arguments, which we skip,
 it is possible to show that {Proposition} \ref{prop:continuation}  implies   \eqref{eq:main2} with $I=[0,\infty)$.
 We reformulate  the continuation argument. Let $a>0$ be the one given in Theorem \ref{thm:main}.
 Without loss of generality we can assume $a\leq 2^{-1}\min(a_0,a_1)$ where $a_0$ is given in \eqref{ass:V} and $a_1$ is given in \eqref{eq:sigmas}.
  We introduce the following norm,
 \begin{align}\label{eq:norm_rho}
 \|f\|_{\widetilde{\Sigma} }^2=  \<\(-\partial_x^2+ \sech ^2\( \frac{a  x}{10}   \)   \)  f,f\> \sim \|f\|_{\dot{H}^1}^2+\|f\|_{L^2_{-\frac{a}{10}}}^2.
 \end{align}
 For $C=A,B$, we set
 \begin{align}\label{eq:zeta}
  &    \zeta _C (x):=\exp   \(  - \frac{|x|}{C}   \(1-\chi (x)\)   \) .
 \end{align}
% and
% \begin{align}
%  \varphi _C(x) =\int _0^x  \zeta _C ^2(t)dt,\ \psi_C(x):=\chi_{C^2}(x)^2\varphi_{C^2}(x)
% \end{align}
We consider the main variables in \cite{KMM3}, given by
 \begin{align}\label{def:wAxiB}&
 w :=\zeta_A \widetilde{\eta}  \text{  and  }
 \xi :=\chi_{B^2}\zeta_B v   \text{ . }
 \end{align}
 We will prove the following  continuation argument.
 \begin{proposition}\label{prop:contreform} For any small $\epsilon>0 $
there exists  a    $\delta _0 = \delta _0(\epsilon )   $ s.t.\  if    in  $I=[0,T]$ we have
\begin{align}&
\|\dot {\mathbf{z}}+\im \boldsymbol{\varpi}(\mathbf{z})\mathbf{z}\|_{L^2(I)}+\sum_{\mathbf{m}\in \mathbf{R}_{\mathrm{min}}}\|\mathbf{z}^\mathbf{m}\|_{L^2(I)}+ \| \xi  \|_{L^2(I, \widetilde{\Sigma})}+\|  w \|_{L^2(I, \widetilde{\Sigma})}\le   \epsilon   \label{eq:main11}
\end{align}
then  for $\delta  \in (0, \delta _0)$
     inequality   \eqref{eq:main11} holds   for   $\epsilon$ replaced by $ o_{\delta}(1)   \epsilon $    where $o_{\delta}(1) \xrightarrow {\delta  \to 0^+   }0 $.
\end{proposition}
Notice that Proposition \ref{prop:contreform} implies Proposition \ref{prop:continuation}.
In the following, we always assume the assumptions of the claim of Proposition \ref{prop:contreform}.

 In complete analogy to \cite{KMM3}, we consider two virial estimates, one for $w$ and the other for $\xi$.  The first is based directly on the equation for $\widetilde{\eta}$,  \eqref{eq:tildeeta}.

 \begin{proposition}[Virial estimate for $\widetilde{\eta}$]\label{prop:1stvirial}
 We have
 \begin{align}\label{eq:1stv}
 \|  w '\|_{L^2(I, L^2)}\lesssim A^{1/2}\delta  +\|w \|_{L^2(I,L^2_{-\frac{a}{10}})}+\sum_{\mathbf{m}\in \mathbf{R}_{\min}}\|\mathbf{z}^{\mathbf{m}}\|_{L^2(I)}+ \epsilon ^2.
 \end{align}
 \end{proposition}

The second  virial estimate,  involves  the transformed problem \eqref{eq:vBg}.

\begin{proposition}[Virial estimate for $v  $]\label{prop:2ndvirial}
Let $A\gg B^2$.
We have
\begin{align}\label{eq:2ndv}
\|\xi \|_{L^2 (I,\widetilde{\Sigma})}\lesssim B   \varepsilon^{-N}    \delta   + \sum_{\mathbf{m}\in \mathbf{R}_{\mathrm{min}}} \|\mathbf{z}^{\mathbf{m}}\|_{L^2(I)}+ o_{\varepsilon}(1) \(   \| w  \| _{L^2(I,\widetilde{\Sigma})}    +   \|\dot {\mathbf{z}}+\im \boldsymbol{\varpi}(\mathbf{z})\mathbf{z}\|  _{L^2(I)} \).
\end{align}
\end{proposition}

The term $\dot {\mathbf{z}}+\im \varpi(\mathbf{z})\mathbf{z}$  can be controlled in term of the $\mathbf{z}^{\mathbf{m}}$, for $\mathbf{m}\in \mathbf{R}_{\mathrm{min}}$.
\begin{proposition}\label{lem:estdtz}
 We have
 \begin{align}
 \|\dot {\mathbf{z}}+\im \varpi(\mathbf{z})\mathbf{z}\|_{L^2(I)}\lesssim \sum_{\mathbf{m}\in \mathbf{R}_{\mathrm{min}}}\|\mathbf{z}^{\mathbf{m}}\|_{L^2(I)}+ \delta^2\epsilon .   \label{eq:lem:estdtz}
 \end{align}
 \end{proposition}
 To bound the $\mathbf{z}^{\mathbf{m}}$, for $\mathbf{m}\in \mathbf{R}_{\mathrm{min}}$, we will use the   Fermi Golden Rule.
 \begin{proposition}[FGR estimate]\label{prop:FGR}
 We have
 \begin{align}\label{eq:FGRint}
 \sum_{\mathbf{m}\in \mathbf{R}_{\mathrm{min}}}\|\mathbf{z}^{\mathbf{m}}\|_{L^2(I)}\lesssim  \varepsilon ^{-N}B ^{2+2\tau}\delta +B ^{-\frac{1}{2}}   \epsilon   + \epsilon ^2.
 \end{align}
 \end{proposition}

In Sect.\ \ref{sec:1virial} we prove Proposition \ref{prop:1stvirial}. In Sect.\ \ref{sec:vir2} we prove  Proposition \ref{prop:2ndvirial}.
Sections \ref{sec:1virial}--\ref{sec:vir2} are very close to \cite{KMM3}.  Sect.\      \ref{sec:coerc} is the analogue of   \cite[Sect. 5.1]{KMM3}.
The proofs of Propositions
\ref{lem:estdtz}  and \ref{prop:FGR} are very similar to the discussion in \cite{CM21}. As is in \cite{BP2,SW3} and many other papers, most referenced in \cite{CM15APDE}, at some point the continuous mode, in fact in this paper the variable $v$,  needs to be decomposed in a part which resonates with the discrete mode $\mathbf{z}$ and a remainder which is supposed to be very small, and which we denote by $g$, see \eqref{eq:expan_v}.  To bound $g$ we use Kato  smoothing estimates, as in the previous literature. So, for example, Lemmas \ref{lem:smooth} and  \ref{lem:lemg9} are a typical tool, see  \cite{BP2,SW3} or the references in \cite{CM15APDE}. Some attention is needed in the use of the weights, to guarantee that some of the terms, i.e. the term in \eqref{eq:decomp1}, are small. Finally,  in Sect.\  \ref{sec:energystab1}, we prove the last sentence of Theorem \ref{thm:main}.

\section{Proof of Proposition \ref{prop:1stvirial}} \label{sec:1virial}
In this section, we prove the 1st virial estimate Proposition \ref{prop:1stvirial}, which is a consequence of the estimate of the time derivative of the following functional,
\begin{align}\label{eq:def I}
\mathcal{I}   (\widetilde{\eta}):=\frac{1}{2}\<\widetilde{\eta},\im S_A\widetilde{\eta}\>,
\end{align}
where the anti-symmetric operator $S_A$ is defined by
\begin{align}\label{eq:defSA}
S_A:=\frac{\varphi_A'}{2}+\varphi_A\partial_x,\ \text{where}\
\varphi_A(x)=\int_0^x \zeta_A    ^2(y)\,dy.
\end{align}

Proposition \ref{prop:1stvirial} is a direct consequence of the following estimate

\begin{proposition}[1st virial estimate in differential form]\label{prop:1stvirial2}
Under the assumptions of Proposition \ref{prop:continuation}, for sufficiently small $\delta>0$  we have
\begin{align}\label{eq:1stvirialdiff}
 \frac{d}{dt}{\mathcal{I}} (\widetilde{\eta})  +\frac{1}{2}\|w  '\|_{L^2}^2 \lesssim  \|w  \|_{L^2_{-\frac{a}{10}}}^2 +\sum_{\mathbf{m}\in \mathbf{R}_{\mathrm{min}}}|\mathbf{z}^{\mathbf{m}}|^2+\delta ^2\|\dot {\mathbf{z}}+\im \boldsymbol{\varpi}(\mathbf{z})\mathbf{z}\|^2.
\end{align}
\end{proposition}

We first prove Proposition \ref{prop:1stvirial} from Proposition \ref{prop:1stvirial2}.
\begin{proof}[Proof of Proposition \ref{prop:1stvirial}]
We have $|\mathcal{I}   (\widetilde{\eta})|\lesssim A\|\widetilde{\eta}\|_{H^1}^2\lesssim A \delta^2$.
Thus, integrating \eqref{eq:1stvirialdiff} over $[0,T]$, we obtain \eqref{eq:1stv}.
\end{proof}

The rest of this section is devoted for the proof of Proposition \ref{prop:1stvirial2}.
First, from \eqref{eq:tildeeta}, we have
\begin{align}
 \frac{d}{dt}\mathcal{I}  (\widetilde{\eta})&=-\<\im \partial_t\widetilde{\eta},S_A\widetilde{\eta}\>\nonumber\\&=-\<H\widetilde{\eta},S_A \widetilde{\eta}\> - \sum_{\mathbf{m}\in \mathbf{R}_{\mathrm{min}}}\<\mathbf{z}^{\mathbf{m}}P_c G_{\mathbf{m}},S_A \widetilde{\eta}\>-\<\mathcal{R}_{\widetilde{\eta}},S_A\widetilde{\eta}\>.\label{eq:1stv1}
\end{align}
We will compute each terms in \eqref{eq:1stv1}

\begin{lemma}\label{lem:1stv1}
We have
\begin{align}\label{Quadrmain1}
\<H\widetilde{\eta},S_A \widetilde{\eta}\>=\<\(-\partial_x^2 -\frac{\varphi_A}{2\zeta_A^2}V'\)w  ,w   \>+\frac{1}{2A}\<V_0w  ,w  \>,
\end{align}
where
\begin{align}\label{def:V0}
V_0(x):=\(\chi'' |x| +2\chi' \frac{x}{|x|}\) .
\end{align}
\end{lemma}

\begin{proof}
By direct computation, we have
 \begin{align}
 \<H\widetilde{\eta},S_A \widetilde{\eta}\>=
  \<  \varphi '_A \partial_x\widetilde{\eta} ,\partial_x\widetilde{\eta} \> - \frac{1}{4}\<  \varphi  ^{\prime\prime\prime}_A \widetilde{\eta}   ,\widetilde{\eta}    \>  -\frac{1}{2} \<    \widetilde{\eta}  ,   \varphi _A      V  '     \widetilde{\eta}  \>  . \label{eq:lem1:1}
 \end{align}
  Following Lemma 1 of  \cite{KMM3},   we have
 \begin{align}
   \<  \varphi ' _A \partial_x\widetilde{\eta} ,\partial_x\widetilde{\eta} \>  =
%    \<\zeta_A^2\(\frac{w}{\zeta_A}\)', \(\frac{w}{\zeta_A}\)'\>= \< w' - \frac{\zeta _A'}{\zeta _A} w ,w' - \frac{\zeta _A'}{\zeta _A} w \> \nonumber \\&  =\< w'   ,w'     \>  +\<  \( \frac{\zeta _A'}{\zeta _A}\)^2 w ,w  \> -2 \< w'   ,  \frac{\zeta '_A}{\zeta_A} w \>   \nonumber     \\& = \< w'   ,w'     \>  +  \< \(  \( \frac{\zeta '_A}{\zeta _A}\) ' + \( \frac{\zeta '_A}{\zeta _A}\) ^2 \)     w   ,w      \>  =
      \< w'   ,w'     \>  +  \<   \frac{\zeta ^{\prime\prime} _A}{\zeta _A}     w   ,w     \> ,  \label{eq:lem1:1--}
 \end{align}
 and
 \begin{align}
  - \frac{1}{4}\<  \varphi   ^{\prime\prime\prime}_A \widetilde{\eta}    ,\widetilde{\eta}   \> =
%  - \frac{1}{4}\<   \frac{(\zeta ^2_A) ^{\prime\prime}}{\zeta ^2_A} w   ,w   \>  =
   -\frac{1}{2}\<  \( \frac{ \zeta ^{\prime\prime}_A}{\zeta  _A}  + \( \frac{ \zeta ^{\prime } _A}{\zeta  _A} \) ^2  \) w  , w  \>   , \label{eq:lem1:2--}
 \end{align}
 so that
 \begin{align}\label{eq:lem1:2}
  \<  \varphi '_A \partial_x\widetilde{\eta}   ,\partial_x\widetilde{\eta}    \> - \frac{1}{4}\<  \varphi  ^{\prime\prime\prime}_A \widetilde{\eta}   ,\widetilde{\eta}     \> =\<  -\partial _x^2w,w\>+\frac12\< \(\frac{\zeta_A''}{\zeta_A}-\(\frac{\zeta_A'}{\zeta_A}\)^2\)w,w\>.
 \end{align}
 Substituting $w=\zeta_A\widetilde{\eta}$ also in $\<    \widetilde{\eta}   ,   \varphi _A      V  '     \widetilde{\eta}   \> $  and using the identity
 \begin{align}
 A\(\frac{ \zeta ^{\prime\prime}_A}{\zeta  _A}  - \( \frac{ \zeta ^{\prime } _A}{\zeta  _A} \) ^2\)=\chi''(x)|x|+2\chi'(x)\frac{x}{|x|}=V_0(x),
 \end{align}
we obtain \eqref{Quadrmain1}.
\end{proof}

\begin{lemma}\label{lem:1stv2}
We have
\begin{align}
\left| \<\mathbf{z}^{\mathbf{m}}P_c G_{\mathbf{m}},S_A \widetilde{\eta}\>\right|\lesssim |\mathbf{z}^{\mathbf{m}}| \| w\|_{L^2_{-\frac{a}{10}}}.
\end{align}
\end{lemma}

\begin{proof}
Since $\|\zeta_A^{-1}S_A P_c G_{\mathbf{m}}\|_{L^2_{\frac{a}{10}}}\lesssim 1$, the conclusion is obvious.
\end{proof}

\begin{lemma}\label{lem:1stv3}
We have
\begin{align}\label{eq:Retaest}
|\<\mathcal{R}_{\widetilde{\eta}},S_A\widetilde{\eta}\>|\lesssim \delta ^2\( \(\sum_{\mathbf{m}\in \mathbf{R}_{\mathrm{min}}}|\mathbf{z}^{\mathbf{m}}|+ \|\dot {\mathbf{z}}+\im \boldsymbol{\varpi}(\mathbf{z})\mathbf{z}\|\) \| w   \|_{L^2_{-\frac{a}{10}}}+\|w   \|_{L^2_{-\frac{a}{10}}}^2\)+\delta ^{2/3}\|w'\|_{L^2}^2.
\end{align}
\end{lemma}

\begin{proof}
We will estimate the contribution of each term in $\mathcal{R}_{\widetilde{\eta}}$, see \eqref{eq:Reta}.
First, since, by $D_{\mathbf{z}}\phi[0]\widetilde{\mathbf{z}}=\boldsymbol{\phi}\cdot \widetilde{\mathbf{z}}$
\begin{align}\label{eq:Reta10}\|P_c D_{\mathbf{z}}\phi[\mathbf{z}] \widetilde{\mathbf{z}}\|_{\Sigma^1}= \|P_c \( D_{\mathbf{z}}\phi[\mathbf{z}] -D_{\mathbf{z}}\phi[0]\) \widetilde{\mathbf{z}}\|_{\Sigma^1}\lesssim \delta ^2\|\widetilde{\mathbf{z}}\| ,
\end{align}
 we have
\begin{align}\label{eq:Reta1}
|\<P_c\(-\im D_{\mathbf{z}}\phi[\mathbf{z}]\( \dot {\mathbf{z}}+\im \boldsymbol{\varpi}(\mathbf{z})\mathbf{z}\)\),S_A\widetilde{\eta}\>|\lesssim \delta  ^2\|\dot {\mathbf{z}}+\im \boldsymbol{\varpi}(\mathbf{z})\mathbf{z}\|  \|w   \|_{L^2_{-\frac{a}{10}}}.
\end{align}
Next, from \eqref{est:Rrp}, we have
\begin{align}\label{eq:Reta2}
|\<P_c \mathcal{R}_{\mathrm{rp}}[\mathbf{z}],S_A\widetilde{\eta}\>|\lesssim \delta  ^2 \sum_{\mathbf{m}\in \mathbf{R}_{\min}}|\mathbf{z}^{\mathbf{m}}| \|w   \|_{L^2_{-\frac{a}{10}}}.
\end{align}
We next estimate the contribution of $P_c\(F[\mathbf{z},\eta]+ L[\mathbf{z}]\eta   +|\eta |^2 \eta  \)$.
First, since $P_c=1-P_d$ and $\|P_d S_A \widetilde{\eta}\|_{\Sigma^1}\lesssim \|w  \|_{L^2_{-\frac{a}{10}}}$, we have
\begin{align}
|\<P_d\(F[\mathbf{z},\eta]+ L[\mathbf{z}]\eta   +|\eta |^2 \eta \),S_A\widetilde{\eta}\>|&\lesssim \|F[\mathbf{z},\eta]+ L[\mathbf{z}]\eta   +|\eta |^2 \eta \|_{\Sigma^{-1}}\|w  \|_{L^2_{-\frac{a}{10}}}\nonumber\\&\lesssim \delta  ^2 \|w \|_{L^2_{-\frac{a}{10}}}^2.\label{eq:1stvdisc}
\end{align}
Next, by elementary integration by parts we have
\begin{align*}
 \<|\widetilde{\eta}|^2\widetilde{\eta},S_A \widetilde{\eta}\> = 2^{-1}\<|\widetilde{\eta}|^4, \varphi _A'\>  +  2^{-2}\< \( |\widetilde{\eta}|^4 \) ', \varphi _A\> = 4^{-1}\<|\widetilde{\eta}|^4, \zeta _A^2\>
\end{align*}
and by \cite{KMM3} and \eqref{eq:main1}, see also Lemma 2.7 \cite{CM19SIMA},   we have
\begin{align}\label{eq:1stvcub}
|\<|\widetilde{\eta}|^2\widetilde{\eta},S_A \widetilde{\eta}\>|\lesssim \delta  ^{2/3}\| w '\|_{L^2}^2.
\end{align}
For the remaining terms, by $\eta = \widetilde{\eta}+\widetilde{\eta_1}$ with $\widetilde{\eta}_1=(R[z]-1)\widetilde{\eta}$, we can expand
\begin{align}
&F[\mathbf{z},\eta]+ L[\mathbf{z}]\eta   +|\eta |^2 \eta-|\widetilde{\eta}|^2\widetilde{\eta}  =v_0 \widetilde{\eta_1}+v_1\overline{\widetilde{\eta_1}}+v_2 \widetilde{\eta}+v_3\overline{\widetilde{\eta}}+v_4\widetilde{\eta}^2+v_5|\widetilde{\eta}|^2,\ \text{where}\nonumber \\
&v_0=2|\phi[\mathbf{z}]|^2+2\phi[\mathbf{z}]\overline{\widetilde{\eta_1}}+\overline{\phi[\mathbf{z}]}\widetilde{\eta_1}+|\widetilde{\eta_1}|^2,\ v_1=\phi[\mathbf{z}]^2,\
v_2=2|\phi[\mathbf{z}]|^2+2\phi[\mathbf{z}]\overline{\widetilde{\eta_1}}+\overline{\phi[\mathbf{z}]}\widetilde{\eta_1}+2|\widetilde{\eta_1}|,\nonumber \\
&v_3=\phi[\mathbf{z}]^2+2\phi[\mathbf{z}]\widetilde{\eta_1}+\widetilde{\eta_1}^2,\
v_4=\overline{\phi[\mathbf{z}]}+\overline{\widetilde{\eta_1}},\
v_5=2\phi[\mathbf{z}]+\widetilde{\eta_1}. \label{eq:F1stvirial}
\end{align}
By Lemma \ref{lem:GNTR}, we have  $\|\widetilde{\eta_1}\|_{\Sigma^1}\lesssim \|w \|_{L^2_{-\frac{a}{10}}}$ and
\begin{align}
&  \text{$\|v_j\|_{\Sigma^1}\lesssim \delta ^2$ for $j=0,1,2,3$ and $\|v_j\|_{\Sigma^1}\lesssim \delta  $ for $j=4,5$.}    \label{eq:vjstvirial}
\end{align}
Thus, we have
\begin{align}
&|\<v_0\widetilde{\eta_1}+v_1\overline{\widetilde{\eta_1}},S_A\widetilde{\eta}\>|\lesssim \|v_0\widetilde{\eta_1}+v_1\overline{\widetilde{\eta_1}}\|_{L^2_{\frac{a}{5}}}\|w   \|_{L^2_{-\frac{a}{10}}}\lesssim \delta  ^2 \|w  \|_{L^2_{-\frac{a}{10}}}^2,\label{eq:v1quad1}\\
&|\<v_2\widetilde{\eta}+v_3\overline{\widetilde{\eta}}+v_4\widetilde{\eta}^2+v_5|\widetilde{\eta}|^2,\varphi_A'\widetilde{\eta}\>|\lesssim \delta ^2 \|  w  \|_{L^2_{-\frac{a}{10}}}^2,\label{eq:v1quad2}
\end{align}
and
\begin{align}
&\left|\<\(v_2+v_4\widetilde{\eta}+v_5\overline{\widetilde{\eta}}\)\widetilde{\eta},\varphi_A\partial_x\widetilde{\eta}\>\right|=\frac{1}{2}\left|\int \zeta_A^{-2}\partial_x\(\varphi_A\(v_2+v_4\widetilde{\eta}+v_5\overline{\widetilde{\eta}}\)\)|w |^2\right|\lesssim \delta  ^2\|w  \|_{L^2_{-\frac{a}{10}}}^2,\label{eq:v1quad3}\\
&|\<v_3\overline{\widetilde{\eta}},\varphi_A \partial_x\widetilde{\eta}\>|=\frac{1}{2}\left|\int \zeta_A^{-2} \partial_x\(\varphi_A v_3\)\overline{w}^2\right|\lesssim \delta  ^2\|w  \|_{L^2_{-\frac{a}{10}}}^2.\label{eq:v1quad4}
\end{align}
Combining \eqref{eq:1stvdisc}, \eqref{eq:1stvcub}, \eqref{eq:v1quad1}, \eqref{eq:v1quad2}, \eqref{eq:v1quad3} and \eqref{eq:v1quad4} we obtain
\begin{align}\label{eq:estF}
|\<F[\mathbf{z},\eta]+ L[\mathbf{z}]\eta   +|\eta |^2 \eta-|\widetilde{\eta}|^2\widetilde{\eta},S_A\widetilde{\eta}\>|\lesssim \delta  ^2\|   w \|_{L^2_{-\frac{a}{10}}}^2+\delta  ^{2/3}\|   w  '\|_{L^2}.
\end{align}

Therefore, from \eqref{eq:Reta1}, \eqref{eq:Reta2} and \eqref{eq:estF} we have the conclusion.
\end{proof}

\begin{proof}[Proof of Proposition \ref{prop:1stvirial2}]
By \eqref{eq:1stv1}, Lemmas \ref{lem:1stv1}, \ref{lem:1stv2} and \ref{lem:1stv3} we obtain the estimate \eqref{eq:1stvirialdiff} for sufficiently small $\delta>0$ and large $A$, satisfying \eqref{eq:relABg}.
\end{proof}

%\begin{remark}For the proof of \ref{prop:1stvirial2}, we have not used the bootstrap bound \eqref{eq:main2}.\end{remark}

Before proving  Proposition \ref{prop:2ndvirial} we need some technical preliminaries, which we state in Sect.\  \ref{sec:Preliminary}.

\section{Technical estimates }\label{sec:Preliminary}

In this section, we collect estimates used in the sequel of the paper.

\begin{lemma}\label{lem:equiv_rho}  Let $U\ge 0$ be a non--zero   potential $U\in L^1(\R , \R )$. Then
   there exists a constant $C  _{U}>0$ such that
for any    function $0\le W$  such that $ \<x\> W\in L^1(\R )$ then
 \begin{align}\label{eq:lem:rhoequiv}
& \<  W f,f\> \le    C_U \| \<x\> W  \| _{L^1(\R )}  \< (-\partial _x^2+ U )f,f\>.\end{align}
In particular,  for  $a>0$  the constant in the norm $\|\cdot\|_{\widetilde{\Sigma}}$ in \eqref{eq:norm_rho},   there exists a constant $C (a )>0$ such that
 \begin{align}\label{eq:lem:rhoequiva}
& \<  W f,f\> \le    C(a ) \| \<x\> W  \| _{L^1(\R )}   \|f\|_{\widetilde{\Sigma}}^2.\end{align}
\end{lemma}
\proof Let $J$ be a compact interval where $I  _{U}:= \int _{J}U(x) dx>0$.
  Let then  $x_0\in J$ s.t.
\begin{align*}
 |f(x_0)|^2\le I  _{U } ^{-1} \int _{J}|f(x )|^2U(x)  dx .
\end{align*}
Then,
\begin{align*}
  |f(x )| \le  |x-x_0|   ^{\frac{1}{2}}     \| f' \| _{L^2(\R )} + |f(x_0 )| \le  |x-x_0|   ^{\frac{1}{2}}     \| f' \| _{L^2(\R )} +I _U ^{-1/2} \<  U f,f\> ^{\frac{1}{2}}.
\end{align*}
Taking second power and multiplying by $W$
it is easy to conclude the following, which after integration yields  \eqref{eq:lem:rhoequiv},
\begin{align*}
W(x)  |f(x )| ^2\le C_U \<x\> W (x)    \< (-\partial _x^2+ U) f,f\> \text{ where }C_U=2\(  1+|x_0| + I _U ^{-1}   \).
\end{align*}
\qed

A direct consequence of Lemma \ref{lem:equiv_rho} is the following.
Recall $A\gg B^2\gg B\gg a^{-1}$, with $a$ like in the previous lemma.

\begin{corollary}\label{cor:rhoequiv}
We have
\begin{align*}
\|\widetilde{\eta}\|_{L^{2}_{-\frac{a}{10}}}+\|w  \|_{L^2_{-\frac{a}{20}}}\lesssim \|w \|_{\widetilde{\Sigma}} \text{   and }
\|\zeta_B^{-1}\xi \|_{L^2_{-\frac{a}{10}}}\lesssim \|\xi \|_{\widetilde{\Sigma}}.
\end{align*}
\end{corollary}
We will use the following standard fact.
\begin{lemma}\label{lem:pd01}  Consider a  0   order Pseudodifferential Operator ($\Psi$DO)
\begin{align}\label{eq:pd011}
&  p(x,\im \partial _x)f(x) = \int _{\R ^2 } e^{ -\im k  ( x-y)}p(x,k)  f(y) dk dy
\end{align}
with symbol $p(x,k)$ such that
\begin{align}\label{eq:pd0110}
&  | \partial _x^\alpha  \partial _k^\beta   p(x,k) |\le C_{\alpha\beta}\<k  \>^{-\beta} \text{ for all $(x,k)\in \R ^2$ and for all $(\alpha,\beta).$}
\end{align}
Then, for any $m\in \R$ and for any $\varepsilon \in (0,1]$
\begin{align}\label{eq:pd012}
&  \|  \< \varepsilon  x  \>^{-m}   p(\varepsilon x,\im \partial _x)f \| _{L^2(\R )}\le  C_m  \|  \< \varepsilon x  \>^{-m}   f \| _{L^2(\R )} \text{  for all $f \in L^2(\R )$,}
\end{align}
where each constant $C_m$ depends on finitely many of the constants $C_{\alpha\beta}$ and is independent from $\varepsilon \in (0,1]$.
\end{lemma}
 \textit{Proof (sketch)}.   We write  \begin{align}&
 p ( \varepsilon x, \im \partial _x) f (x) =   P _{1 }f + P_{2 }f \nonumber \\&  P _{j }f (x) =     \int _{\R ^2}e^{\im  k(x-x') }  p (\varepsilon x, k)\chi _j \( x-x'\)  f (x') dkdx' ,\label{eq:pd013}
  \end{align}
with $\chi _1 \in C^\infty _c(\R ,[0,1])$ a cutoff with $ \chi _1=1$ near 0 and with $\chi _2:=1-\chi _1$. Then
\begin{align*}&    \| \( \< \varepsilon x \> ^{-m} P _{1} \< \varepsilon x \> ^{ m} \) \< \varepsilon  x \> ^{-m}   f \| _{L^2(\R )}\lesssim     \|  \< \varepsilon x \> ^{-m}   f \| _{L^2(\R )} ,
  \end{align*}
  because $ \< \varepsilon x \> ^{-m} P _{1 } \< \varepsilon  x  '\> ^{ m}$ is a $\Psi$DO with symbol
 \begin{align*}&    p (\varepsilon  x, k)\chi _1 \( x-x'\) \< \varepsilon x \> ^{- m}\< \varepsilon x' \> ^{ m}\in S^{0}_{1,0,0}(\R \times \R \times \R),
  \end{align*}
   see  Definition 3.5 p. 43 \cite{taylor pdo}, with, for fixed constants $C_{\alpha \alpha '  \beta}$,
  \begin{align*}
&  \left | \partial _x^\alpha \partial _{x'} ^{\alpha '} \partial _k^\beta  \( p ( \varepsilon x, k)\chi _1 \( x-x'\) \< \varepsilon x \> ^{- m}\<  \varepsilon x' \> ^{ m}  \) \right |\le C_{\alpha \alpha '  \beta}\<k  \>^{-\beta} \text{  , }
\end{align*} for all  $(\alpha ,\alpha ', \beta , x,x',k)$ and all $\varepsilon \in (0,1].$
     Then, by the theory in Ch. II \cite{taylor pdo}, this
  $\Psi$DO defines  an operator from $L^2(\R )$ into itself whose norm can  be bounded in terms of finitely many of the $C_{\alpha \alpha '  \beta}$, and so has a finite upper bound independent from   $\varepsilon \in (0,1 ] .$
  We also have, integrating by parts with respect to $k$ in \eqref{eq:pd013} for $j=2$,
\begin{align*}&     \|  \< \varepsilon   x \> ^{-m} P _{2 }     f \| _{L^2(\R )}\lesssim
\left \|  \<  \varepsilon x\> ^{-m}  \int _{\R} \<  x -x'\> ^{-10-m}    | f (x' )| dx'  \right \| _{L^2(\R )}\\&  \lesssim \left \|    \int _{\R} \<  x -x'\> ^{-10}  \<  \varepsilon  x '\> ^{-m}  | f (x' )| dx'  \right \| _{L^2(\R )}\le     \|  \<    x \> ^{-10}     \| _{L^1(\R )} \|  \< \varepsilon  x \> ^{-m}   f \| _{L^2(\R )} ,
  \end{align*}
where   the last inequality follows from Young's inequality for convolutions, and all the constants are independent from $\varepsilon \in (0,1].$ \qed

Following \cite{KMM3} we will will use a regularizing operator $\< \im  \varepsilon \partial _x \> ^{-N} $, with $\varepsilon >0$ a small constant. We will use the following lemma.

\begin{lemma}\label{claim:l2boundII}  Consider a Schwartz function   $\mathcal{V}\in \mathcal{S}(\R, \C)$. Then, for any $L\in \N\cup \{ 0\}$ there exists a constant $C_L$ s.t. we have for all $\varepsilon \in (0,1]$
\begin{align}\label{eq2stestJ21II}
 \|    \< \im \varepsilon  \partial _x \> ^{-N}  [   \mathcal{V} , \< \im \varepsilon  \partial _x \> ^{N} ] \| _{L ^{2,-L}(\R ) \to L ^{2, L}(\R )} \le C_L \varepsilon , \end{align}
 where $L^{2,s}(\R ) $  is defined  in \eqref{eq:withL2}.

  \end{lemma}
\proof Let us consider   case   $L=0$,     \begin{align}\label{eq2stestJ21}
& \left \| \< \im \varepsilon  \partial _x \> ^{-N} \left  [   \mathcal{V} , \< \im \varepsilon  \partial _x \> ^{N}\right  ] \right \| _{L^2(\R )\to L^2(\R )}\lesssim  \varepsilon .
\end{align}  Taking Fourier transform, it  is equivalent to prove the above  $L^2\to L^2$  bound for the operator
 \begin{align*}& \int_{\R} H(k,\ell ) f(\ell )d\ell \text{  with }  H(k,\ell )= \< \varepsilon k\> ^{-N} \widehat{\mathcal{V}}  (k-\ell )
 \( \< \varepsilon  k \> ^{ N}-\< \varepsilon  \ell \> ^{ N}\) .
\end{align*}
 Multiplying numerator and denominator by $\< \varepsilon  k \> ^{ N}+\< \varepsilon  \ell \> ^{ N}$    we obtain $H(k,\ell )=   \varepsilon \widetilde{H}(k,\ell )$
 \begin{align*}&   \widetilde{H}(k,\ell )=     \< \varepsilon k\> ^{-N} \widehat{\mathcal{V}}  (k-\ell )
 (k-\ell )   \frac{ P (\varepsilon  k,\varepsilon  \ell )} { \< \varepsilon k \> ^{ N}+\< \varepsilon  \ell \> ^{ N}}
\end{align*}
 where $P $ is a $2N-1$ degree polynomial.  It is elementary that the integral operator with integral  kernel  $\widetilde{H}(k,\ell )$  is uniformly bounded in $\varepsilon $ from $L^p (\R )$ to itself, for any $p\in [1,\infty ]$, by Young's inequality, see Theorem 0.3.1 \cite{sogge}. Indeed, it is enough to prove that there exists a constant $C>0$  independent from small $\varepsilon >0$ s.t.
 \begin{align}&  \sup _{k\in \R} \int _{\R} |\widetilde{H}(k,\ell ) |d\ell <C   ,\label{eq:young1}
\end{align}
since by symmetry a similar  bound  can be proved interchanging the role of $k$  and $\ell$. Now, for $M\ge N+1$, we have for fixed $k$
\begin{align*}
  \int _{\R} |\widetilde{H}(k,\ell ) |d\ell &\lesssim   \int _{|\ell |\in \left [ \frac{|k |}{2} , 2|k | \right ]}     \< \varepsilon  k\> ^{-N}  \<k-\ell  \> ^{-M}
  \( \<  \varepsilon  k \> ^{ N-1}+\<  \varepsilon  \ell \> ^{ N-1} \)      d\ell   \\& + \int _{|\ell |\not \in \left [ \frac{|k |}{2} , 2|k | \right ]}     \< \varepsilon  k\> ^{-N}  \<k-\ell  \> ^{-M}
  \( \<  \varepsilon  k \> ^{ N-1}+\<  \varepsilon  \ell \> ^{ N-1} \)      d\ell .\end{align*}
The first integral can be bounded above, for appropriate $C_N$,  by  the elementary inequality
\begin{align*}
  C_N \int _{|\ell |\in \left [ \frac{|k |}{2} , 2|k | \right ]}     \< \varepsilon  k\> ^{-1}  \<k-\ell  \> ^{-M}
         d\ell  \le    C_N \int _{\R }      \< \ell  \> ^{-M}
         d\ell  = C_N  \| \< x \> ^{-M} \| _{L^1(\R )} ,\end{align*}
while the second can
be bounded above by
\begin{align*} &
  2  \int _{|\ell|\le  \frac{|k |}{2}}     \< \varepsilon  k\> ^{-N}   \frac{\<  \varepsilon  k \> ^{ N-1} } { \<    \ell \> ^{M}}
         d\ell  + 2 \int _{|\ell|\ge  2   |k |  }   \frac{ \<  \varepsilon  \ell \> ^{ N-1}} {  \<    \ell \> ^{M}}
   \\&    \le 4 \int _{\R }   \<    \ell \> ^{N-1-M} d\ell  = 4\| \< x \> ^{-M-1+N}\| _{L^1(\R )} ,\end{align*}
with all the constants independent from $\varepsilon >0$. This completes the case $L=0$.

  Let us consider now the case with   $L\ge 1$.  From the proof in the case $L=0$, we have
\begin{align*}
 \< \im \varepsilon  \partial _x \> ^{-N}  [  \mathcal{ V} , \< \im \varepsilon   \partial _x \> ^{N} ]  v =\varepsilon \int _{\R  }K^{0}(x,y) v(y) dy \end{align*}
  where for $ \sigma \ge 0$  we set
  \begin{align}&
 K^{\sigma}(x,y) =    \int _{\R ^2  }
  e^{\im x k - \im y \ell }  \< \varepsilon  k\> ^{-N-\sigma } \widehat{\mathcal{V}} (k-\ell ) (k-\ell )   \frac{ P (\varepsilon  k,\varepsilon  \ell )}{\< \varepsilon  k \> ^{ N}+\< \varepsilon  \ell \> ^{ N}} dk d\ell .\label{eq2stestJ24II}
\end{align}
Notice that the integral is absolutely convergent for $\sigma >0$. Let us consider case  $\sigma>0$.
 When $|x|\ge 1$, integrating by parts we have, up to constants which we ignore,
 \begin{align*}&
 K^{\sigma}(x,y) = \frac{1}{x ^{ L}}    K_{1,0}^{\sigma }(x,y)    \text{  with }  K_{a,b}^{\sigma }(x,y) =\int _{\R  }
  e^{\im x k - \im y \ell }  \widetilde{H}_{a,b}(k,\ell )dk d\ell \text{  with }\\&
  \widetilde{H} _{a,b}^{\sigma }(k,\ell )= (\partial _{k} ^{ L}) ^a (\partial _{\ell} ^{ L} ) ^b \(      \< \varepsilon  k\> ^{-N-\sigma} \widehat{\mathcal{V}} (k-\ell ) (k-\ell )   \frac{ P (\varepsilon  k,\varepsilon   \ell )}{\< \varepsilon  k \> ^{ N}+\< \varepsilon  \ell \> ^{ N}} \)    ,
\end{align*}
 where   $\widetilde{H} _{a,b}^{\sigma }(k,\ell )$  for $a,b\in \{ 0,1\} $ has the properties of $ \widetilde{H} (k,\ell )$.
 Then,  for $   \chi _{0} \in C_c^\infty (\R, [0,1])   $  with $\chi _{0} =1$ near 0        and  $   \chi _{1}  = 1-\chi _{0}    $ and ignoring irrelevant constants,  we have
 \begin{align*}&
 K^{\sigma}(x,y)  =\sum _{a,b =0,1}   \chi _{a} (x)  \chi _{b}(y)   x^{- aL}y^{- bL} K_{a,b}^{\sigma}(x,y) .
\end{align*}
Then, for ${f} _{b} (y)=  \chi _{b}(y)y^{- bL}  f(y)$,
\begin{align*} &
  \left \| \<x \> ^{L}    \int  K^{\sigma}(x,y)   f(y) dy \right \| _{L^2(\R ) }  \le    \sum _{a,b =0,1}        \left \| \<x \> ^{L} x^{- aL}  \chi _{a} (x)   \int  K_{a,b}^{\sigma}(x,y)  \chi _{b}(y)y^{- bL}  f(y) dy \right \| _{L^2(\R ) } \\&
  \le \sum _{a,b =0,1}        \left \|     \int  K_{a,b}^{\sigma}(x,y)  \chi _{b}(y)y^{- bL}  f(y) dy \right \| _{L^2(\R ) }  \le \sum _{a,b =0,1}        \left \|     \int        \widetilde{H} _{a,b}^{\sigma }(k,\ell )    \widehat{f} _{b}(\ell ) d\ell  \right \| _{L^2(\R ) }
\\&\le C _L \sum _{ b =0,1} \|  {f} _{b}   \| _{L^2(\R ) } \lesssim C  _L \|  \<x \> ^{-L} f  \| _{L^2(\R ) },
\end{align*}
where the constants in the last line are uniform  on $\sigma$ by the argument used to prove \eqref{eq:young1}.
Since, furthermore, for a sequence $\sigma _n\to 0$  then   $ \int _{\R  }K^{\sigma _n}(x,y) f(y) dy      \xrightarrow{n\to  +\infty }     \int _{\R  }K^{0}(x,y) f(y) dy$ point--wise  for  $f\in C^0_c(\R )$,
we can assume, by the Fatou lemma and  by the density   of  $  C^0_c(\R )$  in $L^{2,-L}(\R )$,
\begin{align*}&
  \left \| \<x \> ^{L}    \int  K^{0}(x,y)   f(y) dy \right \| _{L^2(\R ) } \le    C _L  \|  \<x \> ^{-L} f \| _{L^2(\R ) } \text{ for all }f\in L^{2,-L}(\R ).
\end{align*}

 This yields  \eqref{eq2stestJ21II} and ends the proof of Lemma \ref{claim:l2boundII}.

 \qed

We will need in Appendix \ref{sec:comm} a variation  of Lemma \ref{claim:l2boundII}.

\begin{lemma}\label{claim:l2boundIII}  Suppose that the function $\mathcal{V} $ in Lemma \ref{claim:l2boundII}  has the additional property that for    $M\ge N+1$ its Fourier transform
satisfies
\begin{align}\label{eq2stestJ22III}& |      \widehat{{\mathcal{V}}} (k_1+ik_2) |\le  C_M \< k_1 \> ^{-M-1} \text{ for all $(k_1, k_2)\in \R \times [  \mathbf{b} ,  \mathbf{b} ]$    and}\\&  \widehat{{\mathcal{V}}}    \in C^0 ( \R \times [ -\mathbf{b} ,  \mathbf{b} ]) \cap H ( \R \times ( - \mathbf{b} ,  \mathbf{b} )), \nonumber
   \end{align}
  with $H(\Omega ) $ the set of holomorphic functions in an open subset $\Omega \subseteq \C$ and with a number  $\mathbf{b}>0$.
    Then
\begin{align}\label{eq2stestJ21II}
 \|     \<  \im \varepsilon  \partial _x \> ^{-N}  [   \mathcal{V} , \< \im \varepsilon    \partial _x \> ^{N} ]     e^{\mathbf{b} \<y \>} \| _{L ^{2  }(\R ) \to L ^{2 }(\R )} \le C_\mathbf{b} \varepsilon . \end{align}
  \end{lemma}
 \proof Formula \eqref{eq2stestJ24II} continues to hold for all  $\sigma\in (0,1]$, as a path integral
 \begin{align}&
 K^{\sigma}(x,y) =  e^{   y \ell _2 }    \int _{\R ^2 }
  e^{\im x k_1 - \im y \ell _1}  \< \varepsilon  k\> ^{-N-\sigma} \widehat{\mathcal{V}} (k-\ell ) (k-\ell )   \frac{ P (\varepsilon  k,\varepsilon  \ell )}{\< \varepsilon  k \> ^{ N}+\< \varepsilon  \ell \> ^{ N}} dk_1 d\ell _1 \label{eq2stestJ24III}
\end{align}
with $  k_2=0  $   and  $ |\ell_2|\le  \mathbf{b}  $.  Then, adjusting to the sign of   $y$, we conclude that
\begin{align*}&
 |K^{\sigma}(x,y)  |\le   e^{  - |y| \mathbf{b} } \left |   \int _{\R ^2 }
  e^{\im x k_1 - \im y \ell _1}  \< \varepsilon  k\> ^{-N-\sigma} \widehat{\mathcal{V}} (k-\ell ) (k-\ell )   \frac{ P (\varepsilon k,\varepsilon  \ell )}{\< \varepsilon  k \> ^{ N}+\< \varepsilon  \ell \> ^{ N}} dk_1 d\ell _1  \right | .
\end{align*}
This implies
\begin{align*}&
  \left \|      \int  K^{\sigma}(x,y) e^{    |y|\mathbf{ b} }  e^{  - |y| \mathbf{b}  }  f(y) dy \right \| _{L^2(\R ) }      \le C  _{\mathbf{b}} \|  e^{  - |x| \mathbf{b}  } f  \| _{L^2(\R ) }.
\end{align*}
Like in the proof of Lemma \ref{claim:l2boundII} we can take the limit $\sigma \to 0^+$  obtaining  \eqref{eq2stestJ21II}.
\qed

We next give estimates on the operator $\mathcal{T} $.

\begin{lemma}\label{lem:coer5} There exist constants $C_0$ and  $C_N$ such that for  $\varepsilon>$ small enough
we have
\begin{align}\label{eq:coer51}
\|\mathcal{T} \|_{L^2\to L^2}\le C_0 \varepsilon^{-N} \text{   and }  \|\mathcal{T}  \|_{\Sigma^N \to \Sigma^0}\le C_N.
\end{align}
Furthermore, let $K_{\varepsilon}(x,y) \in \mathcal{D}'(\R \times \R)$  be the Schwartz kernel of  $\mathcal{T} $. Then, we have
\begin{align}\label{eq:coer52}
| K_{\varepsilon}(x,y)    | \le C_0  e^{-\frac{|x-y|}{3\varepsilon}}  \text{  for all $x,y$ with $|x-y|\ge 1$.}
\end{align} \end{lemma}
\proof   First, for $\mathcal{T}_1= \<\im \partial_x\>^{-N} \mathcal{A}^*$, we have
\begin{align*}
\|\mathcal{T}  \|_{L^2\to L^2}\leq \|\< \im \varepsilon\partial_x\>^{-N}\<\im \partial_x\>^N\|_{L^2\to L^2}\|  \mathcal{T}_1     \|_{L^2\to L^2}.
\end{align*}
Since $\|\<\im \varepsilon\partial_x\>^{-N}\<\partial_x\>^N\|_{L^2\to L^2} = \|\<\varepsilon k\>^{-N}\<k\>^N\|_{L^\infty (\R )}  \lesssim \varepsilon^{-N}$ and $\|\mathcal{T}_1\|_{L^2\to L^2}\lesssim 1$ because $\mathcal{T}_1$ is a degree $0$ $\Psi$DO, we have the first inequality in \eqref{eq:coer51}.

\noindent It is enough to prove \eqref{eq:coer52}  for  operators $\<\im \varepsilon\partial_x\>^{-N} (\im \partial_x) ^m$ for $0\le m\le N$, which, up to irrelevant constant factors,  are convolutions, for $x\neq  0$ by the generalized integrals
\begin{align*}
 K_{\varepsilon}(x)= \int _{\R} e^{\im x k_1}    \frac{k_1^n}{ \(  1+ \varepsilon ^2 k_1^2\) ^{N/2}  }dk_1=   e^{- x k_2}     \int _{\R} e^{\im x k_1}    \frac{(k_1 +\im k_2)^n}{ \(  1 -\varepsilon ^2 k_2^2   + \varepsilon ^2 k_1^2 +2\varepsilon ^2\im k_1k_2\) ^{N/2}  }dk_1,
\end{align*}
which are well defined for $|k_2|\le 2^{-1} \varepsilon ^{-1}$. For $|x|\ge 1$  and  if $k_2=2^{-1} \varepsilon ^{-1} \text{sign}(x)     $,    it is elementary to show, by standard arguments with cutoffs and integration by parts, that   the above is $\lesssim  e^{-\frac{|x |}{3\varepsilon}}$, yielding \eqref{eq:coer52}.

\noindent We turn now to the second inequality in \eqref{eq:coer51}. We have, for functions $b_n(x)$ bounded with all their derivatives,
\begin{align*}
\|\mathcal{T} f \|_{  \Sigma^0}= \|  \widetilde{ \mathcal{T}}          \(  e^{a  \<x\>} \sum _{n=0}^{N} b_n  \partial_x ^{n}f  \)\|_{  L^2} \text{  where } \widetilde{ \mathcal{T}} :=       e^{a \<x\>}   \< \im \varepsilon\partial_x\>^{-N}       e^{-a  \<x\>}.
\end{align*}
For $ \widetilde{ \mathcal{T}} (x,y)$ the integral kernel of $\widetilde{ \mathcal{T}}$, we have, for $m(k) =  \<k\>^{-N}$ and for $m^{\vee}$  its inverse Fourier transform, using  \eqref{eq:coer52} we obtain
\begin{align*} &
 |\widetilde{ \mathcal{T}} (x,y)|= |\widetilde{ \mathcal{T}} (x,y)|\chi _{[0,1]}(|x-y|) +   |\widetilde{ \mathcal{T}} (x,y)|\chi _{[ 1,\infty )}(|x-y|)\\& \lesssim \varepsilon ^{-1} m^{\vee} \(  \frac{x-y}{\varepsilon} \) \chi _{[0,1]}(|x-y|) +  e^{-\frac{|x-y|}{3\varepsilon}}   e^{a \<x\>}  e^{-a  \<y\>} \\& \lesssim \varepsilon ^{-1} m^{\vee} \(  \frac{x-y}{\varepsilon} \) \chi _{[0,1]}(|x-y|) +  e^{-\frac{|x-y|}{4\varepsilon}}
\end{align*}
for $\varepsilon >0$ small enough.
This implies, from Young's inequality,    \cite[Theorem 0.3.1]{sogge}, that
\begin{align*}
\|\mathcal{T} f \|_{  \Sigma^0} \lesssim  \|    e^{a  \<x\>} \sum _{n=0}^{N}b_n  \partial_x ^{n}f   \|_{  L^2} \lesssim   \| f  \|_{\Sigma^N },
\end{align*}
yielding the   2nd inequality  in \eqref{eq:coer51}.
   \qed

The following technical estimates are related to analogous ones in Sect.\  4.4 \cite{KMM3}.
\begin{lemma}\label{lem:coer3}
We have \begin{align}&  \| w \| _{L^2(|x|\le 2B^2)}\lesssim B^2  \| w  \| _{\widetilde{\Sigma}}    \text{  for any $w$,} \label{eq:estimates111}
%\\&  \| w \| _{\rho}\lesssim    \| w' \| _{L^2(\R )} +    \|      \sqrt{\sech \( \frac{x}{2} \)} w  \| _{L^2(\R )}\label{eq:estimates112}
\\& \| \xi \| _{\widetilde{\Sigma}}^2\lesssim      \< (-\partial ^2_x- 2^{-2}  \chi _{B^2} ^{2}xV' _{N+1})\xi, \xi \> \lesssim   \| \xi \| _{\widetilde{\Sigma}}^2  \  \text{  for any $\xi$},  \label{eq:estimates113}\\&
\| v \| _{L^2(\R )} \lesssim \varepsilon ^{-N} B^2  \|  w      \| _{\widetilde{\Sigma}}  \ , \label{eq:estimates115}\\&  \| v ' \| _{L^2(\R )} \lesssim    \varepsilon ^{-N}\|  w     \| _{\widetilde{\Sigma}} \ ,  \label{eq:estimates116} \\& \|   \< x \> ^{-M}    v \|_{H^1(\R )}  \lesssim   \|      \xi   \|_{\widetilde{\Sigma}}    +\varepsilon ^{-N} \< B \> ^{-M+3}    \|    w       \| _{\widetilde{\Sigma}}  \text{ for $M\in \N$, $M\ge 4$}.\label{eq2stestJ21-4}
\end{align}
\end{lemma}
\proof     The proof of \eqref{eq:estimates111}--\eqref{eq:estimates113} is exactly the same in Lemma 4 \cite{KMM3} and is   a consequence of Lemma \ref{lem:equiv_rho}.  Now we consider, still following \cite{KMM3}, the proof of \eqref{eq:estimates115}, which is rather immediate. Indeed,  by  \eqref{eq:coer51}, \eqref{eq:estimates111} and $A\gg B^2$,  we have
 \begin{align}
&  \| v \| _{L^2(\R )}  \lesssim  \varepsilon ^{-N}\| \chi _{B^2}\widetilde{\eta}  \| _{L^2(\R )}  \lesssim \varepsilon ^{-N}\left \| \frac{\chi _{B^2}}{\zeta _A}w   \right \| _{L^2(\R )}\lesssim  \varepsilon ^{-N}B^2  \|  w       \| _{\widetilde{\Sigma}}.\label{eq:estimates1151}
\end{align}
More complicated is the proof of  \eqref{eq:estimates116}.  We have
\begin{align}
&   v' = \mathcal{T}\(  \chi _{B^2}\widetilde{\eta} \) ' + \< \im \varepsilon   \partial _x \> ^{-N} [\partial _x,    \mathcal{A}^* ] \chi _{B^2}\widetilde{\eta}  . \label{eq:estimates1161}
\end{align}
To bound the first term in the right hand side, we use the inequality
\begin{align*}
&    |(\chi _{B^2}\widetilde{\eta} )' |  = \left |\(\frac{\chi _{B^2}}{\zeta _A}w   \)' \right |\le  \left | \frac{\chi _{B^2}}{\zeta _A} w'    \right |  +   \left | \( \frac{\chi _{B^2}}{\zeta _A} \) ' w   \right | \lesssim  \left |  w'   \right | + B ^{-2}|w  | \chi _{|x|\le 2 B^2},
\end{align*}
where we used $A\gg B^2$. Then
\begin{align}
     \|  \mathcal{T}\(  \chi _{B^2}\widetilde{\eta} \) ' \| _{L^2(\R )} &\lesssim    \varepsilon ^{-N}   \|  \(  \chi _{B^2}\widetilde{\eta} \) ' \| _{L^2(\R )}   \lesssim    \varepsilon ^{-N}  \(  \| w '  \| _{L^2(\R )}   + B ^{-2}  \| w    \| _{L^2(|x|\le 2 B^2)}\)  \label{eq:estimates1162}\\&\lesssim   \varepsilon ^{-N}   \| w    \| _{ \widetilde{\Sigma} }   \nonumber
\end{align}
by \eqref{eq:estimates111}. To bound the second term in the right hand side in \eqref{eq:estimates1161}, we use formula
\begin{align*}
&   [\partial _x,   \mathcal{A}^* ] =   \sum _{j=1}^{N} \prod _{i=0} ^{N-1-j} A^{*}_{N-i} \( \log \psi _j \) ^{\prime\prime}  \prod _{i= 1}^{j-1} A^{*}_{j-i} .
\end{align*}
with  the convention $ \prod _{i=0} ^{l}B_i =B_0\circ ...\circ B_l$.
Then we have
\begin{align}\nonumber
&      \left \|\< \im \varepsilon   \partial _x \> ^{-N}\prod _{i=0} ^{N-1-j} A^{*}_{N-i} \( \log \psi _j \) ^{\prime\prime}  \prod _{i= 1}^{j-1} A^{*}_{j-i}   \chi _{B^2}\eta    \right \| _{L^2(\R )}   =  \varepsilon ^{\frac{3}{2}-N}   \left \|\<  \im  \partial _y \> ^{-N}\prod _{i=0} ^{N-1-j}    {\psi _{N-i} (\varepsilon y ) } \circ \right .  \\& \left . \partial_{y} \circ \frac{1}{\psi _{N-i} (\varepsilon y )}             \( \log \psi _j \) ^{\prime\prime} (\varepsilon  y  )   \prod _{i= 1}^{j-1}   {\psi _{j-i} (\varepsilon y ) } \circ  \partial_{y} \circ \frac{1}{\psi _{j-i} (\varepsilon y )}      \chi _{B^2} (\varepsilon y )\eta  (\varepsilon y )   \right \| _{L^2(\R )}  . \label{eq:estimates11620}
\end{align}
Now write the operator inside the last term as
\begin{align*}
&      \mathcal{P}  \<  \im  \partial _y \> ^{-j+1} \( \log \psi _j \) ^{\prime\prime} (\varepsilon  y  )   \prod _{i= 1}^{j-1}   {\psi _{j-i} (\varepsilon y ) } \circ  \partial_{y} \circ \frac{1}{\psi _{j-i} (\varepsilon y )}      \chi _{B^2} (\varepsilon y )  \text{ , where}\\
&      \mathcal{P}  := \<  \im  \partial _y \> ^{-N}\prod _{i=0} ^{N-1-j}    {\psi _{N-i} (\varepsilon y ) }     \partial_{y} \circ \frac{1}{\psi _{N-i} (\varepsilon y )}  \<  \im  \partial _y \> ^{j-1}  .  \end{align*}
 We have  $\|  \mathcal{P} \| _{L^2\to L^2} \lesssim 1 $. Then the term in  \eqref{eq:estimates11620} is
\begin{align} \nonumber
&
\lesssim \varepsilon ^{\frac{3}{2}-N}  \left \|               \( \log \psi _j \) ^{\prime\prime} (\varepsilon  y  )     \<    \im \partial _y \> ^{-j+1}  \prod _{i= 1}^{j-1}   {\psi _{j-i} (\varepsilon y ) } \circ  \partial_{y} \circ \frac{1}{\psi _{j-i} (\varepsilon y )}      \chi _{B^2} (\varepsilon y )\eta  (\varepsilon y )   \right \| _{L^2(\R )}  \\&   +\varepsilon ^{\frac{3}{2}-N}  \left \|  \<   \im  \partial _y \> ^{-j+1}      \left [       \( \log \psi _j \) ^{\prime\prime} (\varepsilon y  )  , \<    \partial _y \> ^{ j-1}  \right ] \right .  \label{eq:estimates116200}\\& \left .\<   \im \partial _y \> ^{-j+1}  \prod _{i= 1}^{j-1}   {\psi _{j-i} (\varepsilon y ) } \circ  \partial_{y} \circ \frac{1}{\psi _{j-i} (\varepsilon y )}      \chi _{B^2} (\varepsilon  y )\eta  (\varepsilon y )   \right \| _{L^2(\R )}. \nonumber
\end{align}
By Lemma \ref{lem:equiv_rho} and  $A\gg B^2$, the   term   in the first   line  is
\begin{align*}
& \lesssim     \varepsilon ^{\frac{3}{2}-N}\left \|       \< \varepsilon y  \> ^{-3}        \chi _B (\varepsilon y )\eta  (\varepsilon y )   \right \| _{L^2(\R )} =  \varepsilon ^{ 1-N}\left \|       \< x  \> ^{-3}        \frac{\chi _{B^2}}{\zeta _A} w    \right \| _{L^2(\R )}
  \lesssim  \varepsilon ^{1 -N} \left \|        w     \right \| _{\widetilde{\Sigma}}.
\end{align*}
 By Lemmas \ref{lem:equiv_rho}, \ref{lem:pd01}  and \ref{claim:l2boundII}, the term in the last two lines  of \eqref{eq:estimates116200} is
\begin{align*}
& \le  \varepsilon ^{j  -N}   \left \|  \<  \im \varepsilon \partial _x \> ^{-j+1}      \left [       \( \log \psi _j \) ^{\prime\prime}    , \<   \im \varepsilon \partial _x \> ^{ j-1}  \right ] \<   \im \varepsilon \partial _x \> ^{-j+1}  \prod _{i= 1}^{j-1}   {\psi _{j-i} } \circ  \partial_{x} \circ \frac{1}{\psi _{j-i}  }      \chi _{B^2}  \eta    \right \| _{L ^{2,3}(\R )} \\& \lesssim  \varepsilon ^{j+1- N} \left \|   \< x \> ^{-3}   \<   \im \varepsilon \partial _x \> ^{-j+1}  \prod _{i= 1}^{j-1}   {\psi _{j-i} } \circ  \partial_{x} \circ \frac{1}{\psi _{j-i}  }      \chi _{B^2}  \eta    \right \| _{L^2(\R )} \\& =     \varepsilon ^{  5/2-N}     \left \|   \< \varepsilon y \> ^{-3}   \<   \im \partial _y \> ^{-j+1}  \prod _{i= 1}^{j-1}   {\psi _{j-i} (\varepsilon y ) } \circ  \partial_{y} \circ \frac{1}{\psi _{j-i}  (\varepsilon y ) }      \chi _{B^2}  (    \varepsilon y ) \eta   (\varepsilon y )  \right \| _{L^2(\R )}\\& \lesssim     \varepsilon ^{ 5/2-N}     \left \|   \< \varepsilon y \> ^{-3}       \chi_{B^2}  (\varepsilon  y ) \eta   (\varepsilon y )  \right \| _{L^2(\R )} =  \varepsilon ^{  2-N}     \left \|   \< x \> ^{-3}       \chi _{B^2}   \eta      \right \| _{L^2(\R )} \lesssim \varepsilon ^{  2-N}    \| w   \| _{ \widetilde{\Sigma }} .
\end{align*}
This completes the proof of \eqref{eq:estimates116}.

\noindent We finally consider the proof of \eqref{eq2stestJ21-4}.  We have
    \begin{align}&
\|   \< x \> ^{-M}    v \|_{H^1(\R )}  \le \|   \< x \> ^{-M}  \chi_B   \xi   \|_{H^1(\R )}     +  \|   \< x \> ^{-M}  \( 1-\chi_B \)   v    \|_{H^1(\R )}\nonumber  \\& \lesssim   \|      \xi   \|_{\widetilde{\Sigma}}    + \< B \> ^{-M+3}    \| \< x \> ^{-3}v  \| _{H^1(\R )},\label{eq2stestJ21-}
  \end{align}
  where we used Lemma \ref{lem:equiv_rho}.
  To evaluate the last term, we observe, introducing $ x=\varepsilon y$,  that
\begin{align*}&
  \< x \> ^{-3} (\mathcal{T} f ) (x) = (-1)^N \< \varepsilon y \> ^{-3} \varepsilon ^{-N}   \<    \im \partial _y \> ^{-N} \prod _{i=0} ^{N-1 }   {\psi _{N-i} (\varepsilon  y ) }\circ  \partial_{y} \circ \frac{1}{\psi _{N-i} (\varepsilon y )} (f(\varepsilon  \cdot ) )(y)\\&  =\< \varepsilon y \> ^{-3} \varepsilon ^{-N} p ( \varepsilon y, \im \partial _y) (f(\varepsilon \cdot ) )(y)) .
  \end{align*}
Then, we can apply Lemma \ref{lem:pd01}, concluding that
\begin{align}&
 \| \< x \> ^{-3}\mathcal{T} f \| _{L^2_x} =  \varepsilon ^{\frac{1}{2}-N}  \| \< \varepsilon y \> ^{-3}  p ( \varepsilon y, \im \partial _y) (f(\varepsilon \cdot ) )(y)) \| _{L^2_y} \nonumber \\&  \lesssim \varepsilon ^{\frac{1}{2}-N}  \| \< \varepsilon y \> ^{-3}  f(\varepsilon y ) \| _{L^2_y} \lesssim \varepsilon ^{ -N}  \| \< x \> ^{-3}  f  \| _{L^2_x}.\label{eq2stestJ21-234}
  \end{align}
By $v=  \mathcal{T} \chi _{B^2}\widetilde{\eta} $,  Lemma \ref{lem:equiv_rho} and $A\gg B^2$, this by implies
 \begin{align}& \label{eq2stestJ21-2}
 B ^{-M+3}    \| \< x \> ^{-3}\mathcal{T} \chi _{B^2}\widetilde{\eta} \| _{L^2(\R )} \lesssim  B ^{-M+3}  \varepsilon ^{-N}  \left \| \< x \> ^{-3}  \frac{\chi _{B^2}}{\zeta _A}\zeta _A\widetilde{\eta} \right \| _{L^2(\R )}\\& \lesssim B ^{-M+3}  \varepsilon ^{-N}  \| \< x \> ^{-3} w  \| _{L^2(\R )}\lesssim B ^{-M+3}  \varepsilon ^{-N}  \|  w   \| _{\widetilde{\Sigma}}.\nonumber
  \end{align}
Next,
 \begin{align}&
 \(   \< x \> ^{-3}\mathcal{T} \(     \chi _{B^2}\widetilde{\eta}    \)   \) '=  \(   \< x \> ^{-3}  \) ' \mathcal{T} \(     \chi _{B^2}\widetilde{\eta }   \)  + [\partial _x, \mathcal{T} ] \(     \chi _{B^2}\widetilde{\eta }   \) +\< x \> ^{-3}\mathcal{T} \(     \chi _{B^2}\widetilde{\eta  }  \)'.\label{eq2stestJ21-233}
  \end{align}
  Since $\(   \< x \> ^{-3}  \) '\sim \< x \> ^{-4} $, the first term in the right can be treated like  \eqref{eq2stestJ21-234}.  The second term in the right hand side of \eqref{eq2stestJ21-233} coincides with the second term in the right in \eqref{eq:estimates1161}.
  So we conclude, using Lemma \ref{lem:equiv_rho} and \eqref{eq:estimates111},
 \begin{align}&
 \< B \> ^{-M+3}    \| \(   \< x \> ^{-3}\mathcal{T} \chi_{B^2}\widetilde{\eta}    \) ' \| _{L^2(\R )}\lesssim  \< B \> ^{-M+3} \varepsilon ^{-N} \(  \|   w   \| _{\widetilde{\Sigma}} +  \left  \| \< x \> ^{-3} \(   \frac{\chi _{B^2}}{\zeta _A} w  \) ' \right \| _{L^2(\R )} \) \nonumber\\& \lesssim  \< B \> ^{-M+3}  \varepsilon ^{-N} \(  \|   w \| _{\widetilde{\Sigma}} +   \|  \< x \> ^{-3} w     '\| _{L^2(\R )}  +\left  \|  \< x \> ^{-3}\(   \frac{\chi _{B^2}}{\zeta _A}   \) 'w    \right  \| _{L^2(|x|\le 2B^2  )}  \) \nonumber\\& \lesssim   \< B \> ^{-M+3}  \varepsilon ^{-N}\|  w  \| _{\widetilde{\Sigma}}.\label{eq2stestJ21-3}
  \end{align}
Entering \eqref{eq2stestJ21-2}--\eqref{eq2stestJ21-3}   in \eqref{eq2stestJ21-},  we obtain   \eqref{eq2stestJ21-4}.

\qed

\section{Proof of Proposition \ref{prop:2ndvirial}} \label{sec:vir2}

For the proof of the 2nd virial estimate Proposition \ref{prop:2ndvirial}, we use the following functional,
\begin{align}\label{def:JBg}
\mathcal{J}   (v ):=\frac{1}{2}\<v  ,\im  \widetilde{{S}}_B v \>,
\end{align}
where the anti-symmetric operator $ \widetilde{{S}}_B$ is defined by
\begin{align}\label{def:tildeSpsi}
 \widetilde{{S}}_B:=\frac{\psi_B'}{2}+\psi_B\partial_x,\ \psi_B(x):=\chi_{B^{2}}^2 (x)\varphi_{B}(x).
\end{align}
Then, by the equation of $v$ in  \eqref{eq:vBg}, we have
\begin{align}\label{eq:diffJ1}
  \frac{d}{dt}\mathcal{J}    (v )=-\<H_{N+1}v  , \widetilde{S}_B v\>-\<\mathcal{R}_{v   }, \widetilde{S}_B v\>     -\sum_{\mathbf{m}\in \mathbf{R}_{\mathrm{min}}} \<\mathbf{z}^{\mathbf{m}}\widetilde{G}_{ \mathbf{m}}, \widetilde{S}_B v\>
\end{align}
where by     remainder   formulas
\eqref{eq:Reta}   and \eqref{def:RvBg},
we have
\begin{align}\label{eq:remaindJB}
\<\mathcal{R}_{v  },\widetilde{S}_B v\>=\sum_{j=1}^{7}\<\mathcal{R} _{v  j},\widetilde{S}_Bv\>,
\end{align}
with
\begin{align}
\mathcal{R} _{v  1}&=-\im \mathcal{T} \chi_{B^2} P_c D_{\mathbf{z}}\phi[\mathbf{z}]\(\dot {\mathbf{z}}+\im \boldsymbol{\varpi}(\mathbf{z})\mathbf{z}\),\label{eq:remaindJB1}\\
\mathcal{R}_{v  2}&=\mathcal{T} \chi_{B^2}P_c\mathcal{R}_{\mathrm{rp}}[\mathbf{z}],     \label{eq:remaindJB2}      \\
\mathcal{R}_{v  3}&=\mathcal{T}\chi_{B^2} P_c\(2\phi[\mathbf{z}]|\eta|^2 + \overline{\phi[\mathbf{z}]}\eta^2\),\label{eq:remaindJB3}\\
\mathcal{R}_{v  4}&=\mathcal{T}\chi_{B^2} P_c |\eta|^2\eta,     \label{eq:remaindJB4}\\
\mathcal{R}_{v  5}&=\mathcal{T} \chi_{B^2}P_cL [\mathbf{z}] {\eta} ,\label{eq:remaindJB5}\\
\mathcal{R}_{v  6}&=\<\varepsilon \im \partial_x\>^{-N}[V_{N+1},\<\varepsilon \im \partial_x\>^N]v  ,\label{eq:remaindJB6}\\
\mathcal{R}_{v  7}&=\mathcal{T} \(2\chi_{B^2}'\partial_x + \chi_{B^2}''\)\tilde{\eta}. \label{eq:remaindJB7}
\end{align}

Proposition \ref{prop:2ndvirial} follows from the following three lemmas.
\begin{lemma}\label{lem:2ndv1v}
We have
\begin{align}\< H _{N+1} v ,  \widetilde{S}_Bv \>  \ge    2^{-1} \<  - \xi ^{\prime\prime} -  \frac{1}{2}  \chi _{B^2} ^{2}x V_{N+1}  '   \xi ,\xi     \>  + B ^{-1/2}  O\(  \|\xi \|_{\widetilde{\Sigma}} ^2+  \|  w   \|_{\widetilde{\Sigma}} ^2   \) .\label{eq:2ndvmainv}
\end{align}

\end{lemma}

\begin{proof}
Like in Lemma \ref{lem:1stv1}  the l.h.s.\  of \eqref{eq:2ndvmainv} equals
\begin{align}\label{eq:lem2ndv1_1v}
&
  \< H _{N+1} v ,  \widetilde{S}_Bv \>  = \<  \psi '_B  v' , v'   \> - \frac{1}{4}\<  \psi  ^{\prime\prime\prime}_B  v  , v  \>  -\frac{1}{2} \<    v ,   \psi _B     V_{N+1}  '      v \> .
%   \\& = -\<  \chi _{B^2} ^2\zeta _B ^2   g_{B,\gamma} ' , g_{B,\gamma} ' \> + \<     \xi  ,  \frac{ \psi _B }{2\zeta _B ^2}    V_{N+1}  '     \xi \>   + \frac{1}{4}\<  \chi _{B^2} ^2 (\zeta _B ^2) ^{\prime\prime}  g_{B,\gamma}   , g_{B,\gamma}    \> \\&  -\<  (\chi _{B^2} ^2)'\varphi  _B     g_{B,\gamma} ' , g_{B,\gamma} ' \> + \frac{1}{4}\<  \( 3(\chi _{B^2} ^2)'(\zeta _B ^2) ^{\prime }+ 3(\chi _{B^2} ^2)^{\prime\prime} \zeta _B ^2  - (\chi _{B^2} ^2)^{\prime\prime\prime}\varphi _B  \)  g_{B,\gamma}   , g_{B,\gamma}    \> ,
\end{align}
For the 1st and 2nd     term in the right hand side of \eqref{eq:lem2ndv1_1v}, we have
\begin{align}
&\<  \psi '_B  v' , v'  \> - \frac{1}{4}\<  \psi  ^{\prime\prime\prime}_B v  , v    \>=
\<\chi_{B^{2}}^2 \zeta_B^2 v',v'\>-\frac{1}{4}\<\chi_{B^{2}}^2\(\zeta_B^2\)'' v,v\>\nonumber\\&
+\<\(\chi_{B^{2}}^2\)'\varphi_B v', v'\>-
\frac{1}{4}\<\((\chi_{B^{2}}^2)'''\varphi_{B}+3\(\chi_{B^{2}}\)''\zeta_B^2+3(\chi_{B^{2}})'(\zeta_B^2)'\) v, v \>.\label{eq:lem2ndv1_2v}
%\<-\partial_x^2\xi_{B,\gamma}\xi_{B,\gamma}\>+\frac{1}{B}\<V_0\xi_{B,\gamma},\xi_{B,\gamma}\>+\mathcal{R}_{J,1},
\end{align}
By  Lemma \ref{lem:coer3}, the last term of \eqref{eq:lem2ndv1_2v} can by bounded as
\begin{align*}
&|\<\((\chi_{B^{2}}^2)'''\varphi_{B}+3\(\chi_{B^{2}}\)''\zeta_B^2+3(\chi_{B^{2}})'(\zeta_B^2)'\) v, v \>|\\&\lesssim \(B^{-5}+B^{-4}e^{-2B }+B^{-4}e^{-B }\)B^{4}\varepsilon ^{-2N} \| w\|_{\widetilde{\Sigma}} ^2 \lesssim
B^{-1/2} \| w \|_{\widetilde{\Sigma}} ^2.
\end{align*}
For the 1st term of the 2nd line of \eqref{eq:lem2ndv1_2v},  we have
\begin{align*} &
   \<   \left |  (\chi  _{B^2} ^2)'\varphi  _B \right |    , |v'|^2      \>   \lesssim   B ^{-1}\| v '  \| _{L^2(|x|\le 2 B^2)}^2\lesssim \varepsilon ^{-2N} B ^{-1}  \|  w     \| _{\widetilde{\Sigma}} ^2.
\end{align*}
 We  consider now  the 1st and 2nd term of \eqref{eq:lem2ndv1_2v}.
Using  $\chi_{B^2}\zeta_B v' =\partial_x \xi  -\chi_{B^2}\zeta_B' v-\chi_{B^2}' \zeta_B v$, we have
\begin{align}
&\<\chi_{B^2}^2 \zeta_B^2 v',v' \>-\frac{1}{4}\<\chi_{B^2}^2\(\zeta_B^2\)'' v,v\>=
\<- \xi  ^{\prime\prime},\xi \>+\frac{1}{2B}\<V_0\xi ,\xi   \>\nonumber\\&
+2\<  \xi ',\chi_{B^{2}}'\zeta_B v\>-2\<\zeta_B'\xi ,\chi_{B^{2}}'v\>+\<\chi_{B^{2}}'\zeta_B v ,\chi_{B^{2}}' \zeta_B v\>.
\label{eq:lem2ndv1_6v}
\end{align}
The 2nd line of \eqref{eq:lem2ndv1_6v} can be bounded as
\begin{align}\nonumber
&|\<  \xi  ',\chi_{B^{2}}'\zeta_B v\>|\lesssim B^{-2}B^{4}e^{-B }\|\xi \|_{\widetilde{\Sigma}}\|w\|_{\widetilde{\Sigma}}\lesssim  B^{-1}\|\xi \|_{\widetilde{\Sigma}}\|w  \|_{\widetilde{\Sigma}},\\&
|\<\zeta_B'\xi  ,\chi_{B^{2}}' v\>|\lesssim B^{-2}B^{8}e^{-B }  \|\<x\>^{-3}\xi \|_{L^2} \varepsilon ^{-N}\|w  \|_{\widetilde{\Sigma}}\lesssim  B^{-1}\|\xi \|_{\widetilde{\Sigma}}\|w\|_{\widetilde{\Sigma}},\label{eq:lem2ndv1_60v}\\&
|\<\chi_{B^{2}}'\zeta_B v,\chi_{B^{2}}' \zeta_B v\>|\lesssim B^{-2}e^{-B} B^4 \|  w   \|_{\widetilde{\Sigma}} ^2   \lesssim B^{-1}  \|  w   \|_{\widetilde{\Sigma}} ^2 ,\nonumber
\end{align}
where we have used $\|\zeta_B' \xi  \|_{L^2}\lesssim B \|\xi \|_{\widetilde{\Sigma}}$  by \eqref{eq:lem:rhoequiv}.

\noindent Summing up, we obtain
\begin{align*}
&
  \< H _{N+1} v ,  \widetilde{S}_Bv \>  =    \<  - \xi  ^{\prime\prime} + \frac{1}{2B} V_0\xi    -2 ^{-1}    \psi _B \zeta_B ^{-2}    V_{N+1}  '  \xi     ,\xi     \>  + B ^{-1/2}  O\(  \|\xi \|_{\widetilde{\Sigma}} ^2+  \|  w   \|_{\widetilde{\Sigma}} ^2   \) \\& \ge    \< (-\partial ^2_x -2 ^{-1}\chi _{B^2} ^{2}xV' _{N+1})\xi     ,\xi     \>  +  \frac{1}{2B}  \< V_0\xi ,\xi     \>  + B ^{-1/2}  O\(  \|\xi \|_{\widetilde{\Sigma}} ^2+  \|  w   \|_{\widetilde{\Sigma}} ^2   \) ,
\end{align*}
where , like in Lemma 3 in \cite{KMM3}, since      $  \frac{ \varphi _B (x) }{\zeta _B ^2(x)} \ge x$ for $x\ge 0$, for $B$ large enough, by Lemma \ref{lem:equiv_rho},
\begin{align*}
 \< (-\partial ^2_x -2 ^{-1}\chi _{B^2} ^{2}xV' _{N+1})\xi     ,\xi     \> \ge B ^{-1}  \< V_0\xi ,\xi     \>  .
\end{align*}
 So
\begin{align*}
& \<  - \xi  ^{\prime\prime} + \frac{1}{2B} V_0\xi    -2 ^{-1}    \psi _B \zeta_B ^{-2}    V_{N+1}  '  \xi     ,\xi     \> \ge 2^{-1} \<  -\xi  ^{\prime\prime} -  2^{-1}  \chi _{B^2} ^{2}x V_{N+1}  '   \xi     ,\xi     \> .
\end{align*}

\end{proof}

\begin{lemma}\label{lem:2ndv2}
We have
\begin{align}\label{eq:est_R_v1_2}
  \sum _{j=1,2}|\<\mathcal{R} _{v  j},\widetilde{S}_B v \>| & \lesssim    \varepsilon ^{-N}B  \delta ^2   \(   \| w  \| _{\widetilde{\Sigma}}  +  \| \xi    \| _{\widetilde{\Sigma}} \)
\(  \sum_{\mathbf{m}\in \mathbf{R}_{\mathrm{min}}}|\mathbf{z}^{\mathbf{m}}|+ \|\dot {\mathbf{z}}+\im \boldsymbol{\varpi}(\mathbf{z})\mathbf{z}\|\)  , \\ \label{eq:est_R_v3_5}\sum _{j=3,4,5}
 |\<\mathcal{R} _{v  j},\widetilde{S}_B v  \>| &\lesssim    \varepsilon ^{-N}B ^3  \delta   ^2  \| w  \| _{\widetilde{\Sigma}}    \(   \| w  \| _{\widetilde{\Sigma}}  +  \| \xi    \| _{\widetilde{\Sigma}} \); \\ \label{eq:est_R_v6}
 |\<\mathcal{R} _{v    6},\widetilde{S}_B v  \>| &\lesssim \varepsilon \(  \|      \xi   \|_{\widetilde{\Sigma}}  ^2  +     \|    w       \| _{\widetilde{\Sigma}}  ^2    \) ; \\ |\<\mathcal{R} _{v  7},\widetilde{S}_B v \>| &\lesssim    \varepsilon ^{-N} B ^{-1}    \|  w      \| _{\widetilde{\Sigma}} \(   \|w  \| _{\widetilde{\Sigma}}  +  \| \xi    \| _{\widetilde{\Sigma}} \).\label{eq:est_R_v7}
\end{align}

\end{lemma}

\begin{proof} First we claim  \begin{align}
 \|  \widetilde{S}_Bv  \| _{L^2}\lesssim    \varepsilon ^{-N}B\| w  \| _{\widetilde{\Sigma}}  + B\| \xi    \| _{\widetilde{\Sigma}} . \label{eq:sbv}
\end{align}
The proof of \eqref{eq:sbv} is like in \cite{KMM3}. By  \eqref{eq:estimates116} and $\|  \psi _{B}   \|  _{L^\infty} \lesssim B$ we have
\begin{align*}
 \|   \widetilde{S}_Bv  \| _{L^2}\lesssim \| \psi _{B}' v   \|  _{L^2} + \| \psi _{B}  v '  \|  _{L^2} \lesssim \| \psi _{B}' v  \|  _{L^2} +  \varepsilon ^{-N}B\| w    \| _{\widetilde{\Sigma}}.
\end{align*}
Next, we have
\begin{align}&
| \psi ' _B| = |2\chi ' _{B^2} \chi _{B^2} \varphi _B + \chi _{B^2}^2 \zeta ^2 _B|\lesssim B ^{-1}\chi _{B^2}+  \chi _{B^2}^2 \zeta ^2 _B.\label{eq2stestJ33}
\end{align}
Then
\begin{align*}
 B ^{-1} \|   v   \|  _{L^2}\lesssim B  \varepsilon ^{-N}B\| w    \| _{\widetilde{\Sigma}}
\end{align*}
by \eqref{eq:estimates115}, by Lemma \ref{lem:equiv_rho} we have
\begin{align*}
   \|   \chi _{B^2}^2 \zeta ^2 _B v   \|  _{L^2} =    \|   \chi _{B^2}  \zeta  _B \xi   \|  _{L^2}       \lesssim
    \sqrt{\| \<   x \> \chi _{B^2}  \zeta  _B \| _{L^1} }\| \xi    \| _{\widetilde{\Sigma}} \sim B \| \xi    \| _{\widetilde{\Sigma}}
\end{align*}
and, finally,
\begin{align*}
   \|   \chi _{B^2}^2 \varphi _B   v  '  \|  _{L^2} \lesssim B   \|      v  '  \|  _{L^2}    \lesssim  B \varepsilon ^{-N} \| w  \| _{\widetilde{\Sigma}}
\end{align*}
by \eqref{eq:estimates116}, so that
so that  we get \eqref{eq:sbv}.

\noindent We have $\| P_cD_{\mathbf{z}}\phi[\mathbf{z}] \| _{L^2}=O(\| \mathbf{z} \| ^2 )$ by Proposition \ref{prop:rp}.  Then, using \eqref{eq:estimates115}--\eqref{eq:estimates116}  and $\|\psi _B \| _{L^\infty}\lesssim B$, we have
\begin{align*}&
\sum _{j=1,2}|\<\mathcal{R} _{v  j},\widetilde{S}_B v   \>|\lesssim  \sum _{j=1,2}\|     \mathcal{R} _{v  j}\|_{L^2}\|       \widetilde{S}_B v   \|_{L^2}  \lesssim  \sum _{j=1,2}\|     \mathcal{R} _{v  j}\|_{L^2} \(  \varepsilon ^{-N}B\| w  \| _{\widetilde{\Sigma}}  + B\| \xi    \| _{\widetilde{\Sigma}} \) \\& \lesssim  \varepsilon ^{-N}B  \delta  ^2   \(   \| w  \| _{\widetilde{\Sigma}}  +  \| \xi   \| _{\widetilde{\Sigma}} \)
\(  \sum_{\mathbf{m}\in \mathbf{R}_{\mathrm{min}}}|\mathbf{z}^{\mathbf{m}}|+ \|\dot {\mathbf{z}}+\im \boldsymbol{\varpi}(\mathbf{z})\mathbf{z}\|\) .
\end{align*}
We claim
\begin{align} &\sum _{j=3,4,5}
 |\<\mathcal{R} _{v  j},\widetilde{S}_B v \>|\lesssim       \sum _{j=3,4,5}\|     \mathcal{R} _{v  j}\|_{L^2} \(  \varepsilon ^{-N}B\| w   \| _{\widetilde{\Sigma}}  + B\| \xi    \| _{\widetilde{\Sigma}} \) \nonumber  \\&  \lesssim    \varepsilon ^{-N}B
 \(  \| \mathbf{z} \|   \|  {\eta}  \| _{H^1}   + B^2 \|  {\eta}  \| _{H^1} ^2 + \| \mathbf{z} \| ^2   \)  \| w    \| _{\widetilde{\Sigma}}    \(   \| w   \| _{\widetilde{\Sigma}}  +  \| \xi    \| _{\widetilde{\Sigma}} \) . \label{eq:est_R_v3_5b}
\end{align}
We have for example, using Lemma \ref{lem:GNTR} and  inequality \eqref{eq:estimates111},
\begin{align*}& \|  \mathcal{T}\chi_{B^2} P_c |\eta|^2\eta    \|_{L^2} \lesssim \varepsilon ^{-N} \( \|   P_c |\eta|^2 \( R[\mathbf{z}]  -1\) \widetilde{\eta}   \|_{L^2} +\|  \chi _{B^2} P_c |\eta|^2  \widetilde{\eta}   \|_{L^2}           \)  \\& \lesssim \varepsilon ^{-N}
\(      \| \mathbf{z} \| ^2 \| \eta \| _{L^\infty} ^2  \| \widetilde{\eta} \| _{L^2 _{-a}} + \|    P_d |\eta|^2  \widetilde{\eta}   \|_{L^2} +\|  \chi _{B^2}   |\eta|^2  \widetilde{\eta}   \|_{L^2}
\) \\& \lesssim \varepsilon ^{-N}
\(    \delta ^2  \| \mathbf{z} \| ^2 \| w \| _{\widetilde{\Sigma}} + \| \eta \| _{L^\infty} ^2  \| \widetilde{\eta} \| _{L^2 _{-a}} +\|  \chi _{B^2}   |\eta|^2  w   \|_{L^2}\) \\& \lesssim   \varepsilon ^{-N}
\(     \(   \| \mathbf{z} \| ^2 + \| \eta \| _{H^1} ^2 \) \| w \| _{\widetilde{\Sigma}}   + \| \eta \| _{H^1} ^2 \|    w   \|_{L^2(|x|\le 2 B^2)}\)
\lesssim   \varepsilon ^{-N}  \(   \| \mathbf{z} \| ^2 + B^2\| \eta \| _{H^1} ^2 \) \| w \| _{\widetilde{\Sigma}},
\end{align*}
with better the bounds for  the other terms in the r.h.s. of \eqref {eq:est_R_v3_5b}.

%Before proving \eqref{eq:est_R_v6} we record that
Using  Lemma \ref{claim:l2boundII}, \eqref{eq2stestJ21-4} and  \eqref{eq2stestJ33} we obtain \eqref{eq:est_R_v6}:
\begin{align*} &
 |\<\mathcal{R} _{v\eta 6},\widetilde{S}_B v  \>|\lesssim \\& \( \| \psi ' _B \| _{L^\infty}\| \< x \>  ^{-10} v  \| _{L^2} + \| \< x \>  ^{-10} \psi   _B \| _{L^\infty} \| \< x \>  ^{-10} v  '\| _{L^2} \)    \|  \< x \>  ^{20} \< \im \varepsilon\partial_x\>^{-N}[V_{N+1},\< \im \varepsilon \partial_x\>^N]v  \| _{L^2}\\& \lesssim \varepsilon \| \< x \>  ^{-10} v   \| _{H^1}^2\lesssim \varepsilon \(  \|      \xi   \|_{\widetilde{\Sigma}}  ^2  +     \|    w       \| _{\widetilde{\Sigma}}  ^2    \).
\end{align*}
Finally, the proof of \eqref{eq:est_R_v7} is the same as in \cite{KMM3}. We write
\begin{align} &
 |\<\mathcal{R} _{v  7},\widetilde{S}_B v \>|\lesssim   \varepsilon ^{-N}    \(  \| \chi_{B^2}' \tilde{\eta} '  \| _{L^2}  +  \| \chi_{B^2} ^{\prime\prime} \tilde{\eta}   \| _{L^2}\)      \( \|    \psi  _B '  v   \| _{L^2} + \|     \psi  _B  v '  \| _{L^2} \) .\label{eq:RrSBv1}
\end{align}
We claim
\begin{align} &
      \| \chi_{B^2}' \widetilde{{\eta}} '  \| _{L^2}  +  \| \chi_{B^2} ^{\prime\prime}  \widetilde{{\eta}}   \| _{L^2}  \lesssim B ^{-2}   \| w    \| _{\widetilde{\Sigma}} .\label{eq:RrSBv2}
\end{align}
Indeed  from $w =\zeta _A\widetilde{\eta} $ we have
\begin{align*}&
 w'  = \zeta _A' \widetilde{\eta} +  \zeta _A \widetilde{ \eta} ',
\end{align*}
so,  for $|x|\le A$,
\begin{align*}&
  |\eta '|\lesssim A^{-1}|\eta|+ |w' | =A^{-1}\zeta_A^{-1} |w |+ |w' |   .
\end{align*} By $A\gg B^2$ and \eqref{eq:estimates111}, we have
\begin{align*}&
  \|  \chi _B'  \widetilde{\eta}' \| _{L^2} \lesssim B ^{-2} \|    \widetilde{\eta}' \| _{L^2(B^2\le |x|\le 2B^2)}  \\&\lesssim B ^{-2} \(   \|    w' \| _{L^2(\R )}      + B^{-2} \|     w  \| _{L^2(B^2\le |x|\le 2B^2)} \) \lesssim  B ^{-2} \|    w  \| _{\widetilde{\Sigma}}
\end{align*}  and the following
\begin{align*}&
  \|  \chi _{B^2} ^{\prime\prime}  \widetilde{\eta} \| _{L^2} \lesssim B ^{-4} \|    \widetilde{\eta } \| _{L^2(B^2\le |x|\le 2B^2)} \lesssim B ^{-4}  \|    w \| _{L^2(  |x|\le 2B^2)} \lesssim B ^{-2}  \|    w  \| _{\widetilde{\Sigma}}.
\end{align*}
% We also have the following, which completes the proof of \eqref{eq:RrSBv2},\begin{align*} &  \| \chi_{B^2}' P_d {\eta}    \| _{L^2}  +  \| \chi_{B^2} ^{\prime\prime}  P_d{\eta}   \| _{L^2}  \lesssim e ^{-B}   \| w  \| _{\widetilde{\Sigma}} .\end{align*}
The next step is to prove the following, which with  \eqref{eq:RrSBv2}    yields \eqref{eq:est_R_v7},
\begin{align}\label{eq2stestJ31}
&  \|    \psi  _B '  v  \| _{L^2} + \|     \psi  _B  v '  \| _{L^2}  \lesssim B \varepsilon ^{-N}  \|    w   \| _{\widetilde{\Sigma}} + B  \|     \xi    \| _{\widetilde{\Sigma}} .
\end{align}
%  We have
%\begin{align}&
%| \psi ' _B| = |2\chi ' _{B^2} \chi _{B^2} \varphi _B + \chi _{B^2}^2 \zeta ^2 _B|\lesssim B ^{-1}+  \chi _{B^2}^2 \zeta ^2 _B,\label{eq2stestJ33}
%\end{align}
From $\xi  = \chi _{B^2}^2 \zeta   _Bv$,       we have by \eqref{eq:estimates115}, \eqref{eq:lem:rhoequiv} and \eqref{eq2stestJ33},
\begin{align*}&
\| \psi  _B '  v    \| _{L^2} \lesssim   B ^{-1} \|    v    \| _{L^2} +   \|     \zeta  _B  \xi   \| _{L^2} \lesssim  B \varepsilon ^{-N}  \|    w   \| _{\widetilde{\Sigma}} + B  \|     \xi   \| _{\widetilde{\Sigma}}.
\end{align*}
Using \eqref{eq:estimates116} and $|\psi  _B |\lesssim B$, we get the following, which completes the proof of  \eqref{eq2stestJ31},
\begin{align*}&
\| \psi  _B   v'    \| _{L^2} \lesssim   B   \|    v ' \| _{L^2}  \lesssim  \varepsilon ^{-N}B\| w    \| _{\widetilde{\Sigma}}   .
\end{align*}

\end{proof}

\begin{lemma}\label{lem:Gv1}
We have
\begin{align}\label{eq:Gv11}
\left| \<\mathbf{z}^{\mathbf{m}}  \widetilde{G}_{\mathbf{m}},\widetilde{S}_B v \>\right|\lesssim |\mathbf{z}^{\mathbf{m}}|    \(     \| \xi    \| _{\widetilde{\Sigma}}    + e^{-B/2} \| w   \| _{\widetilde{\Sigma}}         \).
\end{align}
\end{lemma}

\begin{proof}
We have
\begin{align} & \left| \<\mathbf{z}^{\mathbf{m}}  \widetilde{G}_{\mathbf{m}},\psi '_B v   + 2\psi  _B v '  \>\right| \nonumber  \\& \lesssim  \left| \<\mathbf{z}^{\mathbf{m}}  \widetilde{G}_{\mathbf{m}}, \(\chi  ^{2} _{B^2} \) ' \varphi _B v   \>\right| +\left| \<\mathbf{z}^{\mathbf{m}}  \widetilde{G}_{\mathbf{m}},  \chi  ^{2} _{B^2}  \zeta _B ^{2}v   \>\right| + \left| \<\mathbf{z}^{\mathbf{m}}  \widetilde{G}_{\mathbf{m}}, \psi  _B v '  \>\right| .\label{eq:Gv1--1}
      \end{align}
      We now we examine the three terms in line \eqref{eq:Gv1--1}.  Using \eqref{eq:estimates115},  $|\varphi _{B}|\le B$, $1_{|x|\leq 1}\leq \chi \leq 1_{|x|\leq 2}$
      and $\chi _{B^2}:=\chi (B ^{-2}\cdot )$,
       we get   \begin{align*}      \left| \<\mathbf{z}^{\mathbf{m}} \widetilde{G}_{\mathbf{m}}, \(\chi  ^{2} _{B^2} \) ' \varphi _B v   \>\right| & \lesssim  B ^{-1}   |\mathbf{z}^{\mathbf{m}}   \|   \widetilde{G}_{\mathbf{m}} \| _{L^2( B^2\le |x|\le 2B^2)}          \| v   \| _{L^2}\\& \lesssim
   \varepsilon ^{-N}   B   |\mathbf{z}^{\mathbf{m}}|     \|   \widetilde{G}_{\mathbf{m}} \| _{L^2( B^2\le |x|\le 2B^2)}  \|  w       \| _{\widetilde{\Sigma}}
      \end{align*}
  Now we claim $ \|   \widetilde{G}_{\mathbf{m}} \| _{L^2( B^2\le |x|\le 2B^2)} \le e ^{-B}$,  so that $ \varepsilon ^{-N}   Be ^{-B} \le e^{-B/2}  $.  To prove our claim
  we   split
  \begin{align*}   &    \|   \widetilde{G}_{\mathbf{m}} \| _{L^2( B^2\le |x|\le 2B^2)}  \le  \|  \mathcal{T} 1_{|x|\le B^2/2}   \chi _{B^2} P_c  {G}_{\mathbf{m}} \| _{L^2( B^2\le |x|\le 2B^2)}  + \|  \mathcal{T} 1_{|x|\ge B^2/2}   \chi _{B^2} P_c  {G}_{\mathbf{m}} \| _{L^2 } .
      \end{align*}
  Using \eqref{eq:coer52} we have
 \begin{align*}   &    \|  \mathcal{T} 1_{|x|\le B^2/2}   \chi _{B^2} P_c  {G}_{\mathbf{m}} \| _{L^2( B^2\le |x|\le 2B^2)}  \le   e ^{-3B} \|    {G}_{\mathbf{m}} \| _{L^2 } \le   e ^{-2B}
      \end{align*}
     while
   \begin{align*}   &    \|  \mathcal{T} 1_{|x|\ge B^2/2}   \chi _{B^2} P_c  {G}_{\mathbf{m}} \| _{L^2 }   \lesssim  \varepsilon^{-N}    \|    1_{|x|\ge B^2/2}   P_c  {G}_{\mathbf{m}} \| _{L^2 } \le   \varepsilon^{-N}   e ^{-3B} \le     e ^{-2B}.
      \end{align*}
  Next, we consider the 2nd term in  \eqref{eq:Gv1--1}.  Using \eqref{eq:coer51}  and \eqref{eq2stestJ21-4}
  \begin{align*} &       \left| \<  \< x \> ^{ 20}  \mathbf{z}^{\mathbf{m}}  \widetilde{G}_{\mathbf{m}},  \chi  ^{2} _{B^2}  \zeta _B ^{2}   \< x \> ^{-20}  v    \>\right|   \le   |\mathbf{z}^{\mathbf{m}}| \ \|    \mathcal{T}   \chi _{B^2} P_c  {G}_{\mathbf{m}} \| _{\Sigma^{0}}
         \|  \< x \> ^{-20}  v  \| _{L^2} \\& \lesssim |\mathbf{z}^{\mathbf{m}}| \  \|      \chi _{B^2} P_c  {G}_{\mathbf{m}} \| _{\Sigma^{N}}   \(      \| \xi  \| _{\widetilde{\Sigma} }     + \< B \>  ^{-10}   \| w  \| _{\widetilde{\Sigma} }\)  \lesssim  |\mathbf{z}^{\mathbf{m}}|    \(      \| \xi  \| _{\widetilde{\Sigma} }     + \< B \>  ^{-10}   \| w  \| _{\widetilde{\Sigma} }\).
      \end{align*}
      Finally,  we consider the last term in line  \eqref{eq:Gv1--1}.
      Like in  the estimate of $J_2$ in Sect.\  4.4 \cite{KMM3}, from
\begin{align*} &    \xi   '= \chi _{B^2}  \zeta _B v '  + \(   \chi _{B^2}  \zeta _B  \)  ' v
      \end{align*}
we obtain
\begin{align*} &     |\chi _{B^2}  \zeta _B v '   |\lesssim | \xi   '| +  | \(   \chi _{B^2}  \zeta _B  \)  ' v  | \lesssim | \xi    '|  + B ^{-1}|    \chi _{B^2}  \zeta _B   v  | +  B ^{-2}|    \chi _{[B^2\le |x|\le 2B^2]}  \zeta _B   v | ,
      \end{align*}
so that
\begin{align*} &     |\chi _{B^2}^2  \zeta _B v '   | \lesssim | \xi    '|  + B ^{-1}|     \xi   | .
      \end{align*}
Then, using \eqref{eq:estimates116} and the above estimates,  we have
\begin{align*} &      \left| \<\mathbf{z}^{\mathbf{m}} \widetilde{G}_{\mathbf{m}}, \psi  _B v '  \>\right| \le  \left| \<\mathbf{z}^{\mathbf{m}}   \psi  _B\zeta _B ^{-1}   \widetilde{G}_{\mathbf{m}},  \chi _{B^2}^2      \zeta _B  v '  \>  \right| +\left| \<\mathbf{z}^{\mathbf{m}}   \psi  _B\zeta _B ^{-1}    \widetilde{G}_{\mathbf{m}}, \( 1 - \chi _{B^2}^2  \)    \zeta _B  v '  \>  \right| \\& \lesssim  \left| \<\mathbf{z}^{\mathbf{m}}   \psi  _B   |   \widetilde{G}_{\mathbf{m}}|,  | \xi   '|  + B ^{-1}|     \xi   | \>  \right| +      |  \mathbf{z}^{\mathbf{m}}   |      \|   \( 1 - \chi _{B^2}^2  \)  \psi  _B      \widetilde{G}_{\mathbf{m}}    \| _{L^2}   \|          v '    \| _{L^2}\\& \lesssim  |  \mathbf{z}^{\mathbf{m}}   |
\( \| \xi  ' \| _{L^2} +B ^{-1}\| \xi   \| _{\widetilde{\Sigma}} + e ^{-B}    \varepsilon ^{-N}\| w    \| _{\widetilde{\Sigma}}\) .
      \end{align*}

\end{proof}

\textit{Proof of Proposition \ref{prop:2ndvirial}.}
Using \eqref {eq:diffJ1},  Lemmas \ref{lem:2ndv1v}--\ref{lem:Gv1}  and \eqref{eq:relABg}
\begin{align*} &     \frac{d}{dt}\mathcal{J}   (v)=-\<H_{N+1}v, \widetilde{S}_B v\>-\<\mathcal{R}_{v }, \widetilde{S}_B v\>     -\sum_{\mathbf{m}\in \mathbf{R}_{\mathrm{min}}} \<\mathbf{z}^{\mathbf{m}}\widetilde{G}_{ \mathbf{m}}, \widetilde{S}_B v\>
  \lesssim  - \<  - \xi ^{\prime\prime} -  \frac{1}{4}  \chi _{B^2} ^{2}x V_{N+1}  '   \xi ,\xi     \>
\\&+\sum_{\mathbf{m}\in \mathbf{R}_{\mathrm{min}}}|\mathbf{z}^{\mathbf{m}}|    \(     \| \xi    \| _{\widetilde{\Sigma}}    + e^{-B/2} \| w   \| _{\widetilde{\Sigma}}         \) + o_{\varepsilon}(1)  \(     \| \xi    \| _{\widetilde{\Sigma}}  ^2  +   \| w   \| _{\widetilde{\Sigma}} ^2 +   \|\dot {\mathbf{z}}+\im \boldsymbol{\varpi}(\mathbf{z})\mathbf{z}\| ^2       \)
\\& \lesssim - \| \xi   \| _{\widetilde{\Sigma}}^2 +  \sum_{\mathbf{m}\in \mathbf{R}_{\mathrm{min}}}|\mathbf{z}^{\mathbf{m}}| ^2 +o_{\varepsilon}(1) \(   \| w   \| _{\widetilde{\Sigma}}^2   +   \|\dot {\mathbf{z}}+\im \boldsymbol{\varpi}(\mathbf{z})\mathbf{z}\| ^2  \) ,
      \end{align*}
so that integrating in time we obtain  inequality  \eqref{eq:2ndv} concluding the proof of Proposition \ref{prop:2ndvirial}. \qed

Our next task is to  estimate the discrete modes, that is the contributions from $\mathbf{z}$. While so far in the paper we have drawn from Kowalczyk, Martel and Munoz \cite{KMM3}, we now start drawing   from \cite{CM21}.

\section{Proof of Proposition \ref{lem:estdtz}}
\label{sec:estdtz}

Proposition \ref{lem:estdtz} is an immediate  consequence of the following lemma which is taken from \cite{CM21}.

\begin{lemma}\label{lem:zpres}
Under the assumption of Proposition \ref{prop:continuation}, we have
\begin{align}
\dot  z_j+\im \varpi_j(|\mathbf{z}|^2)z_j=-\im \sum_{\mathbf{m}\in \mathbf{R}_{\mathrm{min}}}\mathbf{z}^{\mathbf{m}}\<G_{\mathbf{m}},\phi_j\>+r_j(\mathbf{z},\eta),\label{equat_z}
\end{align}
where $r_j(\mathbf{z},\eta)$ satisfies
\begin{align*}
\|r_j(\mathbf{z},\eta)\|_{L^2(I)}\lesssim  \delta^2\epsilon .
\end{align*}

\end{lemma}
\proof The proof is in \cite{CM21}, but for completeness we reproduce it here.   Recall that     $\phi [\mathbf{z}]$ satisfies identically equation \eqref{eq:rp}.
Furthermore, differentiating \eqref{eq:rp} w.r.t.\ $\mathbf{z}$ in  any given  direction $\widetilde{\mathbf{z}}\in \C ^N$, we obtain
 \begin{align}\label{eq:rfderiv}
 H[\mathbf{z}]D_\mathbf{z}\phi[\mathbf{z}]\widetilde{\mathbf{z}}=&\im D_{\mathbf{z}}^2\phi [\mathbf{z}]   (-\im \boldsymbol{\varpi}(|\mathbf{z}|^2)\mathbf{z},\widetilde{\mathbf{z}})+
 \im D_{\mathbf{z}}\phi[\mathbf{z}]\(D_{\mathbf{z}}(-\im \boldsymbol{\varpi}(|\mathbf{z}|^2)\mathbf{z})\widetilde{\mathbf{z}}\)
 \\&\nonumber+\sum_{\mathbf{m}\in \mathbf{R}_{\mathrm{min}}}D_{\mathbf{z}}(\mathbf{z}^{\mathbf{m}})\widetilde{\mathbf{z}}G_{\mathbf{m}}+D_{\mathbf{z}}\mathcal{R}_{\mathrm{rp}}(\mathbf{z})\widetilde{\mathbf{z}},
 \end{align}
with $ H[\mathbf{z}]$ defined under \eqref{eq:modnls}. By $\eta\in \mathcal{H}_c[\mathbf{z}]$ we obtain the orthogonality relation
\begin{align*}
\<\im \dot \eta, D_{\mathbf{z}}\phi[\mathbf{z}]\widetilde{\mathbf{z}}\>=-\<\im \eta, D_{\mathbf{z}}^2\phi[\mathbf{z}](\dot {\mathbf{z}},\widetilde{\mathbf{z}})\>.
\end{align*}
By applying the inner product  $\< \eta , \cdot \>$  to equation \eqref{eq:rfderiv},   we have
\begin{align*}
\<H[\mathbf{z}]\eta,D_{\mathbf{z}}\phi(\mathbf{z})\widetilde{\mathbf{z}}\>=\<\im \eta,D_{\mathbf{z}}^2\phi(\mathbf{z})(\boldsymbol{\varpi}(|\mathbf{z}|^2)\mathbf{z},\widetilde{\mathbf{z}})\>+\sum_{\mathbf{m}\in \mathbf{R}_{\mathrm{min}}}\<\eta,\(D_{\mathbf{z}}\(\mathbf{z}^{\mathbf{m}}\)\widetilde{\mathbf{z}}\)G_{\mathbf{m}}\>+\<\eta,D_{\mathbf{z}}\mathcal{R}_{\mathrm{rp}}(\mathbf{z})\widetilde{\mathbf{z}}\>,
\end{align*}
where we exploited the selfadjointness of $H[\mathbf{z}]$  and the orthogonality in Lemma \ref{lem:lincor}. Thus,  applying $\< \cdot  , D_{\mathbf{z}}\phi(\mathbf{z})\widetilde{\mathbf{z}}\>$ to  equation \eqref{eq:modnls} for $\eta$ and  using the last two equalities, we obtain
\begin{align}\label{eq:discfund}
 & \<\im D_{\mathbf{z}}\phi(\mathbf{z})(\dot { \mathbf{z}}+\im \boldsymbol{\varpi}(|\mathbf{z}|^2)\mathbf{z}),D_{\mathbf{z}}\phi(\mathbf{z})\widetilde{\mathbf{z}}\>
  =\<\im \eta, D_{\mathbf{z}}^2\phi(\mathbf{z})\(    \dot { \mathbf{z}}+\im \boldsymbol{\varpi}(|\mathbf{z}|^2)\mathbf{z},\widetilde{\mathbf{z}}\)\>
+\<\eta,D_{\mathbf{z}}\mathcal{R}_{\mathrm{rp}}[\mathbf{z}]\widetilde{\mathbf{z}}\>\\ &+\sum_{\mathbf{m}\in \mathbf{R}_{\mathrm{min}}}\<\eta,\(D_{\mathbf{z}}\(\mathbf{z}^{\mathbf{m}}\)\widetilde{\mathbf{z}}\)G_{\mathbf{m}}\>
 +\<\sum_{\mathbf{m}\in \mathbf{R}_{\mathrm{min}}}\mathbf{z}^{\mathbf{m}}G_{\mathbf{m}}+\mathcal{R}_{\mathrm{rp}}[\mathbf{z}],D_{\mathbf{z}}\phi(\mathbf{z})\widetilde{\mathbf{z}}\>
  +\<F(\mathbf{z},\eta)+|\eta |^2 \eta ,D_{\mathbf{z}}\phi[\mathbf{z}]\widetilde{\mathbf{z}}\>.\nonumber
\end{align} First
since $D_{\mathbf{z}}\phi [0]\widetilde{\mathbf{z}}=\widetilde{\mathbf{z}}\cdot \boldsymbol{\phi}$, we have
\begin{align}\label{eq:zj1}
\<\im D_{\mathbf{z}}\phi[\mathbf{z}]( \dot { \mathbf{z}}+\im \boldsymbol{\varpi}(|\mathbf{z}|^2)\mathbf{z}),D_{\mathbf{z}}\phi(\mathbf{z})\widetilde{\mathbf{z}}\>=\sum_{j=1}^N\Re(\im(\dot z_j+\im \varpi_j(|\mathbf{z}|^2)z_j)\overline{\widetilde{z}_j})+r(\mathbf{z},\widetilde{\mathbf{z}}),
\end{align}
where
\begin{align}\label{eq:zj2}
r(\mathbf{z},\widetilde{\mathbf{z}})=&\<\im \(D_{\mathbf{z}}\phi(\mathbf{z})-D_{\mathbf{z}}\phi(0)\)(\dot {\mathbf{z}}+\im \boldsymbol{\varpi}(|\mathbf{z}|^2)\mathbf{z}),D_{\mathbf{z}}\phi(\mathbf{z})\widetilde{\mathbf{z}}\>\\&+\<\im D_{\mathbf{z}}\phi(0)(\dot {\mathbf{z}}+\im \boldsymbol{\varpi}(|\mathbf{z}|^2)\mathbf{z}),\(D_{\mathbf{z}}\phi(\mathbf{z})-D_{\mathbf{z}}\phi(0)\)\widetilde{\mathbf{z}}\>.\nonumber
\end{align}
Since $\|D_{\mathbf{z}}\phi(\mathbf{z})-D_{\mathbf{z}}\phi(0)\|_{L^2}\lesssim |\mathbf{z}|^2\lesssim \delta^2$    by Proposition \ref{prop:rp}
and inequality \eqref{eq:main1}, by assumption \eqref{eq:main11}  we have
\begin{align}\label{eq:zj4}
\|r(\mathbf{z},\widetilde{\mathbf{z}})\|_{L^2(I)}\lesssim \delta ^2 \epsilon  \text{  for all }  \widetilde{\mathbf{z}}=\mathbf{e}_1,\im \mathbf{e}_1,\cdots,\mathbf{e}_N,\im \mathbf{e}_N.
\end{align}
Setting
\begin{align}\label{eq:zj3}
\widetilde{r}(\mathbf{z},\widetilde{\mathbf{z}},\eta):=&\<\im \eta, D_{\mathbf{z}}^2\phi(\mathbf{z})\(\dot {\mathbf{z}}+\im \boldsymbol{\varpi}(|\mathbf{z}|^2)\mathbf{z},\widetilde{\mathbf{z}}\)\>
+\<\eta,D_{\mathbf{z}}\mathcal{R}_{\mathrm{rp}}(\mathbf{z})\widetilde{\mathbf{z}}\>+\sum_{\mathbf{m}\in \mathbf{R}_{\mathrm{min}}}\<\eta,\(D_{\mathbf{z}}\(\mathbf{z}^{\mathbf{m}}\)\widetilde{\mathbf{z}}\)G_{\mathbf{m}}\>\\&+\sum_{\mathbf{m}\in \mathbf{R}_{\mathrm{min}}}\<\mathbf{z}^{\mathbf{m}}G_{\mathbf{m}},\(D_{\mathbf{z}}\phi(\mathbf{z})-D_{\mathbf{z}}\phi(0)\)\widetilde{\mathbf{z}}\>+\<\mathcal{R}_{\mathrm{rp}}(\mathbf{z}),D_{\mathbf{z}}\phi(\mathbf{z})\widetilde{\mathbf{z}}\>+\<F(\mathbf{z},\eta),D_{\mathbf{z}}\phi(\mathbf{z})\widetilde{\mathbf{z}}\>,\nonumber
\end{align}
  by by assumption \eqref{eq:main11}  we have we have
\begin{align}\label{eq:zj5}
\|\widetilde{r}(\mathbf{z},\widetilde{\mathbf{z}},\eta)\|_{L^2(I)}\lesssim  \delta ^2 \epsilon      \text{  for all }  \widetilde{\mathbf{z}}=\mathbf{e}_1,\im \mathbf{e}_1,\cdots,\mathbf{e}_N,\im \mathbf{e}_N.
\end{align}
Therefore, since $D\phi(0)\im^k\mathbf{e}_j=\im^k \phi_j$ ($k=0,1$), we have
\begin{align*}
-\mathrm{Im}\(\partial_t z_j +\im \varpi_j(|\mathbf{z}|^2)z_j\)&=\sum_{\mathbf{m}\in \mathbf{R}_{\mathrm{min}}}\<\mathbf{z}^{\mathbf{m}}G_{\mathbf{m}},\phi_j\>-r(\mathbf{z},\mathbf{e}_j)+\widetilde{r}(\mathbf{z},\mathbf{e}_j,\eta),\\
\mathrm{Re}\(\partial_t z_j + \im  \varpi_j(|\mathbf{z}|^2)z_j\)&=\sum_{\mathbf{m}\in \mathbf{R}_{\mathrm{min}}}\<\mathbf{z}^{\mathbf{m}}G_{\mathbf{m}},\im\phi_j\>-r(\mathbf{z},\im\mathbf{e}_j)+\widetilde{r}(\mathbf{z},\im\mathbf{e}_j,\eta).
\end{align*}
Since $G_{\mathbf{m}}$ and $\phi_j$ are $\R$-valued (see Lemma \ref{lem:GmRval}), we have
\begin{align*}
\dot z_j +  \im \varpi_j(|\mathbf{z}|^2)z_j=-\im \sum_{\mathbf{m}}\<G_{\mathbf{m}},\phi_j\> \mathbf{z}^{\mathbf{m}} -r(\mathbf{z},\im \mathbf{e}_j)+\im r(\mathbf{z},\mathbf{e}_j)+\widetilde{r}(\mathbf{z},\im \mathbf{e}_j,\eta)-\im \widetilde{r}(\mathbf{z},\mathbf{e}_j,\eta).
\end{align*}
Therefore, from
 \eqref{eq:zj4} and
  \eqref{eq:zj5}, we have the conclusion with
  $r_j(\mathbf{z},\eta)=-r(\mathbf{z},\im \mathbf{e}_j)+\im r(\mathbf{z},\mathbf{e}_j)+\widetilde{r}(\mathbf{z},\im \mathbf{e}_j,\eta)-\im \widetilde{r}(\mathbf{z},\mathbf{e}_j,\eta)$.
  \qed

Our next task, is to examine the terms $\mathbf{z} ^{\mathbf{m}}$. We need to show that these terms  satisfy $\mathbf{z} ^{\mathbf{m}} \xrightarrow {t  \to  + \infty  }0$, that is they are damped by nonlinear interaction with the radiation terms.
In order to do so, we   expand the variable $v$, defined in \eqref{eq:vBg}, in a part resonating with the discrete modes $\mathbf{z}$, which will yield the damping,  and a remainder which we denote by $g$.

\section{Smoothing estimate  for $g$}
\label{sec:smoothing}

In analogy to \cite{BP2,SW3,CM15APDE}  and a large literature, we will introduce and bound an auxiliary variable,    $g$   here.
It appears to be impossible to bound  $g$ or analogues of $g$ by means of virial type inequalities. We will use instead Kato--smoothing, as in \cite{BP2,SW3,CM15APDE}. Fortunately, the fact that the cubic nonlinearity is long range is immaterial,  thanks to the cutoff $\chi _{B^2}$ in front of $|\eta|^2\eta $  in the equation of $v$.

The following is elementary and the proof is skipped.
\begin{lemma}\label{lem:generic} 0 is neither an eigenvalue nor a resonance for the operator $H_{N+1}$.
 \end{lemma}
  % \proof We give only a sketch. If the statement is false,   there exists a nonzero and bounded solution of  $H_{N+1}f=0$. We can assume $f$ is real valued. Now, let$[a,b]$ be an interval where $\left . -xV'_{N+1}\right | _{ [a,b] }>0$ and let $\psi \in C_c ^{ \infty}((a,b), [0,+\infty ) )$ be a nonzero function such that $-x\( V'_{N+1}- \lambda \psi  '\)>0$ in $[a,b]$  for all $\lambda \in [0,1]$.  Then it can be shown that for small $\lambda >0$ the operator $H_{N+1}-\lambda \psi $ has exactly one negative eigenvalue. But it is elementary to see that this is incompatible with the fact that $V _{N+1}-\lambda \psi$   is repulsive.

\qed

We recall that we   have the kernel for $x<y$, with an analogous formula for $x>y$, \begin{align}&   R _{H_{N+1}   }(z ) (x,y) =   \frac{T(\sqrt{z})}{2\im \sqrt{z}}     f_- (x, \sqrt{z})   f_+ (y, \sqrt{z})   =   \frac{T(\sqrt{z})}{2\im \sqrt{z}}     e^{\im \sqrt{z} (x-y)}       m_- (x, \sqrt{z})   m_+ (y, \sqrt{z})  ,    \label{eq:resolvKern} \end{align}where  the  Jost functions  $f_{\pm } (x,\sqrt{z} )=e^{\pm \im \sqrt{z} x}m_{\pm } (x,\sqrt{z} )$  solve $ \( - \Delta +  V _{N+1}  \)u=z u$ with\begin{align*} \lim _{x\to +\infty }   {m_{ + } (x,\sqrt{z} )}  =1 =\lim _{x\to -\infty }  {m_{- } (x,\sqrt{z})}  .\end{align*}  These functions  satisfy, see Lemma 1 p. 130 \cite{DT},\begin{align}  \label{eq:kernel2} &  |m_\pm(x, \sqrt{z} )-1|\le  C _1 \langle  \max \{ 0,\mp x \}\rangle\langle \sqrt{z} \rangle ^{-1}\left | \int _x^{\pm \infty}\langle y \rangle |V_{N+1} (y)| dy \right |     \\ &  |m_\pm(x, k )-1|\le    \<    \sqrt{z}\>  ^{-1}  \left | \int _x^{\pm \infty}   |V _{N+1}(y)| dy \right |    \exp \(  \<    \sqrt{z}\>  ^{-1}  \left | \int _x^{\pm \infty}   | {V}_{N+1}(y)| dy \right | \)   , \label{eq:kernel2aw0}   \end{align} while, by Lemma \ref{lem:generic},  $T(k) =\alpha k (1+o(1))$  near $k=0$ for some $\alpha \in \R$, see \cite[formula (2.45)]{weder},    and $T(k) = 1+O(1/k) $ for $k\to \infty$ and $T\in C^0(\R )$, see Theorem 1 \cite{DT}.

 Looking at the equation for $v$, \eqref{eq:vBg}, we introduce the functions \begin{align}&  \label{eq:def_zetam}
 \rho _ {\mathbf{m}}:= -R ^{+}_{H _{N+1} }(\boldsymbol{\omega} \cdot \mathbf{m})  \widetilde{G}_{ \mathbf{m}},
\end{align}
which solve
\begin{align}&  \label{eq:eq_zetam}
  (H _{N+1}-\boldsymbol{\omega} \cdot \mathbf{m}  )\rho _{ \mathbf{m}} = -  \widetilde{G}_{ \mathbf{m}}
\end{align}
and we set
\begin{align} \label{eq:expan_v}
 g= v+ Z(\mathbf{z}) \text{  where } Z(\mathbf{z}):=-  \sum_{\mathbf{m}\in \mathbf{R}_{\mathrm{min}}}   \mathbf{z} ^{\mathbf{m}} {\rho} _{ \mathbf{m}} ,
\end{align}
which is analogous     the expansion of $\overrightarrow{h}$ in p. 86 in Buslaev and Perelman  \cite{BP2}   or also to the formula under (4.5) in Merle and Raphael \cite{MR4}.   An elementary computation yields
\begin{align*}  &
\im \partial_t     g    = H_{N+1}g - \sum_{\mathbf{m}\in \mathbf{R}_{\min}}\(\im \partial_t \(\mathbf{z}^{\mathbf{m}}\)-\boldsymbol{\omega}\cdot \mathbf{m} \, \mathbf{z}^{\mathbf{m}} \) \rho _{ \mathbf{m}}+ \mathcal{R}_{v}
\end{align*}
or, equivalently,
\begin{align} \label{eq:equation_g11}& g(t) = e^{-\im tH_{N+1} }v(0) + \sum_{\mathbf{m}\in \mathbf{R}_{\min}}\mathbf{z}^{\mathbf{m}}(0) e^{-\im tH_{N+1} }R ^{+}_{H _{N+1} }(\boldsymbol{\omega} \cdot \mathbf{m})  \widetilde{G}_{ \mathbf{m}} \\&  \label{eq:equation_g12}-  \sum_{\mathbf{m}\in \mathbf{R}_{\min}}\im \int _0 ^t e^{-\im (t-t') H_{N+1} }\(  \partial_t \(\mathbf{z}^{\mathbf{m}}\)+\im \boldsymbol{\omega}\cdot \mathbf{m}  \, \mathbf{z}^{\mathbf{m}} \) \rho _{ \mathbf{m}} dt' \\& \label{eq:equation_g13}- \im \int _0 ^t e^{-\im (t-t') H_{N+1} }\mathcal{T} \(2\chi_{B^2}'\partial_x + \chi_{B^2}''\)\tilde{\eta} dt'
\\& \label{eq:equation_g14}- \im \int _0 ^t e^{-\im (t-t') H_{N+1} } \( \<\im \varepsilon \partial_x\>^{-N}[V_{N+1},\< \im \varepsilon  \partial_x\>^N]v + \mathcal{T}  \chi_{B^2}\mathcal{R}_{\widetilde{\eta}} \) .
\end{align}

We will prove the following, where we use the weighted spaces defined in \eqref{eq:withL2}.
\begin{proposition}\label{prop:estg1} For $S>4$  There exist  constants $c_0>0$  and $C(C_0)$ such that  \begin{align}   \label{eq:estg11}&
  \| g   \| _{L^2(I, L^{2,-S}(\R )) }       \lesssim  \varepsilon ^{-N}B^{2+2\tau } \delta  +\varepsilon \epsilon +      \epsilon ^2.
\end{align}
 \end{proposition}
To prove Proposition \ref{prop:estg1}   we will need to bound one by one the terms in \eqref{eq:equation_g11}--\eqref{eq:equation_g14} in various lemmas.

        Lemma \ref{lem:generic}   implies that $H_{N+1} $ is a \textit{generic} operator, and that in particular   the following Kato smoothing holds, which is sufficient for our purposes. The proof is standard, is similar for example to Lemma 3.3 \cite{CT} and we skip it.

   \begin{lemma}\label{lem:smooth}
For any $S >3/2$  there exists a fixed $ c(S)$ s.t.
 \begin{equation} \label{eq:smooth1a}
\|  \< x \> ^{-S }e^{-\im H_{N+1}   t}
f \|_{L^{ 2 }(\R^2 )} \le c(S)  \| f   \|_{L^{ 2 }(\R )}  \text{  for all $f \in L^{ 2 }(\R )$} .
\end{equation}
\end{lemma}

\qed

Lemma \ref{lem:smooth}, inequality \eqref{eq:estimates115},  the definition of $w$ in \eqref{def:wAxiB}, Lemma \ref{lem:GNTR}, the Modulation Lemma \ref{lem:lincor} and
    the conservation of mass and of energy
yield
 \begin{align}   \nonumber
  \| e^{-\im tH_{N+1} }v(0)   \| _{L^2(\R, L^{2,-S}(\R )) }& \lesssim \| v(0)\| _{L^2}\lesssim   \varepsilon ^{-N} B^2 \| w(0) \|_{\widetilde{\Sigma}} \lesssim \varepsilon ^{-N} B^2 \| \widetilde{\eta }(0) \|_{H^1} \\&\lesssim \varepsilon ^{-N} B^2 \|  {\eta }(0) \|_{H1} \lesssim \varepsilon ^{-N} B^2 \| u_0 \|_{H^1} \le \varepsilon ^{-N} B^2 \delta. \label{eq:smooth1}
\end{align}
Next, we have the following lemma, which is standard in this theory, see \cite{BP2,SW3,CM15APDE}.

\begin{lemma}\label{lem:lemg9}
Let $ \Lambda$ be a
finite subset of $( 0,\infty) $
and let $S >4$. Then  there exists a fixed $ c(S,\Lambda)$ s.t.
  for every $t \ge 0$ and  $\lambda \in  \Lambda$
\begin{equation} \label{eq:lemg91}
\|  e^{-\im H_{N+1}   t}R_{ H_{N+1}  }^{+}( \lambda )
 f \|_{L^{ 2, - S}(\R )} \le c(S,\Lambda)\langle  t\rangle ^{-\frac 32} \|  f  \|_{L^{ 2,   S}(\R  )}  \text{  for all $f\in L^{ 2,   S}(\R )$} .
\end{equation}
\end{lemma}
 \textit{Proof (sketch).}     This lemma is  similar to Proposition 2.2 \cite{SW3}. We will consider case $t\ge 1$, while we skip the simpler case  $t\in [0,1]$.

  \noindent  Consider $\Lambda \subset (a,b ) $  with
 $[a,b]\subset \R _+$. Let $\mathbf{g}\in C  ^{\infty}\(  (  a/2, +\infty ) , [0,1]\)$ such that $\mathbf{g}\equiv 1$ in  $[a, +\infty ) $. Let  $\mathbf{g }_1\in  C  ^{\infty} _c \( \R , [0,1]\)$    with $\mathbf{g}_1 =1-\mathbf{g}$  in $\R _+$.  Next we consider
\begin{align*}  & \| \<  x\> ^{-S} e^{-\im H_{N+1}   t}R_{ H_{N+1} }  ^{+}( \lambda )
 \mathbf{g}_1\( H _{N+1} \)    f \|_{L^{ 2 }(\R )}\\& \lesssim \| \<  x\> ^{-S+2} e^{-\im H_{N+1}  t}R_{ H _{N+1} }^{+}( \lambda )
 \mathbf{g}_1\( H _{N+1} \)     f \|_{L^{ \infty }(\R )}  \lesssim  t^{-\frac{3}{2}} \| \<  x\>    R_{ H _{N+1} }^{+}( \lambda )
 \mathbf{g}_1\( H _{N+1} \)     f \|_{L^{ 1}(\R )} \\&  \lesssim  t^{-\frac{3}{2}} \| \<  x\> ^{2}   R_{ H _{N+1} }^{+}( \lambda )
 \mathbf{g}_1\( H _{N+1} \)     f \|_{L^{ 2}(\R )}  \lesssim t^{-\frac{3}{2}} \| \<  x\> ^{2}   f \|_{L^{ 2}(\R )}
\end{align*}
 where we used Theorem 3.1  \cite{Schlag}  and Lemma  \ref{lem:pd01}, since     $R_{ H _{N+1} }^{+}( \lambda )
 \mathbf{g}_1\( H  _{N+1}\)   = \mathbf{g}_2\( H  _{N+1}\)   $,  with $\mathbf{g }_2\in  C  ^{\infty} _c \( \R , \R\)$,
      is   a  0 order $\Psi$DO with symbols satisfying the inequalities \eqref{eq:pd0110} uniformly as $\lambda$ takes finitely many values. Here we used Theorem 8.7 in Dimassi and Sj\"ostrand \cite{DimassiS}.

 \noindent  Next we consider
  \begin{align}
&\langle x \rangle ^{-S }\mathbf{ g} \( H _{N+1} \) e^{-\im H_{N+1}   t}R_{ H_{N+1}   }^{+}( \lambda )
 \langle y \rangle ^{-S }= \nonumber \\& \lim _{\sigma  \to  0 ^+}
   e^{-\im \lambda t}\langle x
\rangle ^{-S } \int _t^{+\infty }e^{-\im (H_{N+1}-\lambda -\im \sigma )s}
 \mathbf{g} \( H_{N+1}  \)
  ds\langle y \rangle ^{-S }.\label{term1}
\end{align}
Using the distorted plane waves $\psi (x,k )$ associated to  $H _{N+1}$, see (1.9) \cite{weder},  we can write the following integral kernel, ignoring irrelevant constants,
\begin{equation}\label{term2} \begin{aligned}&  \langle x
\rangle ^{-S }  \( e^{-\im (H_{N+1} -\lambda -\im \sigma )s}
 \mathbf{g} \( H _{N+1} \) \) (x,y)
   \langle y \rangle ^{-S }  \\& =  \langle x
\rangle ^{-S }  \langle y \rangle ^{-S } \int _{\R_+}       e^{-\im (k^2-\lambda -\im \sigma )s-\im k (x-y)}
 \mathbf{g} \( k^2 \)    \overline{m_+ (x,k )} m_+ (y,k ) dk \\& + \langle x
\rangle ^{-S }  \langle y \rangle ^{-S } \int _{\R_-}       e^{-\im (k^2-\lambda -\im \sigma )s-\im k (x-y)}
 \mathbf{g} \( k^2 \)    \overline{m_- (x,-k )} m_- (y,-k ) dk
    .
\end{aligned}
\end{equation}
Take for example the first term in the right hand side of \eqref{term2}. Then, from $\frac{\im}{2ks}\frac{d}{dk}e^{-\im  k^2 s }  = e^{-\im  k^2 s }$  and taking the limit $\sigma \to 0^+$, we can write it as
\begin{align*}  &  \langle x
\rangle ^{-S }  \langle y \rangle ^{-S } \int _{\R_+}  e^{-\im (k^2-\lambda   )s }  \(  -\frac{d}{dk}  \frac{\im}{2ks} \) ^3 \(  e^{ -\im k (x-y)}
 \mathbf{g} \( k^2 \)    \overline{m_+ (x,k )} m_+ (y,k )\) dk ,
\end{align*}
which,  using for instance the bounds on the $k$ derivatives of $m_\pm  $   in Lemma 2.1 \cite{GPR18},  is absolutely integrable in $k$ ($\mathbf{g}$ is constant outside a bounded interval) and is bounded   in absolute value by
\begin{align*}  &  \lesssim  \langle x
\rangle ^{-S+3 }  \langle y \rangle ^{-S+3 }  s^{-3}.
\end{align*}
Integrating in $[t,\infty )$ we obtain an upper bound $\sim \langle x \rangle ^{-S+3 }  \langle y \rangle ^{-S+3 }  t^{-2}$ for   integral kernel of the operator of the corresponding part  in \eqref{term2}, which  gives an upper bound  of $t ^{-2}$ in the  corresponding  contribution in
\eqref{eq:lemg91}.   So we get the desired result for $t\ge 1$.
\qed

Lemma \ref{lem:lemg9}, \eqref{def:Gbgm}, \eqref{def:Tg}     imply that
 \begin{align}&  \label{eq:smooth2}
 \sum_{\mathbf{m}\in \mathbf{R}_{\min}} |\mathbf{z}^{\mathbf{m}}(0)| \| e^{-\im tH_{N+1} }\rho _{ \mathbf{m}}    \| _{L^2(\R, L^{2,-S}(\R )) }
 \lesssim \sum_{\mathbf{m}\in \mathbf{R}_{\min}} |\mathbf{z}^{\mathbf{m}}(0)|  \| \widetilde{G} _{\mathbf{m}}    \| _{  L^{2, S}(\R  ) }\\& \lesssim  \delta ^2
  \|  \< x \> ^{S}  \< \im \varepsilon \partial _x\> ^{-N}  \< x \> ^{- S}    \| _{  L^{2 }(\R  )\to   L^{2 }(\R  )}    \sup_{\mathbf{m}\in \mathbf{R}_{\min}}   \|  \< x \> ^{ S}    \mathcal{A}^* \chi _{B^2}{G} _{\mathbf{m}}    \| _{  L^{2 }(\R  ) }     \lesssim \delta ^2    ,  \nonumber
\end{align}
where we used the fact that $T:=\< x \> ^{S}  \< \im \varepsilon \partial _x\> ^{-N}  \< x \> ^{- S} $ has integral kernel
\begin{align}  &  T(x,y)  =\varepsilon ^{-1}\< x \> ^{S} \< y \> ^{- S}   f\( \varepsilon ^{-1} (x-y) \) , \text{ where  } \widehat{f}(k)=\< k\> ^{-N}\label{eq:opT}
\end{align}
  where $f$  is a continuous rapidly decreasing function, so that  it is easy to see  that Young's inequality, see \cite[Theorem 0.3.1]{sogge},   gives     $\| T \| _{L^2\to L^2}\lesssim 1$ uniformly in $\varepsilon \in (0,1]$.

Notice that in this way we gave a bound on the contribution  of the terms in the right hand side in \eqref{eq:equation_g11}
  to \eqref{eq:estg11} .

\noindent It is easy to bound the contribution   to \eqref{eq:estg11} of the term \eqref{eq:equation_g12}.
Indeed, using   the identity
\begin{align}\label{eq:difzm}
\( D_{\mathbf{z}}\mathbf{z}^{\mathbf{m}}\)(\im \boldsymbol{\omega}\mathbf{z})=\im \mathbf{m}\cdot \boldsymbol{\omega} \  \mathbf{z}^{\mathbf{m}}\text{ , where }\boldsymbol{\omega}\mathbf{z}:=(\omega_1z_1,\cdots,\omega_Nz_N),
\end{align}
 we have
\begin{align*}&     \partial_t \(\mathbf{z}^{\mathbf{m}}\)+\im \boldsymbol{\omega}\cdot \mathbf{m} \mathbf{z}^{\mathbf{m}}  = D_{\mathbf{z}}\mathbf{z}^\mathbf{m} \( \partial _t \mathbf{z} +\boldsymbol{\omega}\mathbf{z} \)  = D_{\mathbf{z}}\mathbf{z}^\mathbf{m} \( \dot { \mathbf{z}} +\im \boldsymbol{\varpi} (|\mathbf{z}|^2) \mathbf{z} \)   +  D_{\mathbf{z}}\mathbf{z}^\mathbf{m}
\im \(  \boldsymbol{\omega}   -\boldsymbol{\varpi} (|\mathbf{z}|^2) \) \mathbf{ z} \\& = D_{\mathbf{z}}\mathbf{z}^\mathbf{m} \( \dot { \mathbf{z}} +\im \boldsymbol{\varpi} (|\mathbf{z}|^2) \mathbf{z} \)   +  \im  \mathbf{m}\cdot  \(  \boldsymbol{\omega}   -\boldsymbol{\varpi} (|\mathbf{z}|^2) \) \mathbf{z}^\mathbf{m}.
\end{align*}
From this and Lemma \ref{lem:lemg9} and the bound $ \| \widetilde{G} _{\mathbf{m}}    \| _{  L^{2, S}(\R  ) }\lesssim 1$ in \eqref{eq:smooth2}  we obtain
\begin{align}&  \label{eq:smooth3}
  \sum_{\mathbf{m}\in \mathbf{R}_{\min}} \| \int _0 ^t  e^{-\im (t-t') H_{N+1} } \(  \partial_t \(\mathbf{z}^{\mathbf{m}}\)+\im \boldsymbol{\omega}\cdot \mathbf{m} \mathbf{z}^{\mathbf{m}} \) \rho _{ \mathbf{m}} \| _{L^2(I, L^{2,-S}(\R )) }  \\& \nonumber  \lesssim
   \delta ^2    \sum_{\mathbf{m}\in \mathbf{R}_{\min}}\(  \|  \dot { \mathbf{z}} +\im \boldsymbol{\varpi} (|\mathbf{z}|^2) \mathbf{z}\| _{L^2(I,  ) }+\|  \mathbf{z}^{\mathbf{m}}\| _{L^2(I,  ) } \)  \| \widetilde{G} _{\mathbf{m}}    \| _{  L^{2, S}(\R  ) }  \lesssim   \delta ^2 \epsilon .
\end{align}

Now we look at the  contribution   to \eqref{eq:estg11} of the term \eqref{eq:equation_g13}.
We will need the following result about the Limiting Absorption Principle. The following is related to Lemma 5.7 \cite{CT}.
 \begin{lemma} \label{lem:LAP} For  $S>5/2$ and $\tau >1/2$ we have
 \begin{align}&   \label{eq:LAP1}   \sup _{z \in \R } \|   R ^{\pm }_{H_{N+1}}(z ) \| _{L^{2,\tau}(\R ) \to L^{2,-S}(\R )} <\infty  .
\end{align}
\end{lemma}
\proof  It is equivalent to show $\displaystyle \sup _{z \in \R } \|  \< x \> ^{-S} R ^{\pm }_{H_{N+1}}(z )    \< y \> ^{-\tau} \| _{L^{2 }(\R ) \to L^{2 }(\R )}<\infty$.  We will consider only the $+$ case.
 We consider the square of the  Hilbert--Schmidt norm
 \begin{align*}  &   \int _{\R} dx \< x \> ^{-2S} \int_{\R}  |R ^{+ }_{H_{N+1}}(x,y,z )| ^2  \< y \> ^{-2\tau} dy   =    \int _{\R} dx \< x \> ^{-2S} \int_{-\infty}^{x}  |R ^{+ }_{H_{N+1}}(x,y,z )| ^2  \< y \> ^{-2\tau} dy\\& + \int _{\R} dx \< x \> ^{-2S} \int_{x}^{+\infty}  |R ^{+ }_{H_{N+1}}(x,y,z )| ^2  \< y \> ^{-2\tau} dy.
\end{align*}
 We will bound only the second term in the right hand side: for the   first   term the argument is similar.  Recalling formula \eqref{eq:resolvKern}, we have to bound
  \begin{align*}  &   \left |  \frac{T(\sqrt{z})}{2\im \sqrt{z}} \right |    \int _{x<y}   \< x \> ^{-2S}   | m_- (x, \sqrt{z})   m_+ (y, \sqrt{z}) | ^2  \< y \> ^{-2\tau}dx  dy \\& \lesssim   \int _{x<y} \< x \> ^{-2S}  \< y \> ^{-2\tau}  \( 1+ \max \( x, 0\) +   \max \( -y, 0\) \) ^2   dx  dy ,
\end{align*}
  where we used the bound \eqref{eq:kernel2}. Now, in the last integral we can distinguish the region  $|y|\lesssim |x|$, where the corresponding contribution can be bounded by
\begin{align*}  &       \int _{\R ^2}   \< x \> ^{-2(S-2)}  \< y \> ^{-2\tau}     dx  dy  <\infty \text{ for $S>5/2$ and $\tau >1/2$,}
\end{align*}
and the region $|y|\gg |x|$, where we have the same bound, because $x<y$ and $|y|\gg |x|$ imply that $y>0$, and hence $\max \( -y, 0\) =0.$
\qed

We will also need   the following formulas that we take from Mizumachi \cite[Lemma 4.5]{mizu08} and to which we refer for the proof.
 \begin{lemma} \label{lem:lemma11} Let  for $g\in \mathcal{S}(\R \times \R , \C )$ \begin{align*}&  U(t,x) = \frac{1}{\sqrt{2\pi} \im }\int _\R e^{-\im \lambda t}\( R ^{-}_{H_{N+1}}(\lambda )+R ^{+}_{H_{N+1}}(\lambda )   \) \mathcal{F}_{t}^{-1}g( \lambda , \cdot ) d\lambda  ,
\end{align*}
where $\mathcal{F}_{t}^{-1}$  is the inverse Fourier transform in $t$. Then
 \begin{align} \label{eq:lemma11}2\int _0^t e^{-\im (t-t') H_{N+1} }g(t') dt'   &=  U(t,x)  - \int  _{\R _-} e^{-\im (t-t') H_{N+1} }g(t') dt'
 \\& + \int  _{\R _+} e^{-\im (t-t') H_{N+1} }g(t') dt' .\nonumber
\end{align}
\end{lemma}
\qed

The last two lemmas give us the following smoothing estimate.
\begin{lemma} \label{lem:smoothest} For  $S>5/2$ and $\tau >1/2$ there exists a constant $C(S,\tau )$ such that we have
 \begin{align}&   \label{eq:smoothest1}   \left \|   \int   _{0} ^{t   }e^{-\im (t-t') H_{N+1} }g(t') dt' \right \| _{L^2( \R ,L^{2,-S}(\R ))  } \le C(S,\tau ) \|  g \| _{L^2( \R , L^{2,\tau}(\R ) ) }.
\end{align}
\end{lemma}
\proof  We can use formula \eqref{eq:lemma11} and bound $U$, with the bound on the last two terms in the right hand side of \eqref{eq:lemma11} similar. So we have, taking Fourier transform in $t$,
$$ \aligned &
\| U\|_{
L_{t}^2L^{2,-S}} \le   2 \sup _{\pm}\| R_{ H_{N+1} }^\pm (\lambda )
   \widehat{ g}(\lambda,\cdot )\|_{L_{\lambda  }^2 L^{2,-S}}   \le \\& \le
2\sup _{\pm}   \sup _{\lambda \in \R }
\|  R_{ H_{N+1} }^\pm (\lambda )  \| _{ L^{2,\tau}  \to L^{2,-S}   } \|
     \widehat{g} (\lambda,x) \|_{ L^{2,\tau} L_{\lambda  }^2 }\,
 \lesssim  &        \| g\|_{L_{t }^2L^{2,\tau} }.
\endaligned $$
Notice that, while Lemma \ref{lem:lemma11} is stated for $g\in \mathcal{S}(\R \times \R , \C )$, the estimate \eqref{eq:smoothest1} extends to all $g\in L^2( \R , L^{2,\tau}(\R ) )$ by density.

 \qed

\begin{remark} \label{rem:errorCT}  The above is basically Lemma 3.4 \cite{CT}, which in turn is based on an  argument in \cite{mizu08}.  Unfortunately Lemma 3.4 \cite{CT} has a mistake, which however can be corrected using Lemma \ref{lem:lemma11}, as we did here.
\end{remark}

We now examine the term in \eqref{eq:equation_g13}. By Lemma \ref{lem:smoothest} we have
\begin{align*}
  \|   \int   _{0} ^{t   }e^{-\im (t-t') H_{N+1} }\mathcal{T} \(2\chi_{B^2}'\partial_x + \chi_{B^2}''\)\tilde{\eta}  dt' \| _{L^2( I, L^{2,-S}(\R ) )  }  \lesssim \|    \mathcal{T} \(2\chi_{B^2}'\partial_x + \chi_{B^2}''\)\tilde{\eta}   \| _{L^2( I, L^{2,\tau}(\R ) )  }.
\end{align*}
In order to bound the right hand side we expand
\begin{align*}
      \mathcal{T} \(2\chi_{B^2}'\partial_x + \chi_{B^2}''\)\tilde{\eta}   & =     \mathcal{T} \(2\chi_{B^2}'\partial_x + \chi_{B^2}''\)w \( \zeta _A ^{-1}-1  \) - 2 \mathcal{T} \chi_{B^2}' \zeta _A^{-2}\zeta _A'w\\& + \mathcal{T} \(2\chi_{B^2}'\partial_x + \chi_{B^2}''\)w   .
\end{align*}
By $|\chi_{B^2}' \zeta _A^{-2}\zeta _A'|\lesssim  A ^{-1}|\chi_{B^2}'  |$ and $ 1 _{|x|\le 2 B^2}| \zeta _A ^{-1}-1  |\lesssim 1 _{|x|\le 2 B^2} \frac{B^2}{A}$, both of which are small, the main term is the one in the last line, which is the only one we discuss explicitly, because the others are similar, simpler and smaller. We decompose
\begin{align}\nonumber
        \mathcal{T} \(2\chi_{B^2}'\partial_x + \chi_{B^2}''\)w   =\mathcal{I}+\mathcal{II} \text{  where } &\mathcal{I}:=  1 _{2 ^{-1}B^2\le |x|\le 3 B^2}\mathcal{T} \(2\chi_{B^2}'\partial_x + \chi_{B^2}''\)w ,  \\& \mathcal{II} :=  1 _{ \{ |x|\le 2 ^{-1}B^2 \} \cup   \{ |x|\ge 3 B^2\} }\mathcal{T} .\label{eq:decomp1}
\end{align}
 By Lemmas \ref{lem:equiv_rho} and \ref{lem:coer5}, we have
\begin{align} & \| \< x \> ^{\tau} \mathcal{I}  \| _{L^2 (I, L^2(\R ))} = \| \< x \> ^{\tau}    1 _{2 ^{-1}B^2\le |x|\le 3 B^2}\mathcal{T} \(2\chi_{B^2}'\partial_x + \chi_{B^2}''\)w  \| _{L^2 (I, L^2(\R ))}  \nonumber \\&  \nonumber \lesssim B ^{2\tau }\|  \mathcal{T} \(2\chi_{B^2}'\partial_x + \chi_{B^2}''\)w  \| _{L^2 (I, L^2(\R ))} \lesssim \varepsilon ^{-N}B ^{2\tau }\|   2\chi_{B^2}'w' + \chi_{B^2}'' w  \| _{L^2 (I, L^2(\R ))}    \\& \nonumber  \lesssim  \varepsilon ^{-N}  B ^{2\tau -2} \| w'  \| _{L^2 (I, L^2(\R ))} + \varepsilon ^{-N} B ^{2\tau -4} \|  1 _{ B^2\le |x|\le 2 B^2}    w   \| _{L^2 (I, L^2(\R ))}\\& \lesssim \varepsilon ^{-N}  B ^{2\tau -2}\|     w   \| _{L^2 (I, \widetilde{\Sigma} )} + \varepsilon ^{-N} B ^{2\tau -4}    \| \< x\>  1 _{ B^2\le |x|\le 2 B^2}     \| _{  L^1(\R ) } ^{\frac{1}{2}}  \|     w   \| _{L^2 (I, \widetilde{\Sigma} )} \nonumber \\&  \lesssim  \varepsilon ^{-N} B ^{2\tau -2} \|     w   \| _{L^2 (I, \widetilde{\Sigma} )} \le  B ^{- \frac{1}{2}} \epsilon .   \label{eq:smooth5}
\end{align}
By Lemma \ref{lem:coer5} we have
\begin{align*} & \| \< x \> ^{\tau} \mathcal{II} \| _{L^2(\R )} \lesssim \| \< x \> ^{\tau}   1 _{ \{ |x|\le 2 ^{-1}B^2 \} \cup   \{ |x|\ge 3 B^2\} }
\int e ^{-\frac{|x-y|}{2\varepsilon}}   \(2\chi_{B^2}'w' + \chi_{B^2}''w\)  \| _{L^2(\R )}\\& \lesssim e ^{-B^2}\| \mathcal{K} \(   \< y \> ^{\tau} \(2\chi_{B^2}'w' + \chi_{B^2}''w\) \) \| _{L^2(\R )} ,
\end{align*}
where the operator $\mathcal{K}f =\int \mathcal{K}(x,y)f(y) dy$  has integral kernel
\begin{align*} &  \mathcal{K}(x,y) =  \< x \> ^{\tau}  e ^{- {|x-y|} } \< y \> ^{-\tau}.
\end{align*}
Since we have
\begin{align*} &  \| \mathcal{K} \| _{L^2(\R )\to L^2(\R )}^2 \le \| \mathcal{K}(\cdot ,\cdot ) \| _{L^2(\R \times \R )}^2<+\infty,
\end{align*}
  by  the bounds implicit  in \eqref{eq:smooth5}, we have
  \begin{align} \nonumber & \| \< x \> ^{\tau} \mathcal{II} \| _{L^2 (I, L^2(\R ))} \lesssim  e ^{-B^2} \| \< x \> ^{\tau} \(2\chi_{B^2}'w' + \chi_{B^2}''w\) \| _{L^2 (I, L^2(\R ))} \\& \lesssim e ^{-B^2}B^{2\tau}\|   \(2\chi_{B^2}'w' + \chi_{B^2}''w\) \| _{L^2 (I, L^2(\R ))}
   \lesssim e ^{-B^2/2} \epsilon . \label{eq:smooth4}
\end{align}
We next consider the terms in \eqref{eq:equation_g14}, starting with
\begin{align} &  \nonumber \| \< x \> ^{-S}      \int _0 ^t e^{-\im (t-t') H_{N+1} }   \<\im \varepsilon \partial_x\>^{-N}[V_{N+1},\< \im \varepsilon  \partial_x\>^N]v \| _{L^2 (I, L^2(\R ))} \\& \lesssim \| \< x \> ^{\tau } \<\im \varepsilon \partial_x\>^{-N}[V_{N+1},\< \im \varepsilon  \partial_x\>^N]v  \| _{L^2 (I, L^2(\R ))} \nonumber  \\& \lesssim  \varepsilon \| \< x \> ^{-100 }  v  \| _{L^2 (I, L^2(\R ))} \lesssim  \varepsilon  \( \| \xi   \| _{L^2 (I, \widetilde{\Sigma} )}   + B ^{-1} \| \xi   \| _{L^2 (I, \widetilde{\Sigma} )}  \) \lesssim   \varepsilon  \epsilon , \label{eq:smooth6}
\end{align}
where we   used Lemma \ref{eq2stestJ21II} in the first inequality in the last line, and \eqref{eq2stestJ21-4} for the second inequality.

\noindent We now consider  remaining contributions of \eqref{eq:equation_g14} to \eqref{eq:estg11}.
To start with, by Lemma \ref{lem:smooth}  we have
\begin{align} \nonumber &    \| \< x \> ^{-S}      \int _0 ^t e^{-\im (t-t') H_{N+1} }   \mathcal{T}\chi _{B^2} \mathcal{R}_{\widetilde{\eta}} dt' \| _{L^2 (I, L^2(\R ))}   \lesssim \| \< x \> ^{\tau }  \mathcal{T}\chi _{B^2} \mathcal{R}_{\widetilde{\eta}}   \| _{L^2 (I, L^2(\R ))}.
\end{align}
The right hand side il less than $I+II$ where
\begin{align*} \nonumber & I =    \|    1 _{|x|\le 3B^2}  \< x \> ^{\tau }  \mathcal{T}\chi _{B^2} \mathcal{R}_{\widetilde{\eta}}   \| _{L^2 (I, L^2(\R ))} \\& II =    \|    1 _{|x|\ge 3B^2}  \< x \> ^{\tau }  \mathcal{T}\chi _{B^2} \mathcal{R}_{\widetilde{\eta}}   \| _{L^2 (I, L^2(\R ))}
\end{align*}
We have
\begin{align*}   & I   \lesssim B ^{2\tau } (I_1+I_2)   \\& \nonumber I_1 =    \|    P_c\(-\im D_{\mathbf{z}}\phi[\mathbf{z}]\(\dot {\mathbf{z}}+\im \boldsymbol{\varpi}(\mathbf{z})\mathbf{z}\) +\mathcal{R}_{\mathrm{rp}}[\mathbf{z}]+F[\mathbf{z},\eta]+L[\mathbf{z}] \eta    \)  \| _{L^2 (I, L^2(\R ))} \\& I_2=  \|   \chi _{B^2} P_c |\eta | ^2 \eta  \| _{L^2 (I, L^2(\R ))}.\nonumber
\end{align*}
By $\| P_c D_{\mathbf{z}}\phi[\mathbf{z}] \| _{\widetilde{\Sigma}}= O \(  \| \mathbf{z}  \| ^2 \)$ because of $  D_{\mathbf{z}}\phi[0]\widetilde{\mathbf{z}}=\boldsymbol{\phi} \cdot \widetilde{\mathbf{z}}$  for any $\widetilde{\mathbf{z}}\in \C^N$, it is easy to conclude
\begin{align*}   & I_1   \lesssim   \delta ^2 \epsilon   .
\end{align*}
We have
\begin{align*}   & I_2   \lesssim      \|   \chi _{B^2}  |\eta | ^2 \eta       \| _{L^2 (I, L^2(\R ))}  + \|    \chi _{B^2} P_d |\eta | ^2 \eta       \| _{L^2 (I, L^2(\R ))}  \\& \lesssim  \sum _{j=1}^{N}  \| \< x \> ^{\tau }  \chi _{B^2} \phi _j ( |\eta | ^2 \eta , \phi _j)   \| _{L^2 (I, L^2(\R ))} +   \|   \eta \| _{L^\infty  (\R , H^1(\R ))}^2   \|      w   \| _{L^2 (I, L^2( |x|\le 2 B^2 ))}\nonumber \\& \lesssim \sum _{j=1}^{N} \|   \eta \| _{L^\infty  (\R , H^1(\R ))}^2  \| w \|   _{L^2 (I, \widetilde{\Sigma} )} + B ^{ 2}    \|   \eta \| _{L^\infty  (\R , H^1(\R ))}^2   \|      w   \| _{L^2 (I, \widetilde{\Sigma} )\nonumber}\\& \lesssim B^2\|   \eta \| _{L^\infty  (\R , H^1(\R ))}^2   \|      w   \| _{L^2 (I, \widetilde{\Sigma} )}\lesssim B^2 \delta ^2 \epsilon. \nonumber
\end{align*}
So we conclude
\begin{align} \label{eq:smooth7} & I \lesssim B ^{2\tau + 2} \delta ^2 \epsilon.
\end{align}
Turning to the analysis of $II$, we have
\begin{align} \nonumber  & II   \lesssim     \|    1 _{|x|\ge 3B^2}  \< x \> ^{\tau }  \mathcal{T}    \< x \> ^{-\tau } 1 _{|x|\le 2B^2}   \| _{  L^2(\R ) \to L^2(\R ) }     \|      \< x \> ^{\tau }  \chi _{B^2} \mathcal{R}_{\widetilde{\eta}}   \| _{L^2 (I, L^2(\R ))} \\& \lesssim \|      \< x \> ^{\tau }  \chi _{B^2} \mathcal{R}_{\widetilde{\eta}}   \| _{L^2 (I, L^2(\R ))}\lesssim B ^{2\tau + 2} \delta ^2 \epsilon \label{eq:smooth8}
\end{align}
by an analysis similar to the operator $\mathcal{K}$ above and to the analysis of $I$.

Taken together, \eqref{eq:smooth1}, \eqref{eq:smooth2}, \eqref{eq:smooth3}, \eqref{eq:smooth4}--\eqref{eq:smooth8}
yield Proposition \ref{prop:estg1}, and so its proof is completed. \qed

Before   the proof of Propositions \ref{prop:contreform} and \ref{prop:FGR} we need an analogue  of the coercivity results in Sect.\  5 \cite{KMM3}.

\section{  Coercivity results }\label{sec:coerc}

Our main aim is to prove the following.
\begin{proposition}\label{lem:coer}
We have
\begin{align}\label{eq:coer}
\|w \|_{L^2_{-\frac{a}{10}}}\lesssim \|\xi \|_{\widetilde{\Sigma}}+e^{-\frac{B}{20}}\|  w '\|_{L^2}  .
\end{align}
\end{proposition}

Before proving Proposition \ref{lem:coer} we consider the following partial inversion of \eqref{def:vBg}, which is our analogue of Formula (62) in \cite{KMM3}.

\begin{lemma}\label{lem:coer6}
We have
\begin{align}\label{eq:Tinverse}
P_c\(\chi_{B^2}\widetilde{\eta}\)=\prod_{j=1}^{N}R_H(\omega_j) P_c \mathcal{A} \<   \im \varepsilon\partial_x\>^N v .
\end{align}
\end{lemma}

\begin{proof}
We first claim
\begin{align}\label{eq:Tinverse1}
\mathcal{A} \mathcal{A}^*=A_1\circ \cdots \circ A_N \circ A_N^* \circ \cdots \circ A_1^* =\prod_{j=1}^{N}(H-\omega_j).
\end{align}
Then, using \eqref{eq:Tinverse1}, from \eqref{def:Tg} and \eqref{def:vBg} we have
\begin{align*}
\prod_{j=1}^{N}R_H(\omega_j) P_c A_1\circ \cdots \circ A_N \< \im \varepsilon\partial_x\>^N v   &=\prod_{j=1}^{N}R_H(\omega_j) P_c A_1\circ \cdots \circ A_N \circ A_N^* \circ \cdots \circ A_1^* \chi_{B^2} \widetilde{\eta}\\&
=\prod_{j=1}^{N}R_H(\omega_j)P_c  \prod_{j=1}^{N}(H-\omega_j)\chi_{B^2}\widetilde{\eta}=P_c\(\chi_{B^2}\widetilde{\eta}\).
\end{align*}
Thus, it remains to prove \eqref{eq:Tinverse1}.
First, from \eqref{def:Ak}, we have
\begin{align*}
A_N\circ A_N^* =H_N-\omega_N.
\end{align*}
For $2\leq j \leq N$, we assume (notice that the Schr\"odinger operator $H_j$ is fixed)
\begin{align*}
A_j\circ \cdots \circ A_N \circ A_N^* \circ \cdots A_j^* = \prod_{k=j}^N(H_j-\omega_k).
\end{align*}
Then, by
\begin{align*}
A_{j-1}(H_j-\omega_k)  = A_{j-1}(A_{j-1}^*A_{j-1}+ \omega _{j-1}-\omega_k)&=  (A_{j-1}A_{j-1}^*+ \omega _{j-1}-\omega_k)    A_{j-1}  \\& =(H_{j-1}-\omega_k)A_{j-1},
\end{align*}
        we have
\begin{align*}&
A_{j-1}\circ \cdots \circ A_N \circ A_N^* \circ \cdots A_{j-1}^*  =A_{j-1} \prod_{k=j}^N(H_j-\omega_k)   A_{j-1}^*  =\prod_{k=j}^N(H _{j-1}-\omega_k) A_{j-1} \circ A_{j-1}^* \\& = \prod_{k=j}^N(H _{j-1}-\omega_k)   \   (H _{j-1}-\omega _{j-1})
=\prod_{k=j-1}^N(H _{j-1}-\omega_k).
\end{align*}
Therefore, we have \eqref{eq:Tinverse1} by induction.
\end{proof}

The proof of Lemma \ref{lem:coer7}  is postponed  to Appendix \ref{sec:comm}.
\begin{lemma}\label{lem:coer7}
We have
$\|  \prod_{j=1}^{N}R_H(\omega_j)P_c \mathcal{A} \< \im \varepsilon \partial_x\>^N \|_{L^2_{-\frac{a}{20}}\to L^2_{-\frac{a}{10}}}\lesssim 1$ uniformly for $0<\varepsilon\le  1$.
\end{lemma}
\qed

We continue this section, assuming Lemma \ref{lem:coer7}.
\begin{lemma}\label{lem:coer2}
We have
\begin{align*}
\|\chi_{B^2}\widetilde{\eta}\|_{L^2_{-\frac{a}{10}}}\lesssim \|v \|_{L^2_{-\frac{a}{20}}}+e^{-B}\|\widetilde{\eta}\|_{L^2_{-\frac{a}{10}}}.
\end{align*}
\end{lemma}

\begin{proof}
First,
\begin{align}\label{eq:coer2_1}
\|\chi_{B^2}\widetilde{\eta}\|_{L^2_{-\frac{a}{10}}}\leq \| e^{-\frac{a}{10}\<x\>}P_c(\chi_{B^2}\widetilde{\eta})\|_{L^2} + \|e^{-\frac{a}{10}\<x\>}P_d\(\chi_{B^2}\widetilde{\eta}\)\|_{L^2}.
\end{align}
Then, by Lemmas \ref{lem:coer6} and \ref{lem:coer7}, we have
\begin{align}\label{eq:coer2_2}
\| e^{-\frac{a}{10}\<x\>}P_c(\chi_{B^2}\widetilde{\eta})\|_{L^2} \lesssim \|    v\|_{L^2_{-\frac{a}{20}}}.
\end{align}
On the other hand, from $P_d\widetilde{\eta}=0$ and \eqref{def:PdPc}, we have
\begin{align*}
P_d(\chi_{B^2}\widetilde{\eta})=\sum_{j=1}^N(\chi_{B^2}\widetilde{\eta},\phi_j)\phi_j=\sum_{j=1}^N(\widetilde{\eta},\(\chi_{B^2}-1\)\phi_j)\phi_j.
\end{align*}
Then, since $\|e^{\frac{a}{10}\<x\>}(\chi_{B^2}-1)\phi_j\|_{L^2}\lesssim e^{-\(a_1-\frac{a}{10}\) B^2}\lesssim e^{-B}$, we have
\begin{align}\label{eq:coer2_3}
 \|e^{-\frac{a}{10}\<x\>}P_d\(\chi_{B^2}\widetilde{\eta}\)\|_{L^2}\lesssim e^{-B}\|\widetilde{\eta}\|_{L^2_{-\frac{a}{10}}}.
\end{align}
By \eqref{eq:coer2_1}, \eqref{eq:coer2_2} and \eqref{eq:coer2_3} we have the conclusion.
\end{proof}

\textit{Proof of Proposition \ref{lem:coer}}.
First we split
\begin{align}\label{eq:coer1}
\|w \|_{L^2_{-\frac{a}{10}}}\leq \|\chi_{B ^2 } w\|_{L^2_{-\frac{a}{10}}}+\|(1-\chi_{B ^2 })e^{-\frac{a}{10}\<x\>}w\|_{L^2}.
\end{align}
For the 2nd term of r.h.s.\ of \eqref{eq:coer1}, using Corollary \ref{cor:rhoequiv}, we have
\begin{align}\label{eq:coer2}
\|(1-\chi_{B ^2 })e^{-\frac{a}{10}\<x\>}w\|_{L^2}\leq \|(1-\chi_{B ^2})e^{-\frac{a}{20}}\|_{L^\infty}\|e^{-\frac{a}{20}\<x\>}w\|_{L^2}\lesssim e^{-\frac{  aB ^2 }{20}} \|w\|_{\widetilde{\Sigma}}.
\end{align}
For the 1st term of the r.h.s.\ of \eqref{eq:coer1}, by $\|\zeta_A\|_{L^\infty}\leq 1$ and Lemma \ref{lem:coer2},
\begin{align}\nonumber
&\|\chi_{B ^2} w\|_{L^2_{-\frac{a}{10}}}\leq \|\chi_{B  ^2} e^{-\frac{a}{10}\<x\>} \widetilde{\eta}\|_{L^2}\lesssim \|v\|_{L^2_{-\frac{a}{20}}}+e^{-B}\|\widetilde{\eta}\|_{L^2_{-\frac{a}{10}}}\\&
\lesssim  \|\chi_{B } v\|_{L^2_{-\frac{a}{20}}}+\|\(1-\chi_{B }\) v\|_{L^2_{-\frac{a}{20}}} +e^{-B}\|\zeta_A^{-1}e^{-\frac{a}{20}\<x\>}\|_{L^\infty}\|e^{-\frac{a}{20}\<x\>}w\|_{L^2}\label{eq:coer3}
%\\&
%\lesssim  + e^{-B}\|w\|_{\widetilde{\Sigma}}.
\end{align}
From $A\gg a^{-1}$ and Corollary \ref{cor:rhoequiv}, the 3rd term of line \eqref{eq:coer3} can be bounded as
\begin{align}\label{eq:coer4}
e^{-B}\|\zeta_A^{-1}e^{-\frac{a}{20}\<x\>}\|_{L^\infty}\|e^{-\frac{a}{20}\<x\>}w\|_{L^2}\lesssim e^{-B} \|w\|_{\widetilde{\Sigma}}.
\end{align}
For the 2nd term of line  \eqref{eq:coer3}, by Lemma \ref{lem:coer3},
\begin{align}\label{eq:coer5}
\|\(1-\chi_{B }\) v\|_{L^2_{-\frac{a}{20}}}\leq \| e^{-\frac{a}{20}\<x\>}(1-\chi_{B })\|_{L^\infty}\|v\|_{L^2}\lesssim e^{-\frac{ B }{20}} \varepsilon^{-N}B^2 \|w\|_{\widetilde{\Sigma}}.
\end{align}
Finally, for the 1st term of line \eqref{eq:coer3}, by the definition of $\zeta_B$ in \eqref{eq:zeta}, see also the definition of $\chi$ in \eqref{def:chi}, and of $\xi $ in \eqref{def:wAxiB},   we have
\begin{align}
\|\chi_{B } v\|_{L^2_{-\frac{a}{20}}}&\leq    \| \zeta_B ^{-1}    \| _{L^\infty (|x|\le 2B}     \|\chi_{B } \zeta_B      v\|_{L^2_{-\frac{a}{20}} } =  \| \zeta_B ^{-1}    \| _{L^\infty (|x|\le 2B}     \| \xi \|_{L^2_{-\frac{a}{20}} }   \nonumber \\& \lesssim    \| \xi \|_{L^2_{-\frac{a}{20}} } \lesssim    \| \xi\|_{\widetilde{\Sigma}} ,\label{eq:coer6}
\end{align}
where in the last inequality we applied Lemma \ref{lem:equiv_rho}.
Collecting the estimates \eqref{eq:coer2}, \eqref{eq:coer4}, \eqref{eq:coer5} and \eqref{eq:coer6}   we have the conclusion.
\qed

\section{Proof of Proposition \ref{prop:FGR}: Fermi Golden Rule}
\label{sec:FGR}

We substitute $\widetilde{\mathbf{z}}=\im \boldsymbol{\varpi}(|\mathbf{z}|^2)\mathbf{z}$ in  \eqref{eq:discfund} and we make various simplifications.
The first, by  $\<f,\im f\>=0$ the left hand side of \eqref{eq:discfund} can be rewritten as
\begin{align}\label{eq:lFGR2}
\<\im D_{\mathbf{z}}\phi [\mathbf{z}](\dot { \mathbf{z}}+\im \boldsymbol{\varpi}(|\mathbf{z}|^2)\mathbf{z}),D_{\mathbf{z}}\phi[\mathbf{z}]\im \boldsymbol{\varpi}(|\mathbf{z}|^2)\mathbf{z}\>=\<\im D_{\mathbf{z}}\phi[\mathbf{z}] \dot { \mathbf{z}},D_{\mathbf{z}}\phi   [\mathbf{z}]\im \boldsymbol{\varpi}(|\mathbf{z}|^2)\mathbf{z}\>.
\end{align}
Next, we consider the 2nd term in the 2nd line of \eqref{eq:discfund}, which we rewrite as
\begin{align}\nonumber
\<\sum_{\mathbf{m}\in \mathbf{R}_{\mathrm{min}}}\mathbf{z}^{\mathbf{m}}G_{\mathbf{m}}+\mathcal{R}_{\mathrm{rp}}[\mathbf{z}],D_{\mathbf{z}}\phi[\mathbf{z}]\im \boldsymbol{\varpi}(|\mathbf{z}|^2)\mathbf{z}\>=&
\<\sum_{\mathbf{m}\in \mathbf{R}_{\mathrm{min}}}\mathbf{z}^{\mathbf{m}}G_{\mathbf{m}}+\mathcal{R}_{\mathrm{rp}}[\mathbf{z}],D_{\mathbf{z}}\phi [\mathbf{z}]\( \dot {\mathbf{z}}+\im \boldsymbol{\varpi}(|\mathbf{z}|^2)\mathbf{z}\)\>
\\&-\<\sum_{\mathbf{m}\in \mathbf{R}_{\mathrm{min}}}\mathbf{z}^{\mathbf{m}}G_{\mathbf{m}}+\mathcal{R}_{\mathrm{rp}}[\mathbf{z}],D_{\mathbf{z}}\phi [\mathbf{z}]\dot {\mathbf{z}}\>.\label{eq:lFGR1}
\end{align}
The   term in the 1st line  of the r.h.s.\ of \eqref{eq:lFGR1} can be written as
\begin{align}%
 %&\<\sum_{\mathbf{m}\in \mathbf{R}_{\mathrm{min}}}\mathbf{z}^{\mathbf{m}}G_{\mathbf{m}}+\mathcal{R}(\mathbf{z}),D_{\mathbf{z}}\phi(\mathbf{z})\(\dot {\mathbf{z}}+\im \boldsymbol{\varpi}(|\mathbf{z}|^2)\mathbf{z}\)\>=\\
 &\label{eq:lFGR7} \<\sum_{\mathbf{m}\in \mathbf{R}_{\mathrm{min}}}\mathbf{z}^{\mathbf{m}}G_{\mathbf{m}},D_{\mathbf{z}}\phi[0]\(   \dot {\mathbf{z}}+\im \boldsymbol{\varpi}(|\mathbf{z}|^2)\mathbf{z}\)\>+R_1(\mathbf{z}),
\end{align}
where
\begin{align*}
R_1(\mathbf{z})=&\<\sum_{\mathbf{m}\in \mathbf{R}_{\mathrm{min}}}\mathbf{z}^{\mathbf{m}}G_{\mathbf{m}},\(D_{\mathbf{z}}\phi    [\mathbf{z}]-D_\mathbf{z}\phi  [0]\)\(
\dot {\mathbf{z}}+\im \boldsymbol{\varpi}(|\mathbf{z}|^2)\mathbf{z}\)\>\\&+\<   \mathcal{R}_{\mathrm{rp}}[\mathbf{z}],D_{\mathbf{z}}\phi  [\mathbf{z}]\(\dot {\mathbf{z}} +\im \boldsymbol{\varpi}(|\mathbf{z}|^2)\mathbf{z}\)\>,\nonumber
\end{align*}
by \eqref{est:Rrp},   inequalities \eqref{eq:main1} and \eqref{eq:main11},     Proposition \ref{lem:estdtz} and  $ \| D_{\mathbf{z}}\phi    [\mathbf{z}]-D_\mathbf{z}\phi  [0]\| _{H^1} =O(\| \mathbf{z} \| ^2) $  by \eqref{def:refpexp},
satisfies
\begin{align}\label{eq:lFGR4}
\int_0^T |R_1(\mathbf{z}(t))|\,dt\lesssim   \delta ^2  \epsilon ^2.
\end{align}
Using  the stationary Refined Profile equation \eqref{eq:rp}, the last line of \eqref{eq:lFGR1} can be written as
\begin{align}\label{eq:lFGR5}
%-\<\sum_{\mathbf{m}\in \mathbf{R}_{\mathrm{min}}}\mathbf{z}^{\mathbf{m}}G_{\mathbf{m}}+\mathcal{R}(\mathbf{z}),D_{\mathbf{z}}\phi(\mathbf{z})\dot {\mathbf{z}}\>=
-\<H\phi[\mathbf{z}]+ |\phi  [\mathbf{z}]|^2 \phi  [\mathbf{z}],D_{\mathbf{z}}\phi   [\mathbf{z}] \dot {\mathbf{z}}\>
+\<D_{\mathbf{z}}\phi(\mathbf{z})(\im \boldsymbol{\varpi}(|\mathbf{z}|^2))\mathbf{z} ,\im  D_{\mathbf{z}}\phi  [\mathbf{z}]\dot {\mathbf{z}} \>.%\nonumber
\end{align}
Notice that the 2nd term of \eqref{eq:lFGR5} coincides  with the right hand side  of \eqref{eq:lFGR2}, which lies in the left hand side of \eqref{eq:discfund}, so that the two cancel each other.
 On the other hand,  we have
\begin{align}\label{eq:lFGR6}
\<H\phi[\mathbf{z}]+ |\phi  [\mathbf{z}]|^2 \phi  [\mathbf{z}] ,D_{\mathbf{z}}\phi[\mathbf{z}] \dot {\mathbf{z}}\>=\frac{d}{dt}E(\phi   [\mathbf{z}]).
\end{align}
Therefore, from \eqref{eq:discfund} with $\widetilde{\mathbf{z}}=\im \boldsymbol{\varpi}(|\mathbf{z}|^2)\mathbf{z}$, \eqref{eq:lFGR2}, \eqref{eq:lFGR1}, \eqref{eq:lFGR7}, \eqref{eq:lFGR5} and \eqref{eq:lFGR6}, we have
\begin{align}\label{eq:dtE1}
\frac{d}{dt}E(\phi [\mathbf{z}] )-\sum_{\mathbf{m}\in \mathbf{R}_{\mathrm{min}}}\mathbf{m}\cdot \boldsymbol{\omega}\<\eta,\im  \mathbf{z}^{\mathbf{m}}G_{\mathbf{m}}\>=\sum_{\mathbf{m}\in \mathbf{R}_{\mathrm{min}}}\<\mathbf{z}^{\mathbf{m}}G_{\mathbf{m}},D_{\mathbf{z}}\phi   [0]\(\dot {\mathbf{z}}+\im \boldsymbol{\varpi}(|\mathbf{z}|^2)\mathbf{z}\)\>+R_2(\mathbf{z},\eta),
\end{align}
where
\begin{align}\label{eq:R2}&
R_2(\mathbf{z},\eta)= R_1(\mathbf{z})+\<\im \eta, D_{\mathbf{z}}^2\phi  [\mathbf{z}]\(   \dot {\mathbf{z}} +\im \boldsymbol{\varpi}(|\mathbf{z}|^2)\mathbf{z},\im \boldsymbol{\varpi}(|\mathbf{z}|^2)\mathbf{z}\)\>
+\<\eta,D_{\mathbf{z}}\mathcal{R}_{\mathrm{rp}}[\mathbf{z}]\im \boldsymbol{\varpi}(|\mathbf{z}|^2)\mathbf{z}\>\\&
+\sum_{\mathbf{m}\in \mathbf{R}_{\mathrm{min}}}(\boldsymbol{\varpi}(|\mathbf{z}|^2)-\boldsymbol{\omega})\<\eta,\mathbf{z}^{\mathbf{m}}G_{\mathbf{m}}\>+\< L[\mathbf{z}]\eta +F(\mathbf{z},\eta)+|\eta |^2 \eta ,D_{\mathbf{z}}\phi [\mathbf{z}]\im \boldsymbol{\varpi}(|\mathbf{z}|^2)\mathbf{z}\>  \nonumber
\end{align}
  satisfies
\begin{align}\label{eq:estR2}
\int_I|R_2(\mathbf{z}(t),\eta(t))|\,dt\lesssim   \delta ^2 \epsilon ^2.
\end{align}

We consider the first term in the right hand side of \eqref{eq:dtE1}. By  \eqref{def:refpexp}  we have $D_{\mathbf{z}}\phi [0]\widetilde{\mathbf{z}}= \boldsymbol{\phi}\cdot \widetilde{\mathbf{z}}  $,
  this term is the left hand side of \eqref{eq:Fermi-1} below.
\begin{lemma} \label{prop:Fermi-1}    We have
	 \begin{align}\label{eq:Fermi-1}
	  \sum_{\mathbf{m}\in \mathbf{R}_{\mathrm{min}}}\<\mathbf{z}^{\mathbf{m}}G_{\mathbf{m}},\boldsymbol{\phi}\cdot\( \dot { \mathbf{z}}+\im \boldsymbol{\varpi}(|\mathbf{z}|^2)\mathbf{z}\)\>
=\partial_t A_1(\mathbf{z})+R_4(\mathbf{z},\eta)
	\end{align}
where:
\begin{align}  \label{eq:Fermi-2}
A_1(\mathbf{z}) =
\sum_{\substack{\mathbf{m},\mathbf{n}\in \mathbf{R}_{\mathrm{min}}\\ \mathbf{m}\neq \mathbf{n}}}\sum_{j=1}^N
\frac{1}{
\(\mathbf{n}-\mathbf{m}\)\cdot \boldsymbol{\omega}
}
\mathrm{Re}(\mathbf{z}^{\mathbf{m}}\overline{\mathbf{z}^{\mathbf{n}}})
g_{\mathbf{m},j}g_{\mathbf{n},j}, \end{align}
for $g_{\mathbf{m},j}:=\<G_{\mathbf{m}},\phi_j\>$;
\begin{align}   \label{eq:Fermi-3}&
 R_4(\mathbf{z},\eta)=R_3(\mathbf{z})+\sum_{\mathbf{m},\mathbf{n}\in \mathbf{R}_{\mathrm{min}}}\sum_{j=1}^N\<\mathbf{z}^{\mathbf{m}}G_{\mathbf{m}}, r_j(\mathbf{z},\eta)\phi_j\> \text{ where}
\\   \label{eq:Fermi-4}&R_3(\mathbf{z}) =\sum_{\substack{\mathbf{m},\mathbf{n}\in \mathbf{R}_{\mathrm{min}}\\ \mathbf{m}\neq \mathbf{n}}}\sum_{j=1}^N\mathrm{Re}\(\im r_{\mathbf{n},\mathbf{m}}(\mathbf{z})\)g_{\mathbf{m},j}g_{\mathbf{n},j} \text{  for}\\& \label{eq:Fermi-5} r_{\mathbf{n},\mathbf{m}}(\mathbf{z})=
	-\frac{(\mathbf{m}-\mathbf{n})\cdot \(\boldsymbol{\varpi}(|\mathbf{z}|^2)-\boldsymbol{\omega}\)}{ (\mathbf{m}-\mathbf{n})\cdot \boldsymbol{\omega}}  \mathbf{z}^{\mathbf{n}}\overline{\mathbf{z}^{\mathbf{m}}}
	\\&+\frac{\im}{ (\mathbf{m}-\mathbf{n})\cdot \boldsymbol{\omega}}\(D_{\mathbf{z}}(\mathbf{z}^{\mathbf{n}})(\dot { \mathbf{z}}+\im \boldsymbol{\varpi}(|\mathbf{z}|^2\mathbf{z}))\overline{\mathbf{z}^{\mathbf{m}}}+
		\mathbf{z}^{\mathbf{n}}\overline{D_{\mathbf{z}}(\mathbf{z}^{\mathbf{m}})(( \dot { \mathbf{z}}+\im \boldsymbol{\varpi}(|\mathbf{z}|^2\mathbf{z})))}\) ; \nonumber
	\end{align}
we have  \begin{align}\label{eq:Fermi-6}
\int_I |R_4(\mathbf{z}(t),\eta(t))|\,dt\lesssim    \delta ^2 \epsilon ^2.
\end{align}

\end{lemma}
\proof  The left hand side of  \eqref{eq:Fermi-1}  equals  \begin{align*}
& \sum_{\mathbf{m},\mathbf{n}\in \mathbf{R}_{\mathrm{min}}}\sum_{j=1}^N\<\mathbf{z}^{\mathbf{m}}G_{\mathbf{m}}, \phi_j\(-\im \mathbf{z}^{\mathbf{n}}g_{\mathbf{n},j}+r_j(\mathbf{z},\eta)\)\>\nonumber\\&
=\sum_{\substack{\mathbf{m},\mathbf{n}\in \mathbf{R}_{\mathrm{min}}\\ \mathbf{m}\neq \mathbf{n}}}\sum_{j=1}^N\mathrm{Re}\(\im \mathbf{z}^{\mathbf{m}}\overline{\mathbf{z}^{\mathbf{n}}}\)g_{\mathbf{m},j}g_{\mathbf{n},j}
+\sum_{\mathbf{m},\mathbf{n}\in \mathbf{R}_{\mathrm{min}}}\sum_{j=1}^N\<\mathbf{z}^{\mathbf{m}}G_{\mathbf{m}}, r_j(\mathbf{z},\eta)\phi_j\>,
\end{align*}
used the fact that $\<\mathbf{z}^{\mathbf{m}}G_{\mathbf{m}},-\im \mathbf{z}^{\mathbf{m}}\phi_j\>=0$ due to $\phi_j$ and  $G_{\mathbf{m}}$   being $\R$ valued, see \cite{CM20}.

\noindent Since $(\mathbf{m-n})\cdot \boldsymbol{\omega}\neq 0$ for $\mathbf{m}\neq \mathbf{n}$
  by Assumption \ref{ass:disc_ratio_indep}, we have
	\begin{align}\label{znmnormal}
	\mathbf{z}^{\mathbf{n}}\overline{\mathbf{z}^{\mathbf{m}}}=\frac{1}{\im ((\mathbf{m}-\mathbf{n})\cdot \boldsymbol{\omega})}\partial_t(\mathbf{z}^{\mathbf{n}}\overline{\mathbf{z}^{\mathbf{m}}})
	+r_{\mathbf{n},\mathbf{m}}(\mathbf{z}),
	\end{align}
with $r_{\mathbf{n},\mathbf{m}}(\mathbf{z})$ given by  \eqref{eq:Fermi-5} and satisfying
\begin{align}\label{est:rmn}
\int_I |r_{\mathbf{m},\mathbf{n}}(\mathbf{z})|\,dt\lesssim    \delta ^2\epsilon ^2.
\end{align}
We have
\begin{align*}
\sum_{\substack{\mathbf{m},\mathbf{n}\in \mathbf{R}_{\mathrm{min}}\\ \mathbf{m}\neq \mathbf{n}}}\sum_{j=1}^N\mathrm{Re}\(\im \mathbf{z}^{\mathbf{m}}\overline{\mathbf{z}^{\mathbf{n}}}\)g_{\mathbf{m},j}g_{\mathbf{n},j}=
\partial_tA_1(\mathbf{z})+R_3(\mathbf{z}).
\end{align*}
for $A_1(\mathbf{z})$  and $ R_3(\mathbf{z})$ defined above. Finally \eqref{eq:Fermi-6} is straightforward.

\qed

 We focus now on \eqref{eq:dtE1}.
\begin{lemma} \label{prop:Fermi}
There exists a constant   $ \Gamma _0>0$     such that
	 \begin{align}\label{eq:Fermi1}
	  \sum_{\mathbf{m}\in \mathbf{R}_{\mathrm{min}}}\mathbf{m}\cdot \boldsymbol{\omega}\<\eta,\im  \mathbf{z}^{\mathbf{m}}G_{\mathbf{m}}\>  \le - \Gamma _0\sum_{\mathbf{m}\in \mathbf{R}_{\mathrm{min}}}|\mathbf{z}^{\mathbf{m}}| ^2 + \mathcal{E}_1 + \mathcal{E}_2+\mathcal{E}_3,
	\end{align}
where  for some   constants $ c_{\mathbf{m},\mathbf{n}}$  the term  $\mathcal{E}_1$  is of the form
\begin{align*} & \mathcal{E}_1=    \sum_{\substack{\mathbf{m}, \mathbf{n}\in \mathbf{R}_{\mathrm{min}}  \\ \mathbf{m}\neq \mathbf{n}}}   c_{\mathbf{m},\mathbf{n}}   \mathbf{z} ^{ \mathbf{n}} \overline{ \mathbf{z} ^{\mathbf{m} }}  ,
	\end{align*}
\begin{align*} &  \mathcal{E}_2=    \sum_{\mathbf{m}\in \mathbf{R}_{\mathrm{min}}}  \boldsymbol{\omega} \cdot \mathbf{m} \<\im  \mathbf{z}^{\mathbf{m}}  \< \im \varepsilon  \partial _x \> ^{ N} \mathcal{A}^*  \prod _{j=1}^{N}R _{H}( \omega _j)    P_c   G_\mathbf{m},   {g} \> ,
	\end{align*}
\begin{align*} & | \mathcal{E}_3  |\le  o_{\varepsilon}(1) \(  \sum_{\mathbf{m}\in \mathbf{R}_{\mathrm{min}}}|\mathbf{z}^{\mathbf{m}}| ^2  + \|   w\|_{\widetilde{\Sigma}}^2        \) .
	\end{align*}

\end{lemma}
\proof   First of all, notice that
\begin{align*}&
   \eta =P_c  {\eta} +P_d  (R[\mathbf{z}] -1)P_c  {\eta}   ,
\end{align*}
where the following term can be absorbed in $\mathcal{E}_3$,
\begin{align*}&
 \left |  \<\im   \boldsymbol{\omega} \cdot \mathbf{m}\mathbf{z}^{\mathbf{m}}      G_\mathbf{m},P_d  (R[\mathbf{z}] -1)P_c  {\eta}\> \right |
  \lesssim  \delta |\mathbf{z}^{\mathbf{m}}|   \|   w\|_{\widetilde{\Sigma}}
\end{align*}
Now we consider
\begin{align*}&
   \<\im   \boldsymbol{\omega} \cdot \mathbf{m}\mathbf{z}^{\mathbf{m}}      G_\mathbf{m},P_c\eta\>  =  \<\im   \boldsymbol{\omega} \cdot \mathbf{m}\mathbf{z}^{\mathbf{m}}      G_\mathbf{m}, P_c \chi _{B^2}\eta\> +  \<\im   \boldsymbol{\omega} \cdot \mathbf{m}\mathbf{z}^{\mathbf{m}}      G_\mathbf{m}, P_c \( 1- \chi _{B^2}\) \eta\> .
\end{align*}
Then, by \eqref{eq:Tinverse}         we have
\begin{align}&
    \<\im   \boldsymbol{\omega} \cdot \mathbf{m}\mathbf{z}^{\mathbf{m}}      G_\mathbf{m}, P_c \chi _{B^2}\eta\> =
       \<\im   \boldsymbol{\omega} \cdot \mathbf{m}\mathbf{z}^{\mathbf{m}}      G_\mathbf{m},\prod _{j=1}^{N}R _{H}( \omega _j) P_c \mathcal{A}  \< \im \varepsilon   \partial _x \> ^{ N} v \>  \nonumber \\& = \sum_{\mathbf{m}\in \mathbf{R}_{\mathrm{min}}}    \boldsymbol{\omega} \cdot \mathbf{m}|\mathbf{z}^{\mathbf{m}} |^2 \<\im        G_\mathbf{m},\prod _{j=1}^{N}R _{H}( \omega _j) P_c \mathcal{A} \< \im \varepsilon  \partial _x \> ^{ N}  {\rho} _{\mathbf{m}} \> \label{mainFGR1}  \\& +  \sum_{\mathbf{m}\in \mathbf{R}_{\mathrm{min}}}  \boldsymbol{\omega} \cdot \mathbf{m} \<\im  \mathbf{z}^{\mathbf{m}}      G_\mathbf{m},\prod _{j=1}^{N}R _{H}( \omega _j) P_c \mathcal{A} \< \im \varepsilon   \partial _x \> ^{ N}  {g} \> \label{remainFGR1}\\& + \sum_{\substack{\mathbf{m}, \mathbf{n}\in \mathbf{R}_{\mathrm{min}}  \\ \mathbf{m}\neq \mathbf{n}}}   \boldsymbol{\omega} \cdot \mathbf{m}  \<\im   \mathbf{z}^{\mathbf{m}}      G_\mathbf{m}, \mathbf{z}^{\mathbf{n}}\prod _{j=1}^{N}R _{H}( \omega _j) P_c \mathcal{A} \< \im \varepsilon   \partial _x \> ^{ N} {\rho} _{\mathbf{n}} \> . \label{remainFGR2}
\end{align}
Obviously the remainder term in line \eqref{remainFGR1} can be  absorbed  in $\mathcal{E}_2$ and the remainder term in line \eqref{remainFGR2} can be  absorbed  in $\mathcal{E}_1$. We now examine the main term, in the line \eqref{mainFGR1}.   We have
\begin{align}&\nonumber  \<\im        G_\mathbf{m},\prod _{j=1}^{N}R _{H}( \omega _j) P_c \mathcal{A} \< \im \varepsilon  \partial _x \> ^{ N}  {\rho} _{\mathbf{m}} \>  \\&=  - \<\im      \<  \im \varepsilon  \partial _x \> ^{ N}   \mathcal{A}^*\prod _{j=1}^{N}R _{H}( \omega _j) P_c G_\mathbf{m},    R ^{+}_{H _{N+1} }(\boldsymbol{\omega} \cdot \mathbf{m}) \< \im \varepsilon \partial _x \> ^{-N}  \mathcal{A}^* \chi  _{B^2}G_{\mathbf{m}}\>   \nonumber \\& = - \<\im         \mathcal{A}^* \prod _{j=1}^{N}R _{H}( \omega _j) P_c G_\mathbf{m},    R ^{+}_{H _{N+1} }(\boldsymbol{\omega} \cdot \mathbf{m})  \mathcal{A}^* \chi  _{B^2}G_{\mathbf{m}}\>   \label{mainFGR4}\\& \quad   - \<\im      \< \im \varepsilon  \partial _x \> ^{ N}   \mathcal{A}^* \prod _{j=1}^{N}R _{H}( \omega _j) P_c G_\mathbf{m},  \left [  R ^{+}_{H _{N+1} }(\boldsymbol{\omega} \cdot \mathbf{m}) ,\< \im \varepsilon  \partial _x \> ^{-N} \right ]  \mathcal{A}^*\chi  _{B^2}G_{\mathbf{m}}\> ,\label{remainFGR4-}
\end{align}
  where we see now that the quantity in \eqref{remainFGR4-} is a $o_{\varepsilon}(1)$ and its contribution to \eqref{eq:Fermi1} can be absorbed in $\mathcal{E}_3$.  Indeed the quantity in \eqref{remainFGR4-}  can be bounded by the product $ A \  B,$ where
  \begin{align*}&    A=
   \|    \< \im \varepsilon  \partial _x \> ^{ N}   \mathcal{A}^* \prod _{j=1}^{N}R _{H}( \omega _j) P_c G_\mathbf{m}\| _{L^{2,\ell}} \text{  and} \\& B= \|   R ^{+}_{H _{N+1} }(\boldsymbol{\omega} \cdot \mathbf{m})  \left [ V _{N+1}   ,\< \im \varepsilon  \partial _x \> ^{-N} \right ]  R ^{+}_{H _{N+1} }(\boldsymbol{\omega} \cdot \mathbf{m}) \mathcal{A}^*\chi  _{B^2}G_{\mathbf{m}} \| _{L^{2,-\ell}} ,
\end{align*}
  for $\ell \ge 2$.   We have
  \begin{align*}    B\le &  \|   R ^{+}_{H _{N+1} }(\boldsymbol{\omega} \cdot \mathbf{m}) \| _{L^{2, \ell}\to L^{2,-\ell}}^2
    \| \< \im \varepsilon  \partial _x \> ^{-N}  \left [ V _{N+1}   ,\< \im \varepsilon  \partial _x \> ^{ N} \right ]
     \| _{L^{2, -\ell}\to L^{2, \ell}}   \\& \times \|   \< \im \varepsilon  \partial _x \> ^{-N} \| _{L^{2,- \ell}\to L^{2,-\ell}}
     \| \mathcal{A}^*\chi  _{B^2}G_{\mathbf{m}} \| _{L^{2, \ell}}  \lesssim \varepsilon ,
\end{align*}
where the $\varepsilon$ comes from the commutator term in the first line, by Lemma \ref{claim:l2boundII}, while the other terms are uniformly bounded, with $\|   \< \im \varepsilon  \partial _x \> ^{-N} \| _{L^{2,- \ell}\to L^{2,-\ell}} \lesssim 1$  uniformly in $\varepsilon \in (0,1]$, because we have an operator like in
\eqref{eq:opT}.  On the other hand,   uniformly in $\varepsilon \in (0,1]$, we have
\begin{align*}    A\le &  \|  \< \im \varepsilon  \partial _x \> ^{ N}   \< \im   \partial _x \> ^{-2 N}  \| _{L^{2, \ell} \to L^{2, \ell}}  \|      \< \im   \partial _x \> ^{ 2 N}  \mathcal{A}^* \prod _{j=1}^{N}R _{H}( \omega _j) P_c G_\mathbf{m} \| _{L^{2, \ell} } \lesssim 1.
\end{align*}
  We consider  the main term \eqref{mainFGR4}. Essentially by \eqref{eq:DarConj2}, it equals
  \begin{align*}&   - \<\im       \mathcal{A} ^*\prod _{j=1}^{N}R _{H}( \omega _j) P_c G_\mathbf{m},     \mathcal{A}^*   R ^{+}_{H   }(\boldsymbol{\omega} \cdot \mathbf{m})\chi  _{B^2}G_{\mathbf{m}}\>   \\&= - \<\im        \mathcal{A} \mathcal{A} ^*\prod _{j=1}^{N}R _{H}( \omega _j) P_c G_\mathbf{m},        R ^{+}_{H   }(\boldsymbol{\omega} \cdot \mathbf{m})\chi  _{B^2}G_{\mathbf{m}}\>  = - \<\im         P_c G_\mathbf{m},      R ^{+}_{H   }(\boldsymbol{\omega} \cdot \mathbf{m})\chi  _{B^2}G_{\mathbf{m}}\>  ,
\end{align*}
  where we used   \eqref{eq:Tinverse1}.  By the Limit Absorption Principle    and the Sokhotski--Plemelj Formula, the   last term equals
 \begin{align}&  - \<\im         P_c G_\mathbf{m},      R ^{+}_{H   }(\boldsymbol{\omega} \cdot \mathbf{m}) G_{\mathbf{m}}\>     + \<\im         P_c G_\mathbf{m},      R ^{+}_{H   }(\boldsymbol{\omega} \cdot \mathbf{m})\( 1-\chi  _{B^2}\) G_{\mathbf{m}}\>  \nonumber \\  =&  -\pi \<         P_c G_\mathbf{m},       \delta ({H   }-\boldsymbol{\omega} \cdot \mathbf{m}) G_{\mathbf{m}}\>   \label{mainFGR5}  \\&  + \<\im         P_c G_\mathbf{m},      R ^{+}_{H   }(\boldsymbol{\omega} \cdot \mathbf{m})\( 1-\chi  _{B^2}\) G_{\mathbf{m}}\> , \label{remainFGR5}
\end{align}
 where the quantity \eqref{remainFGR5} is  of the form $O(B^{-1})$ and so also of the form $o_{\varepsilon}(1)$ and the corresponding contribution to \eqref{eq:Fermi1} can be absorbed in $\mathcal{E}_3$. Finally, by elementary computation we have
\begin{align}&    -\pi \<         P_c G_\mathbf{m},       \delta ({H   }-\boldsymbol{\omega} \cdot \mathbf{m}) G_{\mathbf{m}}\> = -\frac{\pi}{2\sqrt{\boldsymbol{\omega} \cdot \mathbf{m}}}  \( \left | \widehat{G}_\mathbf{m} (\sqrt{\boldsymbol{\omega} \cdot \mathbf{m}})   \right | ^2+ \left | \widehat{G}_\mathbf{m} (-\sqrt{\boldsymbol{\omega} \cdot \mathbf{m}})   \right | ^2 \) < 0  ,   \label{mainFGR6}
\end{align}
 where $\widehat{G}_\mathbf{m}$  is the distorted Fourier transform associated to the operator $H$ and where the inequality follows from Assumption \ref{ass:FGR}. The corresponding contribution to \eqref{eq:Fermi1} can be absorbed in the first term in the right hand side.

 \qed

\begin{lemma} \label{prop:Fermi2}   For a constant $C _{V,\Gamma _0 }>0$  we have
\begin{align}\label{eq:Fermi21}&
	  \int _{I}\mathcal{E}_1 dt \lesssim     \delta ^2\epsilon ^2, \\& \label{eq:Fermi22} \int _{I}\mathcal{E}_2 dt \le   2^{-1}\Gamma _0\sum_{\mathbf{m}\in \mathbf{R}_{\mathrm{min}}}   \| \mathbf{z}^{\mathbf{m}}\| ^{2} _{L^2(I)} + C _{V,\Gamma _0 } \(  \varepsilon ^{-N}B ^{4+4\tau}\delta  ^2 +   B ^{- 1}  \epsilon ^2+   \varepsilon ^4 \) ,\\& \label{eq:Fermi23} \int _{I}\mathcal{E}_3 dt \lesssim     o_{\varepsilon}(1) \epsilon ^2.
	\end{align}
\end{lemma}
\proof  Inequality \eqref{eq:Fermi23} is straightforward and so is \eqref{eq:Fermi21},  thanks to  \eqref{znmnormal} and \eqref{est:rmn}.

\noindent  Turning to \eqref{eq:Fermi22},     we have, for  constants  $C _{V,\Gamma _0}' >0 $ and  $C _{V,\Gamma _0 } >0,$
\begin{align*} &
\int_I  \mathcal{E}_3\,dt\lesssim     2^{-1}\Gamma _0\sum_{\mathbf{m}\in \mathbf{R}_{\mathrm{min}}}   \| \mathbf{z}^{\mathbf{m}}\| ^{2} _{L^2(I)}\\& + C _{V,\Gamma _0} '
\| {g} \| _{L^2(I, L ^{2,-S} (\R ))}^2  \sup _{\mathbf{m}\in \mathbf{R}_{\mathrm{min}}} \|  \< \im \varepsilon  \partial _x \> ^{ N} \mathcal{A}^*  \prod _{j=1}^{N}R _{H}( \omega _j)    P_c   G_\mathbf{m}\| _{  L ^{2,-S} (\R  )} \\&  \le  2^{-1}\Gamma _0\sum_{\mathbf{m}\in \mathbf{R}_{\mathrm{min}}}   \| \mathbf{z}^{\mathbf{m}}\| ^{2} _{L^2(I)} + C _{V,\Gamma _0 } \(  \varepsilon ^{-N}B ^{4+4\tau}\delta  ^2 +   B ^{- 1}  \epsilon ^2+   \varepsilon ^4 \).
\end{align*}
  \qed

\textit{Conclusion of the proof of Proposition \ref{prop:FGR}.} From \eqref{eq:dtE1}, \eqref{eq:Fermi-1} and \eqref{eq:Fermi1}, we have
\begin{align} &
\frac{d}{dt}E(\phi [\mathbf{z}] )= \sum_{\mathbf{m}\in \mathbf{R}_{\mathrm{min}}}\mathbf{m}\cdot \boldsymbol{\omega}\<\eta,\im  \mathbf{z}^{\mathbf{m}}G_{\mathbf{m}}\> + \partial_t A_1(\mathbf{z})+R_4(\mathbf{z},\eta) +  R_2(\mathbf{z},\eta) \label{energy:stab1}
\\& \le   - \Gamma _0\sum_{\mathbf{m}\in \mathbf{R}_{\mathrm{min}}}|\mathbf{z}^{\mathbf{m}}| ^2 + \mathcal{E}_1 + \mathcal{E}_2+\mathcal{E}_3+ \partial_t A_1(\mathbf{z})+R_4(\mathbf{z},\eta) +  R_2(\mathbf{z},\eta).\nonumber
\end{align}
So, integrating and using \eqref{eq:estR2}, \eqref{eq:Fermi-6} and Lemma \ref{prop:Fermi2}, we have
\begin{align*} &  2^{-1} \Gamma _0\sum_{\mathbf{m}\in \mathbf{R}_{\mathrm{min}}}\|\mathbf{z}^{\mathbf{m}}\| ^2 _{L^2(I)}
\le   \left . \( A_1(\mathbf{z}) - E(\phi [\mathbf{z}] ) \) \right ] _{0} ^{T} + \int_I\( |R_2(\mathbf{z} ,\eta )| +|R_4(\mathbf{z} ,\eta )|\)    \,dt \\& + C _{V,\Gamma _0 } \(  \varepsilon ^{-N}B ^{2+2\tau}\delta  ^2 +   B ^{- 1}  \epsilon ^2+  \epsilon ^4 \)   .
\end{align*}
From $ \left . \( A_1(\mathbf{z}) - E(\phi [\mathbf{z}] ) \) \right ] _{0} ^{T}=O(\delta ^2)$, we conclude
\begin{align*} &  \sum_{\mathbf{m}\in \mathbf{R}_{\mathrm{min}}}\|\mathbf{z}^{\mathbf{m}}\| ^2 _{L^2(I)}
\lesssim \varepsilon ^{-N}B ^{4+4\tau}\delta  ^2 +   B ^{- 1}  \epsilon ^2+   \varepsilon ^4   .
\end{align*}
This completes the proof of Proposition \ref{prop:FGR}.
\qed

\section{  Proof of Proposition \ref{prop:contreform} }\label{sec:conclusion}

By \eqref{eq:FGRint} and by  the relation between $A,B,\varepsilon , \epsilon $ and $\delta$ in \eqref{eq:relABg},  we have
\begin{align}
\sum_{\mathbf{m}\in \mathbf{R}_{\mathrm{min}} }\|\mathbf{z}^{\mathbf{m}}\|_{L^2(I)}\lesssim   \varepsilon ^{-N}B ^{2+2\tau}\delta +B ^{-\frac{1}{2}}   \epsilon   + \epsilon ^2  \lesssim o_\varepsilon(1) \epsilon .\label{FGReqfinal}
\end{align}
Inserting this in \eqref{eq:lem:estdtz} we obtain
 \begin{align}
 \|\dot {\mathbf{z}}+\im \varpi(\mathbf{z})\mathbf{z}\|_{L^2}\lesssim \sum_{\mathbf{m}\in \mathbf{R}_{\mathrm{min}}}\|\mathbf{z}^{\mathbf{m}}\|_{L^2}+ \delta^2 \epsilon  \lesssim o_\varepsilon(1) \epsilon.   \label{eq:lem:estdtzfin}
 \end{align}
By   \eqref{eq:1stv}, \eqref{eq:coer} and  \eqref{eq:lem:estdtzfin}
\begin{align*} &
 \|  w '\|_{L^2 L^2}\lesssim A^{1/2}\delta  +\|w \|_{L^2L^2_{-\frac{a}{10}}}+\sum_{\mathbf{m}\in \mathbf{R}_{\min}}\|\mathbf{z}^{\mathbf{m}}\|_{L^2}+ \epsilon ^2   \lesssim  o_\varepsilon(1) \epsilon +\| \xi \| _{\widetilde{\Sigma}} + o_\varepsilon(1)\|  w '\|_{L^2 L^2},
 \end{align*}
so that
\begin{align} &
 \|  w '\|_{L^2 L^2}\lesssim   o_\varepsilon(1) \epsilon +\| \xi \| _{\widetilde{\Sigma}} .\label{eq:1stvaf}
 \end{align}
 By  \eqref{eq:2ndv}, \eqref{eq:coer},  \eqref {FGReqfinal}--\eqref {eq:1stvaf}
 \begin{align*} &
\|\xi \|_{L^2 \widetilde{\Sigma}}\lesssim B   \varepsilon^{-N}    \delta   + \sum_{\mathbf{m}\in \mathbf{R}_{\mathrm{min}}} \|\mathbf{z}^{\mathbf{m}}\|_{L^2}+ o_{\varepsilon}(1) \(   \| w  \| _{L^2\widetilde{\Sigma}}    +   \|\dot {\mathbf{z}}+\im \boldsymbol{\varpi}(\mathbf{z})\mathbf{z}\|  _{L^2} \)  \\&   \lesssim o_\varepsilon(1) \epsilon   + o_\varepsilon(1)  \| \xi \| _{\widetilde{\Sigma}}
\end{align*}
 which implies
 \begin{align}\label{eq:2ndvfin}
\|\xi \|_{L^2 \widetilde{\Sigma}}\lesssim    o_\varepsilon(1) \epsilon ,
\end{align}
which fed in \eqref{eq:1stvaf} yields
\begin{align} &
 \|  w '\|_{L^2 L^2}\lesssim   o_\varepsilon(1) \epsilon   .\label{eq:1stvfin}
 \end{align}
Obviously, \eqref{FGReqfinal}, \eqref{eq:lem:estdtzfin},  \eqref{eq:2ndvfin}, \eqref{eq:1stvfin}  and \eqref{eq:coer} imply
\begin{align}&
\|\dot {\mathbf{z}}+\im \boldsymbol{\varpi}(\mathbf{z})\mathbf{z}\|_{L^2(I)}+\sum_{\mathbf{m}\in \mathbf{R}_{\mathrm{min}}}\|\mathbf{z}^\mathbf{m}\|_{L^2(I)}+ \| \xi  \|_{L^2(I, \widetilde{\Sigma})}+\|  w \|_{L^2(I, \widetilde{\Sigma})}\le   o_\varepsilon(1)\epsilon  . \label{eq:main12}
\end{align}
In the above discussion we can take $\varepsilon = \varepsilon (\delta )$  with $\varepsilon (\delta ) \xrightarrow{\delta \to 0^+}0$, so that for
the upper bound in \eqref{eq:main12}  we have $o_\varepsilon(1)\epsilon= o_\delta(1)\epsilon$
and the conclusion  of
of   Proposition \ref{prop:contreform} is true.    \qed

\section{Proof of \eqref{eq:main3}}\label{sec:energystab1}

Up to here, we have proved \eqref{eq:main1}  and \eqref{eq:main2}.
It remains to prove   \eqref{eq:main3}.

\begin{proof}[Proof of \eqref{eq:main3}]  By the equality in the 1st line of \eqref{energy:stab1} and \eqref{eq:main2}, we have
$
\frac{d}{dt}\(E(\phi[\mathbf{z}])-A_1(\mathbf{z})\)\in L^1(\R _+).
$
Furthermore, $E(\phi[\mathbf{z}])-A_1(\mathbf{z}) \in L^\infty(\R _+), $   by
  \eqref{eq:main1}.
Thus, $\displaystyle \lim_{t\to +\infty}\(E(\phi[\mathbf{z} ])-A_1(\mathbf{z} )\)$ exists and is finite.  We have  $ A_1(\mathbf{z} )\xrightarrow{t\to +\infty}0$
by  \eqref{eq:main1}, \eqref{eq:main2} and \eqref{eq:Fermi-2}.  This implies  that  $\displaystyle \lim_{t\to +\infty} E(\phi[\mathbf{z} ]) $  exists and is finite.
Now, from \eqref{def:energy} and Proposition \ref{prop:rp}, we have
\begin{align*}
E(\phi[\mathbf{z}])=\sum_{l=1}^N \omega_l |z_l|^2 + O(\|\mathbf{z}\|^4).
\end{align*}
Thus, taking $\delta>0$ small enough, we have \begin{align}\label{energy:stab2} \frac{1}{2}|E(\phi[\mathbf{z}])|\leq \sum_{j=1}^N |\omega_j |   \ |z_j|^2 \leq 2|E(\phi[\mathbf{z}])|.\end{align}
 Now, if $\displaystyle \lim_{t\to +\infty} E(\phi[\mathbf{z}(t)])=0$, we have $|z_j(t)|\to 0$ for all $j=1,\cdots,N$ and we are done. Thus, we can assume \begin{align}\label{energy:stab3} \lim_{t\to +\infty} E(\phi[\mathbf{z}(t)])=-c^2,\ \text{with}\ c>0. \end{align}
 Notice that  we have $c\lesssim \delta$. From \eqref{energy:stab2} and \eqref{energy:stab3}, there exists $T_1>0$ s.t.\ for all $t \geq T_1$, there exists at least one $j(t)\in \{1,\cdots,N\}$ s.t. \begin{align}\label{energy:stab4} \frac{c}{\sqrt{4N|\omega_1|}}\leq |z_{j(t)}(t)|. \end{align} Next, from \eqref{def:Rmin} and \eqref{eq:main2}, there exists $M\in \N$ s.t.\ for any $j,k$ with $j\neq k$ we have $|z_jz_k|^M\in L^1(\R )$. Further, by \eqref{eq:main1}, we have $(z_jz_k)^M\in W^{1,\infty}(\R )$. Thus, we conclude \begin{align}\label{energy:stab5}   z_j(t)z_k(t)\xrightarrow{t\to +\infty}0. \end{align} In particular, there exists $T_2\geq T_1$ s.t.\ for all $t \geq T_2$ and all $j,k=1,\cdots,N$ with $j\neq k$, we have \begin{align}\label{energy:stab6} |z_j(t)z_k(t)|\leq \frac{c^2}{8N|\omega_1|}. \end{align} Combining \eqref{energy:stab4} and \eqref{energy:stab6}, for $t>T_2$ and $k\neq j(t)$, we have \begin{align*} |z_k(t)|\leq \frac{c}{2\sqrt{4N|\omega_1|}}. \end{align*} Thus, we see that $j$ satisfying \eqref{energy:stab4} is unique. Moreover by continuity, we have $j(t)=j(T_2)$ for all $t\geq T_2$. Going back to \eqref{energy:stab5}, we have \begin{align}\label{energy:stab7} \lim_{t\to +\infty} z_k(t)=0,\end{align} for all $k\neq j(T_2)$. Finally, by \eqref{energy:stab7} we have $ \(E(\phi[\mathbf{z}(t)])-E(\phi_{j(T_2)}[z_{j(T_2)}(t)]) \)\xrightarrow{t\to +\infty}0$, which implies the convergence of $E(\phi_{j(T_2)}[z_{j(T_2)}(t)])$. For small $|z_{j(T_2)}|$, the map $|z_{j(T_2)}|\mapsto E(\phi_j[z_{j(T_2)}])$ is one to one with continuous inverse.  Thus,  $\displaystyle  \lim_{t\to +\infty}|z_{j(T_2)}(t)|$ exists.
\end{proof}

\appendix
\section{ Appendix: Proof of Lemma \ref{lem:coer7} }\label{sec:comm}

It is equivalent to show that there  is a constant $C>0$ such that for all $v$
\begin{align}\label{lemma 62}
&   \left \|   \sech  \( \frac{a   x}{10}   \)  \prod _{j=1}^{N}R _{H}( \omega _j) P_c \mathcal{A} \< \im \varepsilon  \partial _x \> ^{ N} v\right  \| _{L^2(\R )}    \le C   \left \|  \sech  \( \frac{a   x}{20}   \)  v  \right \| _{L^2(\R )} .
\end{align}
By \eqref{eq:resolvKern}, for $x<y$   we have  the formula
 \begin{align}&   R _{H}(z^2) (x,y)  =   \frac{T(z)}{2\im z}           f_- (x, z)   f_+ (y, z)  = \frac{1}{z^2+\omega _j}    \frac{ f_- (x, \im \sqrt{|\omega _j|})   f_+ (y, \im \sqrt{|\omega _j|}) }{ \int _{\R}    f_-(x',\im \sqrt{|\omega _j|})  f_+(x',\im \sqrt{|\omega _j|})   dx'} +  \widetilde{R} _{H}(z^2) (x,y) ,\label{lemma 63}\end{align}
 where  $ \frac{T(z)}{2\im z}= \frac{1}{[ f_+ (x, z) ,  f_- (x, z)]}$, where in the denominator in the r.h.s. we have the Wronskian,      where $ \widetilde{R} _{H}(z^2) (x,y)$  is not singular in $z=\im \sqrt{|\omega _j|}$.
 On the other hand,
 \begin{align*}&  T(z) = \frac{\text{Res}(T,\im \sqrt{|\omega _j|}) }{z-\im \sqrt{|\omega _j|} }+  \widetilde{T}(z)     ,\end{align*}
 with $ \widetilde{T}(z)$ non singular and with residue, see p. 146 \cite{DT},
  \begin{align*}& \text{Res}(T,\im \sqrt{|\omega _j|})  =      \im \(
  \int _{\R}    f_-(x',\im \sqrt{|\omega _j|})  f_+(x',\im \sqrt{|\omega _j|})   dx' \) ^{-1} .\end{align*}
 It is elementary to conclude, comparing the terms in \eqref{lemma 63}, that
 \begin{align}&   \widetilde{R} _{H}( \omega _j) (x,y)  = K_j(x,y) +C(\omega _j) \phi _j (x) \phi _j (y)  \text{  with}\nonumber \\&  K_j(x,y)= \frac{1}{2\im \sqrt{|\omega _j|}}    \frac{ \left .\partial _{z}\( f_-(x,z)  f_+(y,z)    \)  \right | _{z=\im \sqrt{|\omega _j|}}  }
 { \int _{\R}    f_-(x',\im \sqrt{|\omega _j|})  f_+(x',\im \sqrt{|\omega _j|})   dx'}    .\label{lemma 65}\end{align}
 for some constant $ C(\omega _j)$.  For $x>y$ we obtain the same formula, interchanging $x$ and $y$.  Denoting by $K_j$ the operator with the kernel  \eqref{lemma 65}   for $x<y$ and the formula obtained from
  \eqref{lemma 65}   interchanging $x$ and $y$ if  $x>y$, we notice that
   \begin{align*}& \prod _{j=1}^{N}R _{H}( \omega _j) P_c = K_1...K_N   .\end{align*}
 It is also easy to check, following the discussion in p. 134 \cite{DT}, that there is a fixed $C>0$ s.t. $|K_j(x,y) |\le C\< x-y\>  e^{-\sqrt{|\omega _j|} |x-y|} $. Then, for any value $a\in [ 0 ,  \sqrt{|\omega_N|}]$  we have
 \begin{align*}&  \|  \sech  \( \frac{a   x}{10}   \)  \prod _{j=1}^{N}R _{H}( \omega _j) P_c \mathcal{A} \< \im \varepsilon  \partial _x \> ^{ N} v  \| _{L^2}\\& \lesssim  \|    \prod _{j=1}^{N}R _{H}( \omega _j) P_c   \sech  \( \frac{a   x}{10}   \) \mathcal{A} \< \im \varepsilon  \partial _x \> ^{ N} v  \| _{L^2}.\end{align*}
  We have
    \begin{align*}&           \sech  \( \frac{a   x}{10}   \) \mathcal{A}  =P_{N}(x, \im  \partial _x )    \sech  \( \frac{a   x}{10}   \),      \end{align*}
  for an $N$--th order differential operator with smooth and bounded coefficients.

\noindent Next, we write
 \begin{align*}&              \sech  \( \frac{a   x}{10}   \)       \< \im \varepsilon  \partial _x \> ^{ N} =  \< \im \varepsilon  \partial _x \> ^{ N}  \sech  \( \frac{a  x}{10}   \)
  + \< \im \varepsilon  \partial _x \> ^{ N} \< \im \varepsilon  \partial _x \> ^{ -N}      \left [   \sech  \( \frac{a  x}{10 }   \)  ,     \< \im \varepsilon  \partial _x \> ^{ N} \right ] ,\end{align*}
so that
\begin{align*}
&    \left \|   \sech  \( \frac{a   x}{10}   \)  \prod _{j=1}^{N}R _{H}( \omega _j) P_c \mathcal{A} \< \im \varepsilon  \partial _x \> ^{ N} v\right  \| _{L^2(\R )}  \\&   \lesssim   \left \|    \prod _{j=1}^{N}R _{H}( \omega _j) P_c  P_{N}(x, \im  \partial _x ) \< \im \varepsilon   \partial _x \> ^{ N}  \sech  \( \frac{a   x}{10}   \) v \right  \| _{L^2(\R )} \\&  +  \left \|    \prod _{j=1}^{N}R _{H}( \omega _j) P_c  P_{N}(x, \im  \partial _x ) \< \im \varepsilon   \partial _x \> ^{ N} \< \im \varepsilon   \partial _x \> ^{ -N}      \left [   \sech  \( \frac{a   x}{10}   \)  ,     \< \im \varepsilon   \partial _x \> ^{ N} \right ] v \right  \| _{L^2(\R )} \\& =:I+II
 .
\end{align*}
We have
\begin{align*}&  I \le  \left \|    \prod _{j=1}^{N}R _{H}( \omega _j) P_c  P_{N}(x, \im  \partial _x ) \< \im \varepsilon   \partial _x \> ^{ N}  \right  \| _{L^2\to L^2} \left \|    \sech  \( \frac{a   x}{10}   \) v \right  \| _{L^2(\R )} \le C  \left \|    \sech  \( \frac{a   x}{10}   \) v \right  \| _{L^2(\R )} \end{align*}
with a fixed constant $C$ independent from $\varepsilon \in (0,1)$. Next, we have
\begin{align*}&  II \le   \left \|     \< \im \varepsilon   \partial _x \> ^{ -N}      \left [   \sech  \( \frac{a   x}{10}   \)  ,     \< \im \varepsilon   \partial _x \> ^{ N} \right ] v \right  \| _{L^2(\R )} \le C  \left \|    \sech  \( \frac{a   x}{20}   \) v \right  \| _{L^2(\R )} \end{align*}
by   Lemma \ref{claim:l2boundIII}, because  $\int e^{-\im kx}  \sech (x) dx = \pi \ \sech \( \frac{\pi}{2} k \)$ (which can be proved by an elementary application of the Residue Theorem) so that in the strip $k=k_1+ \im k_2$ with  $|k_2|\le \mathbf{b}:=a/20$,  then $\sech \( \frac{\pi}{2}    \ \frac{10}{a} k \) $ satisfies the estimates required on $\widehat{\mathcal{V}}$   in
\eqref{eq2stestJ22III}.
This completes the proof of \eqref{lemma 62}.

\section*{Conflict of Interest Statement}
 The authors declare that there was no conflict of interest.

\section*{Acknowledgments}
C.  was supported by the Prin 2020 project \textit{Hamiltonian and Dispersive PDEs} N. 2020XB3EFL.
M. was supported by the JSPS KAKENHI Grant Number 19K03579, G19KK0066A and JP17H02853.

Department of Mathematics and Geosciences,  University
of Trieste, via Valerio  12/1  Trieste, 34127  Italy.
{\it E-mail Address}: {\tt scuccagna@units.it}

Department of Mathematics and Informatics,
Graduate School of Science,
Chiba University,
Chiba 263-8522, Japan.
{\it E-mail Address}: {\tt maeda@math.s.chiba-u.ac.jp}

\end{document}